\documentclass[a4paper,11pt,reqno]{amsart}

\usepackage[left=1 in, right=1 in,top=1 in, bottom=1 in]{geometry}

\usepackage{amsmath,amssymb,amsthm,amsfonts,stmaryrd}

\numberwithin{equation}{section}

\allowdisplaybreaks
\newtheorem{theorem}{Theorem}[section]
\newtheorem{lemma}[theorem]{Lemma}

\newtheorem{proposition}[theorem]{Proposition}
\theoremstyle{definition}

\newtheorem{remark}[theorem]{Remark}

\newcommand{\R}{\mathbb{R}}

\def\r3{\mathbb{R}^3}

\newcommand{\vel}{v}

\newcommand{\bv}{\langle v\rangle}

\newcommand{\FP}{\mathbf{P}}

\newcommand{\FI}{\mathbf{I}}

\newcommand{\na}{\nabla}

\newcommand{\al}{\alpha}
\newcommand{\be}{\beta}
\newcommand{\ga}{\gamma}

\newcommand{\la}{\lambda}

\newcommand{\si}{\sigma}
\newcommand{\pa}{\partial}

\newcommand{\Ga}{\Gamma}

\newcommand{\testF}{C_{c}^{\infty}(\R^3_x \times \R^3_\vel)}

\providecommand{\norm}[1]{\left\Vert#1\right\Vert}

\providecommand{\norms}[1]{\left\vert#1\right\vert}

\def\Llangle{\left \langle}
\def\Rrangle{\right \rangle}

\begin{document}                        

\title[The VPL system in $\r3$]{Global solution and time decay of the Vlasov-Poisson-Landau system in $\r3$}

\author{Yanjin Wang}
\address{School of Mathematical Sciences\\
Xiamen University\\
Xiamen, Fujian 361005, China}
\email{yanjin$\_$wang@xmu.edu.cn}

\thanks{Partially supported by National Natural Science Foundation of China-NSAF (No. 10976026)}


\begin{abstract}
We construct the global unique solution near a global Maxwellian to the Vlasov-Poisson-Landau system in the whole space. The total density of
two species of particles decays at the optimal algebraic rates as the Landau equation in the whole space, but the disparity between two species and the electric potential decay at the faster rates as the Vlasov-Poisson-Landau system in a periodic box.
\end{abstract}

\maketitle


\section{Introduction}\label{intro}
The dynamics of charged dilute particles (e.g., electrons and ions) in the absence of magnetic effects can be described by the
Vlasov-Poisson-Landau system:
\begin{equation}\label{VPL}
\begin{split}
&\partial_tF_++v\cdot\nabla_xF_+-\nabla_x\phi\cdot\nabla_vF_+=Q(F_+,F_+)+Q(F_-,F_+),
\\&\partial_tF_-+v\cdot\nabla_xF_-+\nabla_x\phi\cdot\nabla_vF_-=Q(F_+,F_-)+Q(F_-,F_-),
\\&-\Delta_x\phi=\int_{\mathbb{R}^{3}}\left(F_+-F_-\right)\,dv,
\\ &F_\pm(0,x,v)=F_{0,\pm}(x,v).
\end{split}
\end{equation}
Here $F_\pm(t,x,v)\ge 0$ are the number density functions for the ions $(+)$ and electrons $(-)$
respectively, at time $t\ge 0$, position $x=(x_1,x_2,x_3)\in \mathbb{R}^3$ and velocity $v=(v_1,v_2,v_3)\in \mathbb{R}^3$. The self-consistent
electric potential $\phi(t,x)$ is coupled with $F_\pm(t,x,v)$ through the Poisson equation. The collision between charged particles is given by the Landau (Fokker-Planck) operator:
\begin{equation}\label{collision}
Q(G_1,G_2)(v)
=\nabla _{v}\cdot \int_{\mathbb{R}^{3}}\Phi (v-v')\left(G_1(v')\nabla _{v}G_2(v)
- G_2(v)\nabla _{v'}G_1(v')\right)dv',
\end{equation}
where
\begin{equation}\label{kernel}
\Phi (v)=\frac{1}{|v|}\left( I-\frac{v\otimes v}{|v|^{2}}\right).
\end{equation}%
Since all the physical constants will not create essential mathematical difficulties along our analysis, for notational simplicity, we have normalized all constants in the Vlasov-Poisson-Landau system to be one. Accordingly, we normalize the global Maxwellian as
\begin{equation}
\mu(v)\equiv\mu_+(v)=\mu_-(v)=  {\rm e}^{-|v|^2}.
\end{equation}

We define the standard perturbation $f_{\pm }(t,x,v)$ to $\mu $ as
\begin{equation}
F_{\pm }=\mu +\sqrt{\mu }f_{\pm }.  \label{f}
\end{equation}%
Letting $f(t,x,v)=\binom{f_{+}(t,x,v)}{f_{-}(t,x,v)}$,
the Vlasov-Poisson-Landau system for the perturbation now takes the form
\begin{equation}\label{VPL_per}
\begin{split}
&\left\{\partial _{t}+v\cdot \nabla _{x}\mp \nabla_x\phi\cdot \nabla _{v}\right\}f_{\pm }
\pm 2
 \nabla_x\phi \cdot v \sqrt{\mu }
+
L_{\pm }f
=\Gamma _{\pm
}(f,f)\mp \nabla_x\phi\cdot v f_{\pm },
\\&
-\Delta_x \phi =\int \sqrt{\mu}(f_{+}-f_{-})\,dv.
\end{split}
\end{equation}
For any $g=\binom{{g_{1}}}{{g_{2}}}$, the linearized collision operator $Lg$ in \eqref{VPL_per} is given by the vector
\begin{equation}\label{L}
Lg\equiv \binom{L_{+}g\;}{L_{-}g}\equiv - \binom{2A_\ast g_{1} +K_\ast( g_{1}+g_{2})}{2A_\ast g_{2} +K_\ast( g_{1}+g_{2})}
\end{equation}
with $A_\ast, K_\ast$ defined by
\begin{equation}
A_\ast \varphi\equiv\frac{1}{\sqrt{\mu}}Q(\mu,\sqrt{\mu}\varphi),\quad  {K}_\ast \varphi\equiv\frac{1}{\sqrt{\mu}}Q(\sqrt{\mu}\varphi,\mu).
\end{equation}
For $g=\binom{{g_{1}}}{{g_{2}}}$ and $h=\binom{{h_{1}}}{{h_{2}}}$, the nonlinear collision
operator $\Gamma (g,h)$ in \eqref{VPL_per} is given by the vector
\begin{equation}\label{Gamma}
\Gamma (g,h)\equiv \binom{\Gamma _{+}(g,h)}{\Gamma _{-}(g,h)}\equiv  \binom{\Gamma_\ast(g_1+g_2,h_1)}{\Gamma_\ast(g_1+g_2,h_2)}
\end{equation}
with $\Gamma_\ast$ defined by
\begin{equation}
\Gamma_\ast(\varphi,\psi)\equiv\frac{1}{\sqrt{\mu}}Q(\sqrt{\mu}\varphi,\sqrt{\mu}\psi).
\end{equation}

\hspace{-12pt}{\bf Notation.} For notational simplicity, we use $\langle\cdot,\cdot\rangle$ to denote the $L^2$ inner product in $\mathbb{R}^3_v$, while we
use $(\cdot,\cdot)$ to denote the $L^2$ inner product in either $\mathbb{R}^3_x\times\mathbb{R}^3_v$ or $\mathbb{R}^3_x$ without any ambiguity.
Sometimes, we shall use $\int g$ to denote the integration of $g$ over $\mathbb{R}^3_x\times\mathbb{R}^3_v$ or $\mathbb{R}^3_x$.
Letting the multi-indices $\alpha $ and $\beta $ be $\alpha =[\alpha _{1},\alpha
_{2},\alpha _{3}]$, $\beta =[\beta _{1},\beta _{2},\beta _{3}]$,  we
define $\partial _{\beta }^{\alpha }\equiv \partial _{x_{1}}^{\alpha
_{1}}\partial _{x_{2}}^{\alpha _{2}}\partial _{x_{3}}^{\alpha _{3}}\partial
_{v_{1}}^{\beta _{1}}\partial _{v_{2}}^{\beta _{2}}\partial _{v_{3}}^{\beta
_{3}}.$  If each component of $\theta$ is not greater than that of $\bar{\theta}$'s, we denote by $\theta \leq \bar{\theta};$ $\theta <\bar{\theta}$
means $\theta \leq \bar{\theta},$ and $|\theta |<|\bar{\theta}|$ where $|\theta | = \theta_1 +  \theta_2 +  \theta_3$.
We use $\nabla^\ell$ with an integer $\ell\ge0$ for the  any $\pa^\al$ with $|\al|=\ell$. When
$\ell<0$ or $\ell$ is not a positive integer, $\nabla^\ell$ stands for $\Lambda^\ell$ defined by
\begin{equation}\label{Lambda s}
\Lambda^s f(x)=\int_{\mathbb{R}^3}|\xi|^s\hat{f}(\xi)e^{2\pi ix\cdot\xi}\,d\xi,
\end{equation}
where $\hat{f}$ is the  Fourier transform of $f$. We use $\dot{H}^s(\mathbb{R}^3),
s\in \mathbb{R}$ to denote the homogeneous Sobolev spaces on $\mathbb{R}^3$ with norm defined by
$\norm{f}_{\dot{H}^s}=\norm{\Lambda^s f}_{L^2}$, and we use $H^s(\mathbb{R}^3) $ to denote the usual Sobolev spaces with norm $\norm{\cdot}_{H^s}$. We shall use $\norm{\cdot}_{p}$ to denote $L^{p}$ norms
in either $\mathbb{R}^{3}_x\times \mathbb{R}^{3}_v$ or  $\R^3_x$. Letting $w(v)\geq 1$ be a weight function, we use $\norm{\cdot }_{p,w}$ to
denote the weighted $L^{p}$ norms in $\mathbb{R}^{3}_x\times \mathbb{R}^{3}_v$. We also use $\norms{\cdot}_{p}$ for the $L^{p}$ norms
in $\mathbb{R}^{3}_v$, and $\norms{\cdot}_{p,w}$ for the weighted $L^{p}$ norms in $\mathbb{R}^{3}_v$. We will use the mixed spatial-velocity spaces, e.g., $L^2_v H^s_x=L^2(\r3_v;H^s(\r3_x))$, etc.

Throughout the paper  we let $C$  denote
some positive (generally large) universal constants and $\lambda$ denote  some positive (generally small) universal constants. They {\it do not} depend on either $l$ or $m$; otherwise, we will denote them by $C_l$, $C_{l,m}$, etc.
We will use $A \lesssim B$ ($A \gtrsim B$ and $A\sim B$) if $A \le C B$.
We use $C_0$ and $\lambda_0$ to denote the constants depending on the initial data and $l,m,s$.

For the Landau operator \eqref{collision}, we define
\begin{equation}
\sigma^{ij}(v)=\Phi^{ij}\ast\mu=\int_{\r3}\Phi^{ij}(v-v')\mu(v')\,dv'.
\end{equation}
 We define the weighted norms
\begin{equation}\label{sigma norm}
\norms{f}_{\sigma ,w}=\int_{\r3}w^2\left[\sigma^{ij}\partial_if\partial_jf+\sigma^{ij}v_iv_jf^2\right]dv  \text{ and }\norm{f}_{\sigma ,w}=\norm{\norms{f}_{\sigma ,w}}_2.
\end{equation}
From Lemma 3 in \cite{G02}, we have
\begin{equation}\label{sigma norm =}
\norms{f}_{\sigma ,w}\sim \norms{\langle v\rangle^{-\frac{1}{2}} f}_{2,w}
+\norms{\langle v\rangle^{-\frac{3}{2}} \na_vf\cdot \frac{v}{|v|}}_{2,w}
+\norms{\langle v\rangle^{-\frac{1}{2}} \na_vf\times \frac{v}{|v|}}_{2,w}
\end{equation}
with $\bv=\sqrt{1+|v|^2}$. Let $\norms{f}_{\sigma}=\norms{f}_{\sigma ,1}$ and $\norm{f}_{\sigma}=\norm{f}_{\sigma ,1}$. It is well known that the linear collision operator $L\ge 0$ and is locally coercive in the sense that
\begin{equation}\label{positive L}
\langle Lf,f\rangle\gtrsim \norms{\{\FI-\FP\} f}_\sigma^2,
\end{equation}
where $\FP$ denotes the $L^2_v$ orthogonal projection on the null space of $L$:
\begin{equation}\label{null L}
N(L)\equiv
\mathrm{span}\left\{ \sqrt{\mu }\binom{1}{0},\;\sqrt{\mu }\binom{0}{1}%
,\;v\sqrt{\mu }\binom{1}{1},\;|v|^{2}\sqrt{\mu }\binom{1}{1}\right\}.
\end{equation}

As in \cite{G12}, we define the following velocity weight
\begin{equation}  \label{weight}
w(\alpha ,\beta )(v)\equiv e^{\frac{q|v|^2}{2}}
\langle v\rangle^{2(l-|\alpha |-|\beta |)},\text{ \ \ \ }l\geq |\alpha |+|\beta |,\quad 0\le q\ll 1.
\end{equation}
Letting $l\ge m$, we define the energy functional by
\begin{equation} \label{energy}
\mathcal{E}_{m;l,q}(f)
= \sum_{|\alpha |+|\beta |\leq m}
\norm{ \partial _{\beta }^{\alpha }f}_{2,w(\alpha ,\beta )}^{2}+
\norm{\nabla_x\phi}_2^2,
\end{equation}
and the corresponding dissipation rate by
\begin{equation}\label{dissipation}
\mathcal{D}_{m;l,q}(f)
\equiv \sum_{|\alpha |+|\beta |\leq m}\norm{ \partial_{\beta }^{\alpha }\left\{{\bf I-P}\right\}f}_{\sigma,w(\alpha ,\beta )}^{2}
+\sum_{1\le |\alpha|\le m}
\norm{\partial^\alpha {\bf P }f }_2^2+\norm{\nabla_x\phi
}^2_2+ \norm{(f_+-f_-)}_2^2.
\end{equation}
We also define
\begin{equation}\label{dissipation others}
\widetilde{\mathcal{D}}_{m;l,q}(f)\equiv \mathcal{D}_{m;l,q}(f)-\norm{(f_+-f_-)}_2^2\text{ and }\overline{\mathcal{D}}_{m;l,q}(f)\equiv \widetilde{\mathcal{D}}_{m;l,q}(f)+\norm{ {\bf P }f }_2^2.
\end{equation}
We remark that $\widetilde{\mathcal{D}}_{m;l,q}(f)$ and $\overline{\mathcal{D}}_{m;l,q}(f)$ are the dissipations used in \cite{SZ12} and \cite{G12} respectively, and we introduce them for the presentational convenience. Note that there is a cascade of velocity weights in \eqref{energy} and \eqref{dissipation} so that fewer derivatives of $f$ demand stronger velocity weights. Our first main result of the global unique solution  to the system \eqref{VPL_per} is stated as follows.
\begin{theorem}\label{main theorem}
Assume that $f_{0}$ satisfies $F_{0,\pm }(x,v)=\mu +\sqrt{\mu }f_{0,\pm
}(x,v)\geq 0.$ There exists a sufficiently small $M>0$ such that if
$\mathcal{ {E}}_{2;2,0}({f}_{0})\leq M$, then there exists a unique global solution $f(t,x,v)$ to the
Vlasov-Poisson-Landau system \eqref{VPL_per} with $F_{\pm
}(t,x,v)=\mu +\sqrt{\mu }f_{\pm }(t,x,v)\geq 0.$

(1) If $\mathcal{ {E}}_{2;l,q}({f}_{0})<+\infty$ for $l\ge 2,\ 0\le q\ll 1$, then there exists $C_l>0$ such that
\begin{equation}\label{energy inequality}
\sup_{0\le t\le \infty} \mathcal{ {E}}_{2;l,q}(f(t))+\int_0^\infty \mathcal{ {D}}_{2;l,q}(f(\tau))\,d\tau\le C_l \mathcal{ {E}}_{2;l,q}(f_0).
\end{equation}
Furthermore,
\begin{equation}\label{polynomial decay}
\begin{split}
&\norm{\partial _{t}\phi (t)}_{\infty }+\norm{\nabla _{x}\phi (t)}_{\infty}+\norm{\nabla _{x}\phi (t)}_{2}+\sum_{k=0,1}\norm{\na^k(f_+-f_-)(t)}_2
\\&\quad \leq
 C_{l}(1+t)^{-2(l-1)} \sqrt{\mathcal{ {E}}_{2;l,0}(f_0)}
 \end{split}
\end{equation}
and
\begin{equation}\label{exponential decay}
\begin{split}
&\norm{\partial _{t}\phi (t)}_{\infty }+\norm{\nabla _{x}\phi (t)}_{\infty}+\norm{\nabla _{x}\phi (t)}_{2}+\sum_{k=0,1}\norm{\na^k(f_+-f_-)(t)}_2
\\&\quad \le
 C_{l}e^{-C_lt^{2/3}} \sqrt{\mathcal{ {E}}_{2;l,q}(f_0)}\text{ for }0<q\ll 1.
  \end{split}
\end{equation}

(2) In addition, if $\mathcal{ {E}}_{m;l,q}(f_{0})<\infty $ for any $%
l>2$, $l\geq m \geq 2$, $0\le q\ll 1$, there exists an increasing continuous function
$P_{m,l}(\cdot )$ with $P_{m,l}(0 )=0$ such that the unique solution satisfies
\begin{equation}\label{energy inequality m}
\sup_{0\leq t\leq \infty }\mathcal{ {E}}_{m;l,q}(f(t))+\int_{0}^{\infty }%
\mathcal{D}_{m;l,q}(f(\tau))\,d\tau\leq P_{m,l}(\mathcal{ {E}}_{m;l,q}(f_{0}) ).
\end{equation}
\end{theorem}

Theorem \ref{main theorem} will be proved in sections \ref{energy section} and \ref{global} by using the strategy of \cite{G12} in which Guo proved the first result of the global unique solution near Maxwellians to the Vlasov-Poisson-Landau system in a periodic box. There are two folds in this strategy. First, by the introduction of the exponential weight $e^{\pm(q+1)\phi}$ to cancel the growth of the velocity in the nonlinear term $\mp\na_x\phi\cdot v f_\pm$ and the introduction of the velocity weight \eqref{weight} to capture the weak velocity diffusion in the Landau kernel, assuming that the conservation of mass, momentum as well as energy holds, Guo \cite{G12} first derived the following energy inequality (for $m=l=2,q=0$):
\begin{equation}\label{G12 01}
\mathcal{E}_{2;2,0}(f)+\int_0^t \overline{\mathcal{D}}_{2;2,0}(f)\,ds
\lesssim \mathcal{E}_{2;2,0}(f_0)+\int_0^t\left[\norm{\pa_t\phi }_\infty+\norm{\na_x\phi}_\infty\right]\mathcal{E}_{2;2,0}(f)\,ds.
\end{equation}
The reason why the term $\norm{{\bf P}f}_2^2$ is excluded from our dissipation rate is that the Poincar\'e inequality fails in the whole space. However, this requires us the much more careful arguments in section \ref{energy section} when deriving the energy inequality of type \eqref{G12 01} for the whole space with  $\overline{\mathcal{D}}_{m;l,q}(f)$ replaced by $\widetilde{\mathcal{D}}_{m;l,q}(f)$; we need to prove an improved refined estimate for the nonlinear collision term as Lemma \ref{nonlinear c} in section \ref{pre}. Second, thanks again to the conservation laws and the Poincar\'e inequality in the periodic box, Guo \cite{G12} then derived the following differential inequality:
\begin{equation}\label{G12 02}
\frac{d}{dt}Y+ \sum_{k=0,1} \norm{\na^k f}_\sigma^2\le 0,
\end{equation}
where $Y(t)\sim \sum_{k=0,1} \norm{\na^k f}_2^2$. By applying the method previously developed in Strain and Guo \cite{SG06,SG08}, then a decay rate of the electric potential is extracted from  \eqref{G12 02}:
\begin{equation} \label{G12 03}
\norm{\partial _{t}\phi (t)}_{\infty }+\norm{\nabla _{x}\phi (t)}_{\infty}\lesssim \sqrt{Y(t) } \lesssim  (1+t)^{-2}.
\end{equation}
Hence the energy estimates \eqref{G12 01} is closed by applying the standard Gronwall lemma. However, as remarked in \cite{G12}, the strong decay rate of \eqref{G12 03} (and \eqref{G12 02}) is a consequence of the periodic box (and the conservation laws!), and it remains open if a sufficient decay rate can be obtained for the whole space. To get the sufficient decay rate of $\phi$ in the whole space case, in addition to \eqref{G12 01}, Strain and Zhu \cite{SZ12} further developed another energy inequality:
\begin{equation}\label{SZ12 01}
\frac{d}{dt}\widetilde{\mathcal{E}}^h(f)+\widetilde{\mathcal{D}}^h(f)\lesssim [\norm{\pa_t\phi }_\infty+\norm{\na_x\phi}_\infty]\widetilde{\mathcal{E}}^h(f)+\norm{\na_x\FP f}_2^2,
\end{equation}
where $\widetilde{\mathcal{E}}^h(f)$ denotes some high-order energy functional that does not contain $\norm{\FP f}_2^2$ and $\widetilde{\mathcal{D}}^h(f)$ is the corresponding dissipation. Assuming additionally that the $L^2_vL^1_x$ norm of the initial data is small, by combining these energy estimates and the linear decay analysis, Strain and Zhu \cite{SZ12} obtained a decay rate of
\begin{equation} \label{SZ12 02}
\norm{\partial _{t}\phi (t)}_{\infty }+\norm{\nabla _{x}\phi (t)}_{\infty}
 \lesssim  (1+t)^{-\frac{5}{4}+\varepsilon}\left(\sqrt{\mathcal{E}_{3;l;0}(f_0)}+\norm{f_0}_{L^2_vL^1_x}\right),
\end{equation}
where $\varepsilon=(l-\frac{5}{2})^{-1}\frac{5}{8}\in \left(0, \frac{1}{4}\right)$ if $l>5$.

The $L^2_vL^1_x$ assumption of the initial data seems crucial for the global existence of the solutions to the Vlasov-Poisson-Landau system in the whole space in Strain and Zhu \cite{SZ12}; see also Duan, Yang and Zhao \cite{DYZ11} for the one-species Vlasov-Poisson-Landau system in the whole space. However, as well illustrated in Theorem \ref{main theorem}, we have removed such kind of assumption and our Theorem \ref{main theorem} for the whole space is almost like Theorem 2 for the periodic box case in Guo \cite{G12}. The key motivation is that the real thing we need to close the estimates \eqref{G12 01} is a strong decay rate of $\phi$ rather than the whole solution! Let us look back at the system \eqref{VPL_per}, and we note that $\phi$ depends only on  $f_+-f_-$ and also that there are some cancelations between the ``$+$" and ``$-$" equations. We are then led to consider the sum and difference of $f_+$ and $f_-$:
\begin{equation}
f_1=f_++f_-\text{ and }f_2=f_+-f_-.
\end{equation}
The Vlasov-Poisson-Landau system \eqref{VPL_per} can be equivalently rewritten as
\begin{equation}\label{f_1f_2 equation}
\begin{split}
&\partial_tf_1+ v\cdot\nabla_xf_1+ \mathcal{L}_1f_1=\Gamma_\ast(f_1,f_1)+\nabla_x\phi\cdot \left(\nabla_vf_2-vf_2\right),
 \\ &\partial_tf_2 + v\cdot\nabla_xf_2+4\nabla_x\phi\cdot v\sqrt{\mu} + \mathcal{L}_2 f_2=\Gamma_\ast(f_1,f_2)+\nabla_x\phi\cdot \left(\nabla_vf_1-vf_1\right),
\\ &-\Delta_x\phi=\int_{\r3}f_2\sqrt{\mu}\,dv.
\end{split}
\end{equation}
Here $\mathcal{L}_1=-2(A_\ast+K_\ast)$ is the one-species linearized Landau operator and $\mathcal{L}_2=-2A_\ast$.
Notice that $[\mathcal{L}_1f_1,\mathcal{L}_2f_2]$ is equivalent to $Lf$, and their null spaces are
\begin{equation}
N(\mathcal{L}_1)\equiv
\mathrm{span}\left\{ \sqrt{\mu } ,\; v\sqrt{\mu } ,\;|v|^{2}\sqrt{\mu } \right\}\text{ and }N(\mathcal{L}_2)\equiv
\mathrm{span}\left\{ \sqrt{\mu }\right\}.
\end{equation}
Let ${\bf P_i}$ be the $L^2_v$ orthogonal projection on the null space of $\mathcal{L}_i$ respectively, then
\begin{equation}\label{positive L1L2}
\langle \mathcal{L}_if,f\rangle\ge \delta_0 \norms{\{{\bf I-P_i}\} f}_\sigma^2,\  i=1,2.
\end{equation}
Notice that the linear homogeneous system of \eqref{f_1f_2 equation} is decoupled into two independent subsystems: one is the Landau equation for $f_1$; the other one is a system almost like the one-species Vlasov-Poisson-Landau system for $f_2$ and $\phi$ but with the linear collision operator $\mathcal{L}_2$. Then our key observation is that we can include the full $\norm{f_2}_\sigma^2$ in our dissipation rate \eqref{dissipation}: $\mathcal{L}_2$ controls the microscopic part $\{{\bf I-P_2}\}f_2$ by \eqref{positive L1L2}; while the presence of the electric field can control the hydrodynamic part ${\bf P_2}f_2$! This special coupling effect between $\mathcal{L}_2$ and the Poisson equation and the decoupling make us be able to derive a differential inequality similar as \eqref{G12 02}:
\begin{equation}\label{key inequality}
\frac{d}{dt}\mathcal{E}_0^1(f_2)+\sum_{k=0,1} \norm{\nabla^k f_2}_\sigma^2+\norm{\nabla_x\phi}_2^2 \le 0,
\end{equation}
where $\mathcal{E}_0^1(f_2)(t)\sim \sum_{k=0,1} \norm{\nabla^k f_2}_2^2+\norm{\nabla_x\phi}_2^2 $. Hence, in the whole space we can obtain the same decay rate of $\phi$ as the periodic case in Guo \cite{G12}. It is also interesting to point out that our observation also works for the periodic case, and hence we can still prove the global solution even without the assumption of the conservation laws which is crucial in Guo \cite{G12} so that the Poincar\'e inequality can be employed.

Our second main result is on some further decay rates of the solution to the Vlasov-Poisson-Landau system \eqref{VPL_per} by making the much stronger assumption on the initial data. We remark that our main purpose is to clarify how to derive these decay rates for the solution and its higher-order spatial derivatives, so we will not pursue the optimal spatial regularity and velocity moments and smallness assumptions on the initial data.
\begin{theorem}\label{further decay}
Assume that $f$ is the solution to the
Vlasov-Poisson-Landau system \eqref{VPL_per} constructed in Theorem \ref{main theorem}. Fix $%
l\ge  m\geq 2$, then there exists a sufficiently small $M=M(m,l)$ such that if $\mathcal{ {E}}_{m;l,0}(f_{0})\le M$, then

(1) If $l\ge m+\frac{3}{4}$, then for $\ell=1,\dots,m$,
\begin{equation}\label{polynomial decay k}
\sum_{k=0}^\ell \norm{\na^k(f_+-f_-)(t)}_2
\leq
 C_{l,\ell}(1+t)^{-2(l-\ell)} \sqrt{\mathcal{ {E}}_{m;l,0}(f_0)};
\end{equation}
and if in addition $\mathcal{ {E}}_{m;l,q}(f_{0})<+\infty $ for $0<q\ll 1$, then
\begin{equation}\label{exponential decay k}
\sum_{k=0}^m \norm{\na^k(f_+-f_-)(t)}_2
 \le
 C_{l,m}e^{-C_{l,m}\,t^{2/3}}\sqrt{ \mathcal{ {E}}_{m;l,q}(f_0)}.
\end{equation}

(2) If $ \norm{\Lambda^{-s}(f_{0,+}+f_{0,-}) }_2<+\infty$ for some $s\in [0,3/2)$, then
\begin{equation}\label{H-sbound}
\norm{\Lambda^{-s}(f_++f_-)(t)}_2\le C_0.
\end{equation}
Moreover, if $l\ge \max\{m+\frac{3}{4},\frac{5}{4}m+\frac{s-1}{4}\}$, then for $\ell=0,\dots,m-1$,
\begin{equation}\label{decay f1}
 \sum_{\ell\le k \le m}\norm{\nabla^k (f_++f_-)(t)}_2\le C_0(1+t)^{-\frac{\ell+s}{2}};
 \end{equation}
for any fixed small $\varepsilon>0$, if $l\ge \frac{m-1+s}{4\varepsilon} +m-\frac{9}{4}$, then for $\ell=0,\dots,m-2$,
\begin{equation}\label{decay f1 micro}
\norm{\nabla^\ell \{{\bf I-P_1 }\}(f_++f_-)(t)}_2\le C_0(1+t)^{-\frac{\ell+s+1-\varepsilon}{2}}.
 \end{equation}
\end{theorem}

 The followings are several remarks for our main theorems.

\begin{remark}{\it
We can consider the generalized Landau operator with, see \cite{H83,DL97,G02},
\begin{equation}\label{kernels}
\Phi (v)=\frac{1}{|v|^{\gamma+2}}\left( I-\frac{v\otimes v}{|v|^{2}}\right),\quad \gamma\ge -3.
\end{equation}
It is easy to conclude from our proof that the global unique solution $f$ to the Vlasov-Poisson-Landau system near Maxwellians exists for all $\gamma\ge -3$.
If $\gamma\ge -2$, then the decay rates of $f_+-f_-$ and $\phi$ in our theorems can be improved to be an exponential rate; and we can take $\varepsilon=0$ in \eqref{decay f1 micro} so that the decay rates of $ \{{\bf I-P_1 }\}(f_++f_-)$ are optimal. Very recently, Duan, Yang and Zhao \cite{DYZ121, DYZ122} established a global existence theory for the one-species Vlasov-Poisson-Boltzmann system in the whole space with hard potentials and soft potentials respectively, which generalized the pioneering work of Guo \cite{G022} with hard-sphere interaction. We believe that if the Vlasov-Poisson-Boltzmann system is of two-species, then our observation can be used to remove the $L^2_vL^1_x$ assumption of the initial data in \cite{DYZ121, DYZ122} and get the faster decay of electric potential.}
 \end{remark}

\begin{remark}{\it
The constraint $s<3/2$ in Theorem \ref{further decay} comes
from applying Lemma \ref{Riesz lemma} to estimate the nonlinear terms when doing
the negative Sobolev estimates via $\Lambda ^{-s}$. For $s\geq 3/2$, the
nonlinear estimates would not work.  Since Lemma \ref{Riesz lemma} implies that for $p\in (1,2]$, $L^p\subset \Dot{H}^{-s}$ with $s=3(\frac{1}{p}-\frac{1}{2})\in[0,3/2)$, so as a byproduct, we obtain the usual $L^p$--$L^2$ ($1< p\le 2$) type of the optimal decay rates for $f_++f_-$. Note also that the $L^2$ optimal decay rate of the higher-order spatial derivatives of the solution are obtained. Then the general optimal $L^q$ ($2\le q\le \infty$) decay rates of the solution follow by applying the optimal Sobolev interpolation. We also remark that we do not require the $\Dot{H}^{-s}$ or $L^p$ norm of initial data be small.  It is worth to pointing out that the optimal decay rates of $f_++f_-$ are new even for the Landau equation in the whole space with soft potentials \cite{G02,HY07}. We also believe that our method can be applied to show the optimal decay rates for the Boltzmann equation in the whole space with soft potentials both with and without angular cut-off \cite{C80,UA82,G03,HY07,SG06,SG08,GS11,S10,AMUXY11}.}
\end{remark}

Theorem \ref{further decay} will be proved in section \ref{decays}. To prove \eqref{polynomial decay k}--\eqref{exponential decay k}, we will establish a family of general versions of the differential inequality \eqref{key inequality}:
 \begin{equation}
  \frac{d}{dt}\mathcal{E}_0^\ell(f_2)+ \sum_{k=0}^\ell  \norm{\nabla^k f_2 }_\sigma^2+ \norm{ \nabla_x\phi}_2^2
\le 0, \text{ for }\ell=1,\dots,m,
\end{equation}
where $\mathcal{E}_0^\ell(f_2)(t)\sim \sum_{0\le k\le \ell} \norm{\nabla^k f_2}_2^2+\norm{\nabla_x\phi}_2^2 $. Then \eqref{polynomial decay k}--\eqref{exponential decay k} follow by applying the interpolation method (among velocity moments) developed in Strain and Guo \cite{SG06} and the splitting method (velocity-time) developed in Strain and Guo \cite{SG08}. To prove \eqref{decay f1}, the key is to establish the following a family of energy estimates with \textit{minimum} spatial derivative counts on $f_1$:
\begin{equation}
 \begin{split}
& \frac{d}{dt}\left\{\mathcal{E}_\ell^m ({f_1})+\mathcal{E}_0^m(f_2)\right\}
+\lambda  \left(\sum_{k=\ell+1}^m  \norm{\nabla^k f_1 }_\sigma^2+ \norm{\nabla^\ell \{{\bf I-P_1}\}f_1 }_\sigma^2 \right.\\&\qquad\qquad\qquad\qquad\qquad\qquad\ \left.+\sum_{k=0}^m  \norm{\nabla^k f_2 }_\sigma^2+ \norm{\nabla_x\phi}_2^2\right)\le 0, \text{ for }\ell=0,\dots,m-1,
 \end{split}
\end{equation}
where $\mathcal{E}_\ell^m(f_1)(t)\sim \sum_{\ell\le k\le m} \norm{\nabla^k f_1}_2^2$.  To achieve this, we will
extensively and carefully use the Sobolev interpolation of the
Gagliardo-Nirenberg inequality between high-order and low-order spatial
derivatives to control the nonlinear estimates. After deriving the negative Sobolev estimates \eqref{H-sbound}, \eqref{decay f1}--\eqref{decay f1 micro} follows by combing a Sobolev interpolation method among spatial regularity with the methods in \cite{SG06,SG08,S10}.

The rest of our paper is organized as follows. In section \ref{pre}, we establish an improved refined estimates for the nonlinear collision terms and collect some analytic tools. In section \ref{energy section}, we establish the nonlinear energy estimates for the local solutions. In section \ref{global}, we derive the basic time decay estimates and show that the local solution is global. In section \ref{decays}, we obtain the further decay rates for the global solution.

\section{Preliminary}\label{pre}

In this section, we  use $\widetilde{L}$ to uniformly denote the linear collision operators $L$, $\mathcal{L}_1$ and $\mathcal{L}_2$, and we use $\widetilde{\Gamma}$ to denote the nonlinear collision operators $\Gamma$ and $\Gamma_\ast$. We first recall the basic property of the linear collision operator $\widetilde{L}$.
\begin{lemma}\label{linear c}
We have $\Llangle \widetilde{L} g,h\Rrangle=\Llangle g,\widetilde{L}h\Rrangle ,\ \Llangle \widetilde{L} g,g\Rrangle\ge 0$, and $\widetilde{L} g=0 $ if and only if $g= \widetilde{P} g$, where $\widetilde{P}$ is the $L^2_v$ orthogonal projection onto the null space of $\widetilde{L}$, correspondingly. Moreover,
\begin{equation} \label{positive L all}
\Llangle \widetilde{L}g,g\Rrangle\gtrsim \norms{\{I- \widetilde{P}\}g}_\sigma^2.
\end{equation}
\end{lemma}
\begin{proof}
We only need to consider the case $\widetilde{L}=\mathcal{L}_2$. We use the following reexpression for $\mathcal{L}_2$ as Lemma 1 in \cite{G02}:
\begin{equation}
\mathcal{L}_2 g=-2\mu^{-1/2}\nabla_v\cdot\left\{\mu^{1/2}\sigma\left[\nabla_v g+ v g\right]\right\}
=-2\mu^{-1/2}\nabla_v\cdot\left\{\mu \sigma\nabla_v\left[\mu^{-1/2}g\right]\right\}.
\end{equation}
Then we integrate by parts to obtain
\begin{equation}
\Llangle \mathcal{L}_2 g,h \Rrangle=
\Llangle-\mu^{-1/2}\nabla_v\cdot\left\{\mu \sigma\nabla_v\left[\mu^{-1/2}g\right]\right\}, h \Rrangle
=\Llangle \mu \sigma\nabla_v\left[\mu^{-1/2}g\right] , \nabla_v\left[\mu^{-1/2}h \right]\Rrangle.
\end{equation}
Hence, $\Llangle \mathcal{L}_2 g,h\Rrangle=\Llangle
g,\mathcal{L}_2h\Rrangle ,\ \Llangle \mathcal{L}_2 g,g\Rrangle\ge 0$, and $\mathcal{L}_2 g=0 $ if and only if $g={\bf P_2} g$ since $\sigma$ is positively definite. The estimate \eqref{positive L all} for $\mathcal{L}_2$ follows as Lemma 5 in \cite{G02}.
\end{proof}

Next we recall the weighted estimates for $\widetilde{L}$.
\begin{lemma}\label{linear c weight}
Let $w=w(\al,\beta)$ in \eqref{weight}. For any small $\eta>0$, there exists $C_\eta>0$ such that
\begin{equation}
\begin{split}
(1-q^2-\eta)\norms{\partial^\alpha g}_{\sigma,w(\alpha,0)}^2
-C_{l,m,\eta} \norms{\partial^\alpha g}_{\sigma}^2
&\le\Llangle w^2(\alpha,0)\partial^\alpha  \widetilde{L}
g ,\partial^\alpha g\Rrangle
\\&  \le \norms{\partial^\alpha g}_{\sigma,w(\alpha,0)}^2
+C_{l,m,\eta} \norms{\partial^\alpha g}_{\sigma}^2,
\end{split}
 \end{equation}
and for $\beta\neq 0$
\begin{equation}
\begin{split}
 &\norms{\partial_\beta^\alpha g}_{\sigma,w(\alpha,\beta)}^2-\eta\sum_{|\beta_1|=|\beta|}\norms{\partial_{\beta_1}^\alpha g}_{\sigma,w(\alpha,\beta)}^2
-C_{l,m,\eta}\sum_{|\beta_1|<|\beta|}\norms{\partial_{\beta_1}^\alpha g}_{\sigma,w(\alpha,\beta_1)}^2
\\&\quad\le \Llangle w^2(\alpha,\beta)\partial_\beta^\alpha  \widetilde{L}
g ,\partial_\beta^\alpha g\Rrangle
\\&\quad\le \norms{\partial_\beta^\alpha g}_{\sigma,w(\alpha,\beta)}^2+\eta\sum_{|\beta_1|=|\beta|}\norms{\partial_{\beta_1}^\alpha g}_{\sigma,w(\alpha,\beta)}^2
+C_{l,m,\eta}\sum_{|\beta_1|<|\beta|}\norms{\partial_{\beta_1}^\alpha g}_{\sigma,w(\alpha,\beta_1)}^2.
\end{split}
 \end{equation}
\end{lemma}
\begin{proof}
We refer to Lemma 8 and Lemma 9 in \cite{SG08} for the proof. Indeed, the upper bounds were not written down there, but the proof is the same.
\end{proof}

We now prove the refined estimates for the nonlinear collision operator $\widetilde{\Gamma}$.
\begin{lemma}\label{nonlinear c}
Let $w=w(\al,\beta)$ in \eqref{weight}. Then we have
\begin{equation}\label{ga es}
\begin{split}
&\Llangle w^2\partial _\beta^\alpha \widetilde{\Gamma} [g_1,g_2],\;\partial
_\beta ^\alpha g_3\Rrangle
\\&\quad\le
\sum_{\alpha_1 \le \alpha\atop \bar{\beta}\le\beta_1 \le \beta}C_{\al}^{\al_1}C_{\beta}^{\beta_1}C_{(\beta_1,\bar{\beta})}
\norms{\mu^\delta\partial _{\bar{\beta}}^{\alpha _1}g_1}_{2}
\norms{\partial _{\beta -\beta _1}^{\alpha -\alpha _1}g_2}_{\sigma ,w}
\left(\norms{\partial _\beta^\alpha g_3}_{\sigma ,w}+
l
\norms{\partial _\beta^\alpha g_3}_{2 ,\frac{w}{\langle v\rangle^{\frac {3}{ 2}}}}\right).
\end{split}
\end{equation}
Hereafter $\delta>0$ is a sufficiently small universal number and $C_{(\beta_1,\bar{\beta})}$ denotes  constants depending on $\beta_1,\bar{\beta}$ but universal when $\bar{\beta}=\beta_1$. In particular, we have
\begin{equation}\label{ga es 0}
\Llangle\widetilde{\Gamma} [g_1,g_2],g_3\Rrangle
\lesssim \norms{\mu^\delta g_1}_{2}\norms{g_2}_{\sigma}\norms{g_3}_{\sigma}.
\end{equation}
\end{lemma}
Importantly, when $\al=\al_1,\beta=\bar{\beta}$ or $\al_1=\beta_1= 0$ the constant in front of the first term of \eqref{ga es} does not depend on either $l$ or $m$; otherwise, it depends on $m$! Note that there is a large factor $l$ in the second term of \eqref{ga es}.
Our key improvement is that we are free to bound the term $\norms{\mu^\delta f}_{2}$ by either $\norms{f}_2$  or $\norms{f}_\sigma$  and  the term $\norms{f}_{2 ,\frac{w}{\langle v\rangle^{\frac {3}{ 2}}}}$ by either $\norms{f}_{2,w}$  or $\norms{f}_{\sigma,w}$ as we want.
\begin{proof}
We prove the lemma for $\widetilde{\Gamma}=\Gamma_\ast$. As Proposition 7 in \cite{G12}, by the product rule we expand
\[
\langle w^2\partial _\beta ^\alpha \Gamma_\ast[g_{1}, g_2],\partial
_\beta ^\alpha g_3\rangle =\sum C_\alpha ^{\alpha _1}C_\beta ^{\beta
_1}\times G_{\alpha _1\beta _1},
\]
where $G_{\alpha _1\beta _1}$ takes the form
\begin{eqnarray}
&&-\Llangle w^2\left\{\Phi ^{ij}*\partial _{\beta _1}\left[\mu ^{1/2}\partial
^{\alpha _1}g_1\right]\right\}\partial _j\partial _{\beta -\beta _1}^{\alpha -\alpha
_1}g_2,\partial _i\partial _\beta ^\alpha g_3\Rrangle  \label{gamma1} \\
&&-(1+2q)\Llangle w^2\left\{\Phi ^{ij}*\partial _{\beta _1}\left[v_i\mu
^{1/2}\partial ^{\alpha _1}g_1\right]\right\}\partial _j\partial _{\beta -\beta
_1}^{\alpha -\alpha _1}g_2,\partial _\beta ^\alpha g_3\Rrangle  \label{gamma2}
\\&&+(1+2q)\Llangle w^2\left\{\Phi ^{ij}*\partial _{\beta _1}\left[v_i\mu
^{1/2}\partial _j\partial ^{\alpha _1}g_1\right]\right\}\partial _{\beta -\beta
_1}^{\alpha -\alpha _1}g_2,\partial _\beta ^\alpha g_3\Rrangle  \label{gamma3}\\
&&+\Llangle w^2\left\{\Phi ^{ij}*\partial _{\beta _1}\left[\mu
^{1/2}\partial _j\partial ^{\alpha _1}g_1\right]\right\}\partial _{\beta -\beta
_1}^{\alpha -\alpha _1}g_2,\partial _i\partial _\beta ^\alpha g_3\Rrangle
\label{gamma4}
\\
&&- 4(l-|\al|-|\beta|)\Llangle   \frac{w^2}{\langle v\rangle^{2}}\left\{\Phi ^{ij}*\partial _{\beta _1}\left[v_i\mu
^{1/2}\partial ^{\alpha _1}g_1\right]\right\}\partial _j\partial _{\beta -\beta
_1}^{\alpha -\alpha _1}g_2,\partial _\beta ^\alpha g_3\Rrangle  \label{part1}
\\
&&+4(l-|\al|-|\beta|) \Llangle   \frac{w^2}{\langle v\rangle^{2}}\left\{\Phi ^{ij}*\partial _{\beta _1}\left[v_i\mu
^{1/2}\partial _j\partial ^{\alpha _1}g_1\right]\right\}\partial _{\beta -\beta
_1}^{\alpha -\alpha _1}g_2,\partial _\beta ^\alpha g_3\Rrangle
\label{part2}
\end{eqnarray}
with double summations over $1\le i,j\le 3$ and $\pa_i=\pa_{v_i}$.

Note that there is a (large) factor $4(l-|\al|-|\beta|)$ in both \eqref{part1} and \eqref{part2}. The first two terms \eqref{gamma1}--\eqref{gamma2} would not encounter this factor, so they can be bounded in the same way as in Proposition 2.2 of \cite{SZ12} via the first term on the right-hand side of \eqref{ga es}. We shall now first estimate the term \eqref{part1}. Since $\Phi ^{ij}(v)=O(|v|^{-1})\in L_{loc}^2({\R}^3)$ and
$|\partial _{\beta _1-\bar{\beta}}\{v_i\mu ^{1/2}\}| \le C_{(\beta_1,\bar{\beta})}
\mu ^{1/4}$, by Lemma 2 in \cite{G02}, the Cauchy-Schwartz inequality implies
\begin{equation}\label{convolution}
\begin{split}
 \Phi ^{ij}*\partial _{\beta _1}\left[v_i\mu
^{1/2}\partial ^{\alpha _1}g_1\right]
&\le C_{ \beta_1}^{\bar{\beta}}
 \sum_{\bar{\beta}\le \beta _1}\left|\Phi ^{ij}*\left\{\partial _{\beta _1-\bar{\beta}%
}\left[v_i\mu ^{1/2}\right]\partial _{\bar{\beta}}^{\alpha _1}g_1\right\}\right|
\\
 &\le C_{(\beta_1,\bar{\beta})}
\left\{|\phi ^{ij}|^2*\mu ^{1/4}\right\}^{1/2}(v)\left\{ \sum_{\bar{\beta}\le
\beta _1}\int \mu ^{1/4}(v_*)
\left|\partial _{\bar{\beta}}^{\alpha _1}g_1(v_*)\right|^{2} dv_*\right\} ^{1/2}    \\
 &\le C_{(\beta_1,\bar{\beta})}
\langle v\rangle^{-1}\sum_{\bar{\beta}\le \beta _1}\left|\mu^\delta
\partial _{\bar{\beta}}^{\alpha _1}g_1\right| _{2}
  \end{split}
\end{equation}
for a sufficiently small $\delta>0$. Hence, recalling \eqref{sigma norm =}, we bound \eqref{part1} by
\begin{equation}\label{bound0}
\begin{split}
 & l C_{(\beta_1,\bar{\beta})}\sum_{\bar{\beta}\le \beta _1}
\left| \mu^\delta\partial _{\bar{\beta}}^{\alpha _1}g_1\right|_{2}
\int \norms{w^2\langle v\rangle^{-3} \partial
_j\partial _{\beta -\beta _1}^{\alpha -\alpha _1}g_2\partial _\beta ^\alpha
g_3}dv \\
&\quad\le  l C_{(\beta_1,\bar{\beta})}
\sum_{\bar{\beta}\le \beta _1}
\left|\mu^\delta\partial _{\bar{\beta}}^{\alpha _1}g_1\right|_{2}
\left| w \langle v\rangle^{-\frac {3}{ 2}}\partial_j\partial _{\beta -\beta _1}^{\alpha -\alpha _1}g_2\right| _2\left|
w \langle v\rangle^{-\frac {3}{ 2}}\partial _\beta ^\alpha g_3\right| _2 \\
&\quad\le  l C_{(\beta_1,\bar{\beta})}
\sum_{\bar{\beta}\le \beta _1}\norms{\mu^\delta\partial _{\bar{\beta}}^{\alpha
_1}g_1}_{2 }
\norms{\partial _{\beta -\beta _1}^{\alpha -\alpha_1}g_2}_{\sigma ,w}
\norms{\partial _\beta ^\alpha g_3}_{2 ,\frac{w}{\langle v\rangle^{\frac {3}{ 2}}}}.
  \end{split}
\end{equation}

For the rest three terms, we  use an integration by parts inside the convolution (in $v_\ast$) to split
\begin{equation}\label{parts}
\Phi ^{ij}*\partial _{\beta _1}\left[\mu ^{1/2}\partial _j\partial ^{\alpha
_1}g_1\right]=\partial _j\Phi ^{ij}*\partial _{\beta _1}\left[\mu ^{1/2}\partial
^{\alpha _1}g_1\right]-\Phi ^{ij}*\partial _{\beta _1}\left[\partial _j\mu
^{1/2}\;\partial ^{\alpha _1}g_1\right].
\end{equation}
We first estimate the most singular term \eqref{gamma4}. The second part of \eqref{gamma4} corresponding to the split \eqref{parts} has the same upper bound as \eqref{gamma1}--\eqref{gamma2}, we then focus on the first part of \eqref{gamma4}:
\begin{equation}\label{haha}
\Llangle w^2\left\{\partial _j\Phi ^{ij}*\partial _{\beta _1}\left[\mu
^{1/2}\partial ^{\alpha _1}g_1\right]\right\}\partial _{\beta -\beta
_1}^{\alpha -\alpha _1}g_2,\partial _i\partial _\beta ^\alpha g_3\Rrangle
\end{equation}
Note that $\partial _j\Phi ^{ij}(v) = O \left( |v|^{-2}\right)\notin L_{loc}^2({\R}^3)$.
We thus decompose the integral region in the convolution $\partial _j\Phi ^{ij}*\partial _{\beta _1}[\mu ^{1/2}\partial^{\alpha _1}g_1]$
into two parts: $|v-v_*| \ge 1$ and $|v-v_*| \le 1$.  When the integral is restricted to the region $|v-v_*| \ge 1$ the singularity is avoided, so this part of \eqref{haha} can have the same upper bound as \eqref{gamma1}--\eqref{gamma2}. For the remaining part of \eqref{haha} corresponding to the region $|v-v_*| \le 1$, we bound it by
\begin{equation}\label{HLS}
\begin{split}
&C_{(\beta_1,\bar{\beta})}\iint_{|v-v_*| \le 1} |v-v_*|^{-2} w^2(v)\left|\mu^{1/4}(v_\ast)\partial^{\alpha _1}_{\bar{\beta}}g_1(v_*)
\partial _{\beta -\beta_1}^{\alpha -\alpha _1}g_2(v) \partial _i\partial _\beta ^\alpha g_3(v)
\right|dv dv_*
\\&\quad\le C_{(\beta_1,\bar{\beta})}\iint \norms{|v-v_*|^{-2} (\mu^{\delta}\partial^{\alpha _1}_{\bar{\beta}}g_1)(v_*)}
\norms{(w \mu^{ \delta} \partial _{\beta -\beta_1}^{\alpha -\alpha _1}g_2)(v)}
\norms{ (w \mu^{ \delta} \partial _i\partial _\beta ^\alpha g_3)(v)}
 dv dv_*.
\end{split}
\end{equation}
As in Proposition 3.5 of \cite{GS11}, applying H\"older's inequality, the Hardy-Littlewood-Sobolev theorem and the Sobolev embedding, we obtain
\begin{equation}
\begin{split}
\eqref{HLS}&
\le C_{(\beta_1,\bar{\beta})}  \norms{|\cdot|^{-2}*(\mu^{\delta}\partial^{\alpha _1}_{\bar{\beta}}g_1)}_{L^6_v}
\norms{w \mu^{ \delta} \partial _{\beta -\beta_1}^{\alpha -\alpha _1}g_2}_{L^3_v}
\norms{ w \mu^{ \delta} \partial _i\partial _\beta ^\alpha g_3}_{L^2_v}
 \\&\le C_{(\beta_1,\bar{\beta})}\left| \mu^\delta\partial _{\bar{\beta}}^{\alpha _1}g_1\right| _{L^2_v}
\left| \mu^\delta w \partial _{\beta -\beta_1}^{\alpha -\alpha _1}g_2\right| _{H^1_v}
\left|  \mu^\delta w \partial _i\partial _\beta ^\alpha g_3\right| _{L^2_v}
\\
&\le C_{(\beta_1,\bar{\beta})}
\left|\mu^\delta\partial _{\bar{\beta}}^{\alpha _1}g_1\right| _{2}
\left|   \partial _{\beta -\beta_1}^{\alpha -\alpha _1}g_2\right| _{\sigma, w}
\left|  \partial _\beta ^\alpha g_3\right| _{\sigma, w}.
\end{split}
\end{equation}
Therefore, \eqref{gamma4} is bounded via the first term on the right-hand side of \eqref{ga es}.

We next estimate the term \eqref{gamma3}. The second part of \eqref{gamma3} corresponding to the split \eqref{parts} has the same upper bound as \eqref{gamma1}--\eqref{gamma2}; while the first part has the same upper bound as \eqref{haha}. Hence, \eqref{gamma3} is also bounded via the first term on the right-hand side of \eqref{ga es}. We then finally turn to the last term \eqref{part2}. The second part of \eqref{part2} corresponding to the split \eqref{parts} has the same upper bound as \eqref{part1}; while the first part we again decompose the integral region in the convolution into two parts: $|v-v_*| \ge 1$ and $|v-v_*| \le 1$.  When the integral is restricted to the region $|v-v_*| \ge 1$ the singularity is avoided, so this part  can have the same upper bound as \eqref{part1}. For the remaining part corresponding to the region $|v-v_*| \le 1$, we employ the same argument for \eqref{HLS} to bound it by
\begin{equation}
\begin{split}
 &l C_{(\beta_1,\bar{\beta})}\left| \mu^\delta\partial _{\bar{\beta}}^{\alpha _1}g_1\right| _{L^2_v}
\left| \mu^\delta w \partial _{\beta -\beta_1}^{\alpha -\alpha _1}g_2\right| _{H^1_v}
\left|  \mu^\delta w \partial _\beta ^\alpha g_3\right| _{L^2_v}
\\
&\le l C_{(\beta_1,\bar{\beta})}
\left|\mu^\delta\partial _{\bar{\beta}}^{\alpha _1}g_1\right| _{2}
\left|   \partial _{\beta -\beta_1}^{\alpha -\alpha _1}g_2\right| _{\sigma, w}
\left|  \partial _\beta ^\alpha g_3\right| _{2, \frac{w}{\langle v\rangle^{\frac{3}{2}}}}.
\end{split}
\end{equation}
Therefore, \eqref{part2} is bounded via the second term on the right-hand side of \eqref{ga es}.
\end{proof}

In what follows, we will collect the analytic tools which will be used in this paper. The first one is the Sobolev interpolation among the spatial regularity:
 \begin{lemma}\label{interpolation lemma}
Let $2\le p< \infty$ and $ k,\ell,m\in \mathbb{R}$, then we have
\begin{equation}\label{A.1}
\norm{\nabla^k f}_{p}\lesssim \norm{ \nabla^\ell f}_{2}^{\theta}\norm{ \nabla^mf}_{2}^{1-\theta}
\end{equation}
where $0\le \theta\le 1$ and $k$ satisfies
\begin{equation}\label{A.2}
k+3\left(\frac{1}{2}-\frac{1}{p}\right) =m(1-\theta)+\ell\theta.
\end{equation}
\end{lemma}
\begin{proof}
For $2\le p<\infty$, it follows from the classical Sobolev inequality \cite{S70} that
\begin{equation}\label{11in}
\norm{\nabla^k f}_{p}\lesssim \norm{ \nabla^\zeta f}_{2}\hbox{ with }\zeta=k+3\left(\frac{1}{2}-\frac{1}{p}\right).
\end{equation}
By the Parseval theorem and H\"older's inequality, we have
\begin{equation}\label{12in}
\norm{ \nabla^\zeta f}_{2}\lesssim \norm{ \nabla^\ell f}_{2}^{\theta}\norm{ \nabla^mf}_{2}^{1-\theta},
\end{equation}
where $0\le \theta\le 1$ is defined by \eqref{A.2}. Hence, \eqref{A.1} follows by \eqref{11in}--\eqref{12in}.
\end{proof}

We have the following $L^p$ inequality for $\Lambda^{-s}$:
\begin{lemma}\label{Riesz lemma}
Let $0<s<3/2,\ 1<p<2<\infty,\ 1/2+s/3=1/p$, then
\begin{equation}\label{Riesz estimate}
\norm{\Lambda^{-s}f}_{2}\lesssim\norm{ f}_{p}.
\end{equation}
\end{lemma}
\begin{proof}
It follows from the Hardy-Littlewood-Sobolev theorem, see \cite{S70}.
\end{proof}

In many places, we will use the Minkowski's integral inequality to interchange the orders of integration over $x$ and $v$.
\begin{lemma}\label{Minkowski}
For $1\le p\le q\le \infty$, we have
\begin{equation}\label{minkowski inequality}
\norm{f}_{L^q_zL^p_y}\le \norm{f}_{L^p_yL^q_z}.
\end{equation}
\end{lemma}
\begin{proof}
For $q=\infty$ or $p=1$ \eqref{minkowski inequality} is standard, see \cite{S70}. Now for $1\le p\le q<\infty$ and hence $1\le q/p<\infty$, we have
\begin{equation}
\norm{f}_{L^q_zL^p_y}= \left(\int_{\r3_z}\left(\int_{\r3_y}|f|^p\,dy\right)^{q/p}\,dz\right)^{1/q}
 \le\left(\int_{\r3_y}\left(\int_{\r3_z}|f|^q\,dz\right)^{p/q}\,dv\right)^{1/p}=\norm{f}_{L^p_yL^q_z}.
\end{equation}
We thus conclude the lemma.
\end{proof}

We also need the following version of the Gronwall lemma.
\begin{lemma}\label{Gronwall}
Let $A(t),B(t),y(t)\ge 0$ satisfy $y(t)\le \int_0^t A(s)y(s)\,ds+B(t)$, then
\begin{equation}
y(t)\le e^{\int_0^t A(s)\,ds}\int_0^t A(s)B(s)\,ds+B(t).
\end{equation}
\end{lemma}
\begin{proof}
The lemma is standard, see \cite{G12}.
\end{proof}

The last one is the basic time decay estimates of certain integrals.
\begin{lemma}\label{basic decay}
Suppose that $0\le \varepsilon \le 1$, $\lambda > 0$ and  $\vartheta \ge 0$, then
\begin{equation}
\int_0^t  e^{-\lambda \left((1+t)^{1-\varepsilon}- (1+\tau)^{1-\varepsilon}  \right) }
(1+\tau)^{-\vartheta}d\tau
\le C_{\lambda,\vartheta,\varepsilon}(1+ t)^{-\vartheta+\varepsilon}.
\end{equation}
\end{lemma}
\begin{proof}
The proof is standard, see \cite{SZ12}.
\end{proof}

\section{Local solution and basic energy estimates}\label{energy section}

We first record the local-in-time existence of unique solution to the Vlasov-Poisson-Landau system \eqref{VPL_per} if $\mathcal{ {E}}_{2;2,0}(f_{0})$ is sufficiently small.
\begin{theorem}\label{local solution}
Assume that $\mathcal{ {E}}_{2;2,0}(f_{0})$ is
sufficiently small. Then there exist $0<T\leq 1$ and $M>0$ small such that
there is a unique solution $F=\mu +\sqrt{\mu }f \ge 0$ with
\begin{equation}
\mathcal{ {E}}_{2;2,0}(f(t))+\int_{0}^{t} \overline{\mathcal{D}}_{2;2,0}(f(s))\,ds\lesssim
\mathcal{ {E}}_{2;2,0}(f_0) \lesssim M.
\end{equation}%
In general, if $0\leq t\leq T$,  there exists an increasing continuous
function $P_{m,l}(\cdot )$ with $P_{m,l}(0)=0$ such that
\begin{equation}
\mathcal{ {E}}_{m;l,q}(f(t))+\int_{0}^{t} \overline{\mathcal{D}}_{m;l,q}(f(s))\,ds
\lesssim
P_{m,l}(\mathcal{ {E}}_{m;l,q}(f_{0})).
\end{equation}
\end{theorem}

Notice that in this local existence theorem we have included the term $\norm{{\bf P}f}_2^2$ in the dissipation; this allows us to
prove Theorem \ref{local solution} exactly in the same way as \cite{G12} with a little additional attention on the Poisson term which was already presented in \cite{SZ12}. Since the proof is similar to those of \cite{G12} and \cite{SZ12}, we then omit it. However, for our global existence theorem, we need to exclude this term from the dissipation rate. So we can not use the estimates of Lemma 14 in \cite{G12} and we need to refine the energy estimates. We shall establish the following proposition:

\begin{proposition}\label{energy estimate}
Let $f_{0}\in \testF$ and assume $f$ is the solution constructed in Theorem \ref{local solution} with $\mathcal{ {E}}_{2;2,0}(f)\leq M$.

(1) We have
\begin{equation}\label{2lq energy estimate}
\begin{split}
&\mathcal{ {E}}_{2;l,q}(f)+\int_{0}^{t}\widetilde{\mathcal{D}}_{2;l,q}(f)\,ds
\\&\quad\le
 \mathcal{ {E}}_{2;l,q}(f_0) + C_l\int_0^t \left[\norm{\pa_t\phi}_{\infty}+\norm{\na_x \phi }_{\infty}+\widetilde{\mathcal{D}}_{2;0,0}(f)\right]\mathcal{ {E}}_{2;l,q}(f)\,ds .
\end{split}
\end{equation}

(2) If $m\ge 3$, we have
\begin{equation}\label{mlq energy estimate}
\begin{split}
&\mathcal{ {E}}_{m;l,q}(f)+\int_{0}^{t}\widetilde{\mathcal{D}}_{m;l,q}(f)\,ds
\\&\quad\le
\mathcal{ {E}}_{m;l,q}(f_0) +C_{l,m}\int_0^t \left[\norm{\pa_t\phi}_{\infty}+\norm{\na_x \phi }_{\infty}+\widetilde{\mathcal{D}}_{m-1;l,q}(f)\right]\mathcal{ {E}}_{m;l,q}(f)\,ds.
\end{split}
\end{equation}
\end{proposition}

Proposition \ref{energy estimate} will be proved by the following a series of lemmas. We begin with the energy estimates without the velocity weight.

\begin{lemma}\label{energy lemma 1}
Let $f_{0}\in \testF$ and assume $f$ is the solution constructed in Theorem \ref{local solution} with $\mathcal{ {E}}_{2;2,0}(f)\leq M$.
Then we have
\begin{equation}\label{energy 1 1}
 \frac{d}{dt}\left[\int  {\frac{|f|^{2}}{2}} +\int|\nabla_x \phi |^{2}\right] + \lambda \norm{\{{\bf I-P}\}f}_\sigma^2
\lesssim  \sqrt{M}\left( \norm{\na_x  f}_\sigma^2+\norm{\na_x \phi}_2^2 \right).
\end{equation}
If $k=1,2$, we have
\begin{equation}\label{energy 1 2}
\begin{split}
& \frac{d}{dt}\left[\int\sum_{\pm } \frac{{e^{\pm 2\phi }|\na^k f_{\pm }|^{2}}}{2} + \int| \na^k \nabla_x\phi |^{2}\right] +\la\norm{\na^k\{{\bf I-P}\}f}_\sigma^2\\
&\quad\lesssim\norm{\pa_t\phi}_\infty\norm{\na^k f}_2^{2}+%
\sqrt{M}\left(\sum_{1\le \ell\le 2}\norm{\na^\ell f}_{\sigma }^{2}+\norm{\na_x \phi}^2_2\right).
\end{split}
\end{equation}
If $m\geq 3$, we have that for any $\eta >0$,
\begin{equation}\label{energy 1 3}
\begin{split}
& \frac{d}{dt}\left[\int\sum_{\pm } \frac{{e^{\pm 2\phi }|\na^m f_{\pm }|^{2}}}{2} + \int| \na^m \nabla_x\phi |^{2}\right] +\la\norm{\na^m\{{\bf I-P}\}f}_\sigma^2\\
&\quad\lesssim\norm{\pa_t\phi}_\infty\norm{\na^m f}_2^{2}+%
\left(\sqrt{M}+\eta\right)\left(\norm{\na^m f}_{\sigma }^{2}+\norm{\na_x \phi}^2_2\right)
+C_{m,\eta }\mathcal{ E}_{m;m,0}(f)\widetilde{\mathcal{D}}_{m-1;m-1,0}(f).
\end{split}
\end{equation}
\end{lemma}
\begin{proof}
We will use the continuity equation of
\begin{equation}
\partial_t \rho+\nabla_x\cdot J=0
\end{equation}
with $\rho=\int \sqrt{\mu}[f_+-f_-]\,dv$ and $J=\int v\sqrt{\mu}[f_+-f_-]\,dv$. From this and the Poisson equation $-\Delta_x\phi=\rho$, we have
\begin{equation}
2\int \na^k\nabla_x\phi\cdot v\sqrt{\mu}[\na^k f_+-\na^k f_-]= \frac{d}{dt}\int|\na^k\nabla_x\phi|^2.
\end{equation}

For \eqref{energy 1 1}, by \eqref{VPL_per}, we obtain
\begin{equation}\label{110}
\frac{d}{dt}\left[ \int   {\frac{|f|^{2}}{2}} + \int |\nabla_x \phi |^{2}\right] +\int
\langle Lf,f\rangle  = \int \sum_{\pm }\left(  \Ga_\pm(f,f)\mp \nabla_x\phi\cdot v f_\pm\right)f_{\pm }.
\end{equation}
By the collision invariant property and the estimate \eqref{ga es 0} of Lemma \ref{nonlinear c}, we obtain
\begin{equation}\label{111}
\begin{split}
\int \sum_{\pm } \Ga_\pm(f,f)f_{\pm }&= \int \sum_{\pm } \Ga_\pm(f,f)\{\FI-\FP\}f_{\pm }
\\&\lesssim \norm{\norms{ f}_2}_3\norm{\norms{f}_\sigma}_6\norm{\{\FI-\FP\}f}_\sigma
\\&\lesssim \sqrt{\mathcal{ {E}}_{2;2,0}(f)}\left(\norm{\nabla_xf}_\sigma^2+\norm{\{\FI-\FP\}f}_\sigma^2\right).
\end{split}
\end{equation}
Here we have taken $L^3-L^6-L^2$ in the $x$ integration and used the Minkowski's integral inequality, Sobolev's and Cauchy's inequalities. Recalling \eqref{sigma norm =}, we have
\begin{equation}\label{112}
\begin{split}
\int \sum_{\pm } \mp\nabla_x\phi\cdot v f_\pm f_{\pm }&\lesssim  \int  |\nabla_x\phi|\norms{\langle v\rangle^{-1/2}f}_2\norms{\langle v\rangle^{3/2}f}_2
\\&\lesssim \norm{\nabla_x\phi }_2\norm{\norms{f}_\sigma}_6\norm{\norms{\langle v\rangle^{3/2}f}_2}_3
\\&\lesssim \sqrt{\mathcal{ {E}}_{2;2,0}(f)}\left(\norm{\nabla_xf}_\sigma^2+\norm{\nabla_x\phi }_2^2\right).
\end{split}
\end{equation}
Here we have used the fact that $\norm{\langle v\rangle^{3/2}\na^\ell f}_2\lesssim \sqrt{\mathcal{ {E}}_{2;2,0}(f)}$ for $\ell=0, 1$ by \eqref{energy}. Then \eqref{energy 1 1} follows by further using Lemma \ref{linear c} since $\mathcal{ {E}}_{2;2,0}(f)\lesssim M$ is small.

Next letting $k\ge 1$, we prove \eqref{energy 1 2} and \eqref{energy 1 3}. Applying $\na^k$ to \eqref{VPL_per} and then taking the $L^2$ inner product of the resulting identity with $e^{\pm2\phi }\na^k f_{\pm } $, we obtain
\begin{equation}
\begin{split}
& \frac{d}{dt}\left[\int \sum_{\pm }\frac{e^{\pm 2\phi }|\na^k f_{\pm }|^{2}}{2} +\int | \na^k \nabla_x\phi |^{2}\right]
+\int \langle L \na^k f,\na^k f\rangle
 \\&\quad=\sum_{\pm }\int e^{\pm 2\phi }\pa_t\phi|\na^k f_{\pm }|^{2}
+\sum_{\pm }\int \mp(e^{\pm 2\phi }-1)\na^k f_{\pm } \na^k\nabla _{x}\phi \cdot v%
\sqrt{\mu }  \\
&\qquad+\sum_{\pm }\int (1-e^{\pm 2\phi })\na^k f_{\pm }L_{\pm
}\na^k f  +\sum_{\pm }\int e^{\pm 2\phi }\na^k f_\pm\na^k \Gamma _{\pm }(f,f)   \\
&\qquad+\sum_{\pm ,j<k }C_k^j\int e^{\pm 2\phi
}\na^k f_{\pm }\na^{k-j}\nabla _{x}\phi
\cdot \na^j \left(\nabla _{v}f_{\pm }-vf_{\pm }\right):=\sum_{i=1}^5 I_i.
\end{split}
\end{equation}
Note that the weight function $e^{\pm2\phi}$ is so designed such that there was an exact cancelation for the high momentum contributions in the integration, see \cite{G12}.

We now estimate $I_1\sim I_5$. Since $|e^{\pm 2\phi
}-1|\lesssim \norm{\phi }_{\infty }\lesssim \sqrt{\mathcal{ {E}}_{2;2,0}(f)}\lesssim
\sqrt{M}$ by the elliptic estimate, clearly,
\begin{equation}\label{I1b}
I_1 \lesssim \norm{\pa_t\phi}_\infty\norm{\na^k f}_2^{2},
\end{equation}
and by using the exponential decay of $\mu$, we obtain
 \begin{equation}\label{I2b}
I_2  \lesssim  \sqrt{M}\left( \norm{\na^k f}_{\sigma }^2+\norm{\na^k \na_x\phi}^2_2\right)
\lesssim \sqrt{M}\left(\norm{\na^k f}_{\sigma }^2+\norm{\na_x  \phi}^2_2\right).
\end{equation}
By Lemma \ref{linear c weight}, we have
\begin{equation}\label{I3b}
I_3\lesssim \sqrt{M}\norm{\na^k f}_{\sigma }^{2}.
\end{equation}

Now for the term $I_4$, we apply the estimate \eqref{ga es 0} of Lemma \ref{nonlinear c} to obtain
\begin{equation}\label{I40}
I_4\le C \sum_{j\le k}C_k^j\int\norms{\na^j f}_2 \norms{\na^{k-j} f}_\sigma\norms{\na^k f}_\sigma dx.
\end{equation}
For $k=1$, if $j=0$ we take $L^\infty-L^2-L^2$; if $j=1$ we take $L^3-L^6-L^2$ in \eqref{I40} respectively to have an upper bound of
\begin{equation}
C\sqrt{\mathcal{ {E}}_{2;2,0}(f)}\norm{\na f}_{\sigma }^{2}.
 \end{equation}
For $k=2$, if $j=0$ we take $L^\infty-L^2-L^2$; if $j=1$ we take $L^3-L^6-L^2$; if $j=2$ we take $L^2-L^\infty-L^2$ in \eqref{I40} respectively to have an upper bound of
\begin{equation}
C\sqrt{\mathcal{ {E}}_{2;2,0}(f)}\sum_{\ell=1, 2}\norm{\na^\ell f}_{\sigma }^{2}.
 \end{equation}
For $k=m\ge 3$, if $j=0$ we take $L^\infty-L^2-L^2$; if $j=m$ we take $L^2-L^\infty-L^2$ in \eqref{I40} respectively (note that now $C_m^m=1$) to have an upper bound of
\begin{equation}
\begin{split}
&C\sqrt{\mathcal{ {E}}_{2;2,0}(f)}\norm{\na^m f}_{\sigma }^{2}+C\norm{\na^m  f}_2 \sum_{\ell=1,2}\norm{\na^\ell f}_{\sigma }\norm{\na^m f}_\sigma \\&\quad\lesssim \left(\sqrt{\mathcal{ {E}}_{2;2,0}(f)}+\eta\right)\norm{\na^m f}_\sigma^2+  C_\eta \mathcal{ E}_{m;m,0}(f)\widetilde{\mathcal{D}}_{2;2,0}(f).
\end{split}
 \end{equation}
Now if $j=1$ we take $L^\infty-L^2-L^2$; if $2\le j\le m-1$ we take $L^4-L^4-L^2$ in \eqref{I40} respectively to bound them by
\begin{equation}
\begin{split}
&C_m \sum_{\ell=1,2}\norm{\na^\ell \na f}_2 \norm{\na^{m-1} f}_{\sigma }\norm{\na^m f}_\sigma+C_m\sum_{\ell=0, 1\atop 2\le j\le m-1}\norm{\na^\ell \na^j f}_2 \norm{\na^\ell\na^{m-j}  f}_\sigma \norm{\na^m f}_\sigma
\\&\quad\le C_m\sqrt{\mathcal{ E}_{m;m,0}(f)}\sqrt{\widetilde{\mathcal{ D}}_{m-1;m-1,0}(f)}\norm{\na^m f}_\sigma
\\&\quad\le \eta\norm{\na^{m} f}_{\sigma }^2+C_{m,\eta}\mathcal{ E}_{m;m,0}(f)\widetilde{\mathcal{D}}_{m-1;m-1,0}(f).
 \end{split}
 \end{equation}
We thus conclude the estimates of $I_4$. Note that the difference between our estimates of $I_4$ and that of Lemma 8 in \cite{G12} is that we excluded the term $\norm{f}_\sigma^2$ therein.

Finally, we turn to the term $I_5$ and we perform an integration by parts in $v$ and recall \eqref{sigma norm =} to have
\begin{equation} \label{I50}
\begin{split}
I_5&\lesssim \sum_{j<k } C_k^j\int\left|\na^j f  \na^{k-j}\nabla _{x}\phi \cdot \na^k\left(\nabla _{v}f+vf\right) \right|
\\&\lesssim \sum_{j<k } C_k^j\int \norms{\langle v\rangle ^{3/2}\na^j f}_{2}|\na^{k-j}\nabla_{x}\phi|\norms{\na^k f}_{\sigma }dx.
\end{split}
\end{equation}
For $k=1$ then $j=0$, we take $L^3-L^6-L^2$ in \eqref{I50} to have an upper bound of
\begin{equation}
 \sum_{\ell=0, 1}\norm{\langle v\rangle ^{3/2}\na^\ell f }_{2}\norm{\na^2 \nabla _{x} \phi }_{2}\norm{ \na  f}_{\sigma }
\lesssim\sqrt{\mathcal{ {E}}_{2;2,0}(f)}\norm{\na  f}_\sigma^2.
 \end{equation}
For $k=2$, if $j=0$ we take $L^3-L^6-L^2$; if $j=1$ we take $L^2-L^\infty-L^2$ in \eqref{I50} respectively to have an upper bound of
\begin{equation}
\begin{split}
&\sum_{\ell=0, 1}\norm{\langle v\rangle ^{3/2}\na^\ell f }_{2}\norm{\na^3 \nabla _{x} \phi }_{2}\norm{ \na^2 f}_{\sigma }
+\norm{\langle v\rangle ^{3/2}\na  f }_{2}\sum_{\ell=1,2}\norm{\na^{\ell+1}  \nabla _{x} \phi }_{2}\norm{ \na^2 f}_{\sigma }
\\& \quad\lesssim\sqrt{\mathcal{ {E}}_{2;2,0}(f)}\left(\norm{\na^2 f}_{\sigma }^2+\norm{\na_x  \phi}_2^2\right).
\end{split}
 \end{equation}
For $k=m\ge 3$, if $j=0$ we take $L^3-L^6-L^2$  in \eqref{I40} to have an upper bound of
\begin{equation}
 \sum_{\ell=0, 1}\norm{\langle v\rangle ^{3/2}\na^\ell f }_{2}\norm{\na^{m+1} \nabla _{x} \phi }_{2}\norm{ \na^m  f}_{\sigma }
\lesssim\sqrt{\mathcal{ {E}}_{2;2,0}(f)}\norm{\na^m  f}_\sigma^2.
 \end{equation}
Now if $j=m-1$ we take $L^2-L^\infty-L^2$; if $1\le j\le m-2$ we take $L^4-L^4-L^2$ in \eqref{I50} respectively to bound them by
\begin{equation}
\begin{split}
&C_m \norm{\langle v\rangle^{3/2}\na^{m-1}  f }_{2} \sum_{\ell=1, 2}\norm{\na^{\ell+1}  \nabla _{x}\phi }_{2}\norm{ \na^m f}_{\sigma}
\\&\quad+C_m
\sum_{\ell=0,1\atop 1\le j\le m-1}\norm{\langle v\rangle ^{3/2}\na^\ell \na^j f }_{2}\norm{\na^\ell \na^{m-j} \nabla _{x}\phi }_{2} \norm{ \na^m f}_{\sigma}
\\&\qquad\le C_m\sqrt{\mathcal{ E}_{m;m,0}(f)}\sqrt{\widetilde{\mathcal{ {D}}}_{m-1;m-1,0}(f)}\norm{\na^m f}_\sigma
\\&\qquad\le \eta\norm{\na^m f}_\sigma^2+ C_{m,\eta} \mathcal{E}_{m;m,0}(f)\widetilde{\mathcal{D}}_{m-1;m-1,0}(f).
\end{split}
 \end{equation}
Here we have use the fact that $\norm{\langle v\rangle ^{3/2}\na^\ell f }_{2} \le \sqrt{\mathcal{ E}_{m;m,0}(f)}$ for $0\le\ell\le m-1$. We therefore conclude the estimates of $I_5$. Note that the difference between our estimates of $I_5$ and that of Lemma 10 in \cite{G12} is that we had to replace $\norm{\na^k f}_\sigma^2$ therein by $\norm{\na^k f}_\sigma^2+\norm{\na_x\phi}_2^2$ since the Poincar\'e inequality fails.

Consequently, summing up the estimates for $I_1\sim I_5$, by Lemma \ref{linear c}, we obtain \eqref{energy 1 2} and \eqref{energy 1 3}. We thus conclude the lemma.
\end{proof}

Notice that the dissipation estimate in Lemma \ref{energy lemma 1} only controls the microscopic part $\{\FI-\FP\}f$, we then want to include the hydrodynamic part $\FP f$ and the electric potential $\phi$ to get the full dissipation estimate.
\begin{lemma}\label{inter lemma}
Let $f_{0}\in \testF$ and assume $f$ is the solution constructed in Theorem \ref{local solution} with
$\mathcal{{E}}_{2;2,0}(f)\leq M$. For $k\ge 0$, there exists a function $G^k(t)$ with
\begin{equation}\label{Gk}
G^k(t)\lesssim  \norm{\na^k   f}_{ 2}^2+\norm{ \na^{k+1} f}_{ 2}^2+\norm{ \na^{k} \na_x\phi}_{ 2}^2
\end{equation}
such that if $k=0,1$, then
\begin{equation}\label{Gk inequality}
\begin{split}
&\frac{d}{dt}G^k(t)+ \norm{ \na^{k+1}{\bf P} f}_{2}^2+\norm{ \na^k \nabla_x\phi}_{2}^2
\\&\quad\lesssim  \norm{\na^k \{{\bf I- P}\} f}_{\sigma}^2+\norm{\na^{k+1} \{{\bf I- P}\} f}_{\sigma}^2 +M \norm{\na_x f}_{\sigma}^2;
\end{split}
\end{equation}
if $ k=m-1$ with $m\ge 3$, then
\begin{equation}
\begin{split}
&\frac{d}{dt}G^k(t)+ \norm{ \na^{k+1}{\bf P} f}_{2}^2+\norm{ \na^k \nabla_x\phi}_{2}^2
\\&\quad\lesssim  \norm{\na^k \{{\bf I- P}\} f}_{\sigma}^2+\norm{\na^{k+1} \{{\bf I- P}\} f}_{\sigma}^2 +C_m\mathcal{E}_{m-1;m-1,0}(f ) \widetilde{\mathcal{D}}_{m-1;m-1,0}(f ).
\end{split}
\end{equation}
\end{lemma}
\begin{proof}
From Proposition 16 of Guo \cite{G12} by using the local conservation laws and the macroscopic equations which are derived from the the so-called macro-micro decomposition, we have that for any $k\ge 0$, there exists a function $G^k(t)$ satisfying \eqref{Gk} ($\|\na^{k} \na_x\phi\|_{ 2}^2$ is included since the Poincar\'e inequality fails) such that
\begin{equation}
\begin{split}
&\frac{d}{dt}G^k(t)+ \norm{ \na^{k+1}{\bf P} f}_{2}^2+\norm{ \na^k \nabla_x\phi}_{2}^2
\\&\quad\lesssim  \norm{\na^k \{{\bf I- P}\} f}_{\sigma}^2+\norm{\na^{k+1} \{{\bf I- P}\} f}_{\sigma}^2 +\norm{\na^k N_\parallel }_{2}^2.
\end{split}
\end{equation}
Here $N_\parallel$ denotes the $L^2_v$ projection of $N_\pm(f)$ with respect to the subspace
\begin{equation}
X_v\equiv \mathrm{span}\left\{\sqrt{\mu}, v_i\sqrt{\mu},v_iv_j\sqrt{\mu},v_i|v|^2\sqrt{\mu}\right\},
\end{equation}
where $N_\pm(f) \equiv\Gamma_\pm(f,f)\pm\nabla_x\phi\cdot(\nabla_vf_\pm- vf_\pm)$ representing the nonlinear term of \eqref{VPL_per}. It then suffices to estimate $\norm{\na^k N_\parallel }_{2}^2$. Let $\tilde{\mu}$ be any function in $X_v$, we have
\begin{equation}
\norm{\na^k N_\parallel }_{2}^2\lesssim \sum_{j \le k}C_k^j\norm{\Llangle \Gamma(\na^j f,\na^{k-j} f)+\na^j\nabla_x\phi\cdot   \na^{k-j}(\nabla_v f-vf), \tilde{\mu}\Rrangle}_{2}^2.
\end{equation}

We apply the estimate \eqref{ga es 0} of Lemma \ref{nonlinear c} to obtain
\begin{equation}\label{120}
\norm{\Llangle \Gamma(\na^j f,\na^{k-j} f), \tilde{\mu}\Rrangle}_{2}^2\lesssim
\norm{\norms{\na^j f}_2 \norms{\na^{k-j}f}_\sigma}_2^2.
\end{equation}
 When $k=0,1$, if $j=k$ we take $L^3-L^6$; and if $k-j=1$ we take $L^\infty-L^2$ in \eqref{120} respectively to bound these two cases by
\begin{equation}
\sum_{0\le \ell\le 2}\norm{\na^\ell f}_2^2 \norm{\na_x f}_\sigma^2\lesssim  {\mathcal{ {E}}_{2;2,0}(f)}\norm{\na_x f}_\sigma^2.
\end{equation}
When $k=m-1$ with $m\ge 3$, if $j=m-1$ we take $L^2-L^\infty$; if $j=0$ we take $L^\infty-L^2$; if $1\le j\le m-2$ we take $L^3-L^6$ in \eqref{120} respectively to bound them by
\begin{equation}
\mathcal{E}_{m-1;m-1,0}(f ) \widetilde{\mathcal{D}}_{m-1;m-1,0}(f ).
\end{equation}

Now, we use the integration by parts in $v$ and the exponential decay of $\tilde{\mu}$ to get
\begin{equation}\label{121}
\begin{split}
 \norm{\na^j\nabla_x\phi\cdot\langle  \na^{k-j}(\nabla_v f-vf), \tilde{\mu}\rangle}_{2}^2
  &= \norm{\na^j\nabla_x\phi\cdot\langle  \na^{k-j}  f , \nabla_v\tilde{\mu}+v\tilde{\mu}\rangle}_{2}^2
  \\&\lesssim \norm{\norms{\na^j\nabla_x\phi}\norms{ \na^{k-j}  f}_{\sigma}}_2^2.
  \end{split}
\end{equation}
Obviously, applying the same arguments as those for \eqref{120}, we obtain the same upper bounds for \eqref{121}. We then conclude our lemma.
\end{proof}

Next, we turn to the energy estimates with the velocity weight, and we first deal with the pure spatial derivatives of the solution.
\begin{lemma}\label{energy lemma 2}
Let $f_{0}\in \testF$ and assume $f$ is the solution constructed in Theorem \ref{local solution} with
$\mathcal{{E}}_{2;2,0}(f)\leq M$. Let $w=w(\al,0)$ in \eqref{weight}. Then for $|\alpha |=1,2$,
\begin{equation}
\begin{split}
&\frac{d}{dt} \int \sum_{\pm }\frac{e^{\pm 2(q+1)\phi }w^2|\partial ^{\alpha
}f_{\pm }|^{2}}{2} +\lambda\norm{\partial^\al f}^2_{\si,w}  \\
&\quad\le  C_l \left[\norm{\pa_t\phi}_{\infty}+\norm{\na_x \phi }_{\infty}\right]\norm{\partial ^{\alpha }f }_{2,w}^{2}+C\left(\sqrt{M}+\eta\right)\widetilde{\mathcal{D}}_{2;l,q}(f)
\\&\qquad+C_{l,\eta}\left(\norm{\partial^\al f}_\si^2+ \norm{\na_x \phi}^2_2+\mathcal{E}_{2;l,q}(f ) \widetilde{\mathcal{D}}_{2;2,0}(f )\right).
\end{split}
\end{equation}
For $|\alpha|=m\ge 3$,
\begin{equation}
\begin{split}
&\frac{d}{dt} \int \sum_{\pm }\frac{e^{\pm 2(q+1)\phi }w^2|\partial ^{\alpha
}f_{\pm }|^{2}}{2} +\lambda\norm{\partial^\al f}^2_{\si,w}  \\
&\quad\le   C_{l,m}\left[\norm{\pa_t\phi}_{\infty}+\norm{\na_x \phi }_{\infty}\right]\norm{\partial ^{\alpha }f }_{2,w}^{2}+C\left(\sqrt{M}+\eta\right)\widetilde{\mathcal{D}}_{m;l,q}(f)
\\&\qquad+ C_{l,m,\eta}\left(\norm{\partial^\al f}_\si^2+ \norm{\na_x \phi}^2_2
+\mathcal{E}_{m;l,q}(f ) \widetilde{\mathcal{D}}_{m-1;l,q}(f )\right).
\end{split}
\end{equation}
\end{lemma}
\begin{proof}
Applying $\partial^\alpha$ with $|\alpha|=m\ge 1$ to \eqref{VPL_per} and then taking the $L^2$ inner product of the resulting identity with
$e^{\pm2(q+1)\phi }w^{2}\partial^{\alpha }f_{\pm } $, we obtain
\begin{equation}
\begin{split}
&\frac{d}{dt} \int \frac{e^{\pm 2(q+1)\phi  }w^{2}|\partial^{\alpha }f_{\pm }|^{2}}{2} +\int \left\langle
w^{2}L_\pm\partial^{\alpha } f ,\partial^{\alpha
}f_\pm \right\rangle    \\
& \quad=\mp \int \left[\frac{2(l-|\alpha |)}{1+|v|^{2}}\nabla _{x}\phi
 \cdot v-(q+1)\pa_t\phi \right]e^{\pm 2(q+1)\phi  }w^{2}|\partial^{\alpha }f_{\pm }|^{2}
 \\&\qquad\mp 2\int e^{\pm 2(q+1)\phi  }w^{2}\partial ^{\alpha }\nabla _{x}\phi
 \cdot v\sqrt{\mu }\partial^{\alpha }f_{\pm
}
\\
&\qquad+\int\left(1-e^{\pm 2(q+1)\phi  }\right)w^{2}\partial^{\alpha }f_{\pm
}L_\pm\partial^{\alpha } f+\int e^{\pm 2(q+1)\phi  } w^{2}\partial^{\alpha }f_{\pm }\partial^{\alpha }\Gamma _{\pm
}(f ,f )
\\&\qquad\pm \sum_{\alpha _{1}<\alpha }C_{\alpha }^{\alpha _{1}}\int e^{\pm
2(q+1)\phi  }w^{2}\partial^{\alpha }f_{\pm } \partial
^{\alpha -\alpha _{1}}\nabla _{x}\phi  \cdot \partial^{\alpha _{1}}\left(\nabla _{v}f_{\pm } -vf_{\pm } \right):=\sum_{i=1}^5I_i.
\end{split}
\end{equation}

First, by Lemma \ref{linear c weight}, we have
\begin{equation}
\int \Llangle w^{2}L\partial ^{\alpha }f,\partial
^{\alpha }f\Rrangle dx\gtrsim \norm{\partial ^{\alpha
}f}_{\sigma ,w}^{2}-C_{l,m}\norm{ \partial ^{\alpha
}f}_{\sigma}^{2}.
\end{equation}%

We now estimate $I_1\sim I_5$. Clearly,
\begin{equation}
I_1\le C_{l,m} \left[\norm{\pa_t\phi}_{\infty }+\norm{\nabla
_{x}\phi }_{\infty }\right]\norm{\partial^{\alpha }f
}_{2,w}^{2},
\end{equation}
and since $0\le q\ll 1$,
\begin{equation}
I_2 \le C_{l,m}\left( \norm{\partial ^{\alpha }\nabla _{x}\phi
 }_2^2+ \norm{\mu^\delta\partial^{\alpha }f}_2^2\right)
\le C_{l,m}\left(\norm{\nabla _{x}\phi
}_2^2+ \norm{\partial ^{\alpha }f }_\sigma^2\right).
\end{equation}
By Lemma \ref{linear c weight} again, we have
\begin{equation}
I_3 \lesssim \sqrt{M}\norm{\partial ^{\alpha
}f}_{\sigma ,w}^{2}+\sqrt{M}C_{l,m}\norm{ \partial ^{\alpha
}f}_{\sigma}^{2}.
\end{equation}

Now for the term $I_4$, we apply the estimate \eqref{ga es} of Lemma \ref{nonlinear c} to get
\begin{equation} \label{J40}
\begin{split}
&I_4\lesssim \sum_{\alpha_1 \le \alpha}C_{\al}^{\al_1} \norms{\partial^{\alpha -\alpha _1}f}_{\sigma ,w}\left\{
\norms{\partial^{\alpha _1}f}_{2}
\norms{\partial^\alpha f}_{\sigma ,w}
 +l \norms{\partial^{\alpha _1}f}_{\sigma}
\norms{\partial^\alpha f}_{2 ,w}\right\},
\end{split}
\end{equation}
where we have used $\norms{\mu^\delta\partial^{\alpha _1}f}_{2}\le \norms{\partial^{\alpha _1}f}_{2}$, $\norms{\mu^\delta\partial^{\alpha _1}f}_{2}\le \norms{\partial^{\alpha _1}f}_{\sigma}$ and $\norms{\partial^\alpha f}_{2 ,\frac{w}{\langle v\rangle^\frac{3}{2}}}\le \norms{\partial^\alpha f}_{2 ,w}$.

For $|\al|=1$, if $\al_1=0$ we take $L^2-L^\infty-L^2$; if $\al-\al_1=0$, we take $L^6-L^3-L^2$
 in \eqref{J40} to have an upper bound of
 \begin{equation}
\begin{split}
&\norm{\partial ^{\alpha } f}_{\sigma,w}\sum_{1\le |\gamma|\le 2}\left\{ \norm{\pa^\ga  f}_2\norm{\partial ^{\alpha } f}_{\sigma,w} +C_l \norm{\pa^\ga  f}_\sigma\norm{\partial ^{\alpha } f}_{2,w}\right\}
\\&\quad +\sum_{|\gamma|= 1} \norm{\pa^\ga f}_{\sigma,w}\sum_{|\gamma|\le 1} \left\{\norm{\pa^\ga \pa^\al f}_2\norm{\partial ^{\alpha } f}_{\sigma,w} +C_l \norm{\pa^\ga \pa^\al f}_\sigma\norm{\partial ^{\alpha } f}_{2,w}\right\}
\\&\qquad\lesssim  \sqrt{\widetilde{\mathcal{ D}}_{1;l,q}(f)}\sqrt{\mathcal{E}_{2;2,0}(f)}\sqrt{\widetilde{\mathcal{ D}}_{1;l,q}(f)}+  l\sqrt{\widetilde{\mathcal{ D}}_{1;l,q}(f)}\sqrt{\widetilde{\mathcal{ D}}_{2;2,0}(f)}\sqrt{\mathcal{ {E}}_{1;l,q}(f)}
\\&\qquad\lesssim \left(\sqrt{\mathcal{E}_{2;2,0}(f)}+\eta\right)\widetilde{\mathcal{ D}}_{2;l,q}(f)+  C_{l,\eta} \mathcal{ {E}}_{2;l,q}(f) \widetilde{\mathcal{ D}}_{2;2,0}(f).
\end{split}
 \end{equation}
Here we have used the fact $w(\alpha,0)= w(\gamma,0)$ for $|\gamma|=1$.

For $|\al|=2$, if $\al_1=0$ we take $L^2-L^\infty-L^2$; if $\al-\al_1=0$ we take $L^\infty-L^2-L^2$; if $|\al_1|=|\al-\al_1|=1$ we take $L^4-L^4-L^2$ in \eqref{J40} respectively to have an upper bound of
 \begin{equation}\label{te1}
\begin{split}
&\norm{\partial ^{\alpha } f}_{\sigma,w}\sum_{1\le |\gamma|\le 2}\left\{ \norm{\pa^\ga  f}_2\norm{\partial ^{\alpha } f}_{\sigma,w} +l \norm{\pa^\ga  f}_\sigma\norm{\partial ^{\alpha } f}_{2,w}\right\}
\\& \quad+\sum_{1\le |\gamma|\le 2} \norm{\pa^\ga f}_{\sigma,w} \left\{\norm{  \pa^\al f}_2\norm{\partial ^{\alpha } f}_{\sigma,w} +l \norm{  \pa^\al f}_\sigma\norm{\partial ^{\alpha } f}_{2,w}\right\}
\\& \quad+\sum_{|\gamma|\le 1\atop |\al_1|=1} \norm{\pa^\ga\pa^{\al-\al_1} f}_{\sigma,w}\left\{\norm{\pa^\ga \pa^{\al_1} f}_2\norm{\partial ^{\alpha } f}_{\sigma,w} +l\norm{\pa^\ga \pa^{\al_1} f}_\sigma\norm{\partial ^{\alpha } f}_{2,w}\right\}
\\&\qquad\lesssim  \sqrt{\widetilde{\mathcal{ D}}_{2;l,q}(f)}\sqrt{\mathcal{E}_{2;2,0}(f)}\sqrt{\widetilde{\mathcal{ D}}_{2;l,q}(f)}+  l\sqrt{\widetilde{\mathcal{ D}}_{2;l,q}(f)}\sqrt{\widetilde{\mathcal{ D}}_{2;2,0}(f)}\sqrt{\mathcal{ {E}}_{2;l,q}(f)}
\\&\qquad\lesssim \left(\sqrt{\mathcal{E}_{2;2,0}(f)}+\eta\right)\widetilde{\mathcal{ D}}_{2;l,q}(f)+  C_{l,\eta} \widetilde{\mathcal{ D}}_{2;2,0}(f)\mathcal{ {E}}_{2;l,q}(f).
\end{split}
 \end{equation}
 Here we have used the fact $w(\alpha,0)\le w(\gamma,0)$ for $|\gamma|\le 2$.

For $|\al|=m\ge 3$, note that if either $\al_1=0$ or $\al-\al_1=0$ we have $C_\al^{\al_1}=1$. Then as in \eqref{te1} we take $L^\infty$ of the term without derivatives and $L^2$ of the other two terms in \eqref{J40} respectively to bound these two cases by
\begin{equation}
\left(\sqrt{\mathcal{E}_{2;2,0}(f)}+\eta\right)\widetilde{\mathcal{ D}}_{m;l,q}(f)+  C_{l,m,\eta} \widetilde{\mathcal{ D}}_{2;2,0}(f)\mathcal{ {E}}_{m;l,q}(f).
 \end{equation}
If $|\al_1|=1$ we take $L^2-L^\infty-L^2$ in \eqref{J40} to have an upper bound of
 \begin{equation}
\begin{split}
 &C_{l,m}\norm{\pa^{\al-\al_1} f}_{\sigma,w}\sum_{1\le |\gamma|\le 2}\left\{\norm{\pa^\ga\pa^{\al_1} f}_2\norm{\partial ^{\alpha } f}_{\sigma,w} +\norm{\pa^\ga \pa^{\al_1} f}_\sigma\norm{\partial ^{\alpha } f}_{2,w}\right\}
\\&\quad\le C_{l,m} \sqrt{\widetilde{\mathcal{ D}}_{m-1;l,q}(f)}\left\{\sqrt{\mathcal{ {E}}_{3;3,0}(f)}\sqrt{ \widetilde{\mathcal{ D}}_{m;l,q}(f)}+  \sqrt{\widetilde{\mathcal{ D}}_{3;3,0}(f)}\sqrt{ \mathcal{ {E}}_{m;l,q}(f)}\right\}
\\&\quad\le\eta \widetilde{\mathcal{ D}}_{m;l,q}(f)+C_{l,m,\eta}\widetilde{\mathcal{ D}}_{m-1;l,q}(f)\mathcal{ {E}}_{m;l,q}(f).
\end{split}
 \end{equation}
If $2\le |\al_1|\le m-1$ we take $L^4-L^4-L^2$ in \eqref{J40} to have an upper bound of
 \begin{equation}
\begin{split}
 & C_{l,m}\sum_{|\gamma|\le 1} \norm{\pa^\ga\pa^{\al-\al_1} f}_{\sigma,w}\left\{\norm{\pa^\ga\pa^{\al_1} f}_2\norm{\partial ^{\alpha } f}_{\sigma,w} +\norm{\pa^\ga\pa^{\al_1} f}_\sigma\norm{\partial ^{\alpha } f}_{2,w}\right\}
\\&\quad\le C_{l,m} \sqrt{\widetilde{\mathcal{ D}}_{m-1;l,q}(f)}\left\{\sqrt{\mathcal{ {E}}_{m;m,0}(f)}\sqrt{ \widetilde{\mathcal{ D}}_{m;l,q}(f)}+  \sqrt{\widetilde{\mathcal{ D}}_{m;m,0}(f)}\sqrt{ \mathcal{ {E}}_{m;l,q}(f)}\right\}
\\&\quad\le\eta \widetilde{\mathcal{ D}}_{m;l,q}(f)+C_{l,m,\eta}\widetilde{\mathcal{ D}}_{m-1;l,q}(f)\mathcal{ {E}}_{m;l,q}(f).
\end{split}
 \end{equation}

Finally, we turn to deal with the term $I_5$ and we shall use the split $f={\bf P}f+\{{\bf I-P}\}f$. For the hydrodynamic part, we have
\begin{equation}
\begin{split}
&\sum_{\alpha _{1}<\alpha }C_{\alpha }^{\alpha _{1}}\int e^{\pm
2(q+1)\phi  }w^{2}\partial^{\alpha }f_{\pm } \partial
^{\alpha -\alpha _{1}}\nabla _{x}\phi  \cdot \partial^{\alpha _{1}}\left(\nabla _{v}\FP f_{\pm } -v \FP f_{\pm } \right)
\\&\quad\lesssim
\sum_{\alpha _{1}<\alpha }C_{\alpha }^{\alpha _{1}}
\int \norms{\partial ^{\alpha
_{1}}f}_{2,w}\norms{\partial
^{\alpha -\alpha _{1}}\nabla _{x}\phi}\norms{\partial ^{\alpha }f}_{\sigma,w}dx.
\end{split}
\end{equation}
This term can be estimated as \eqref{I50} with only the  unweighted norms replaced by the weighted norms in the upper bounds. Since the microscopic part is always part of our dissipation rate, we can use the argument of Lemma 9 in \cite{G12} to obtain that for $|\al|=m\ge 2$,
\begin{equation}
\begin{split}
&\sum_{\alpha _{1}<\alpha }C_{\alpha }^{\alpha _{1}}\int e^{\pm
2(q+1)\phi  }w^{2}\partial^{\alpha }f_{\pm } \partial
^{\alpha -\alpha _{1}}\nabla _{x}\phi  \cdot \partial^{\alpha _{1}}\left(\nabla _{v}\{\FI-\FP\} f_{\pm } -v \{\FI-\FP\} f_{\pm } \right)
\\&\quad\le C_m \norm{\partial ^{\alpha }  f }_{\sigma
,w}\left(\norm{\partial^\alpha \nabla_x\phi}_{H^1}\norm{\norms{\{\FI-\FP\} f}_{\sigma,\frac{w(0,0)}{\langle v\rangle^2}}}_{H^{\frac{3}{4}}} +\norm{ \nabla_x^2\phi}_{H^2 \cap H^{m-1}}\widetilde{\mathcal{D}}%
_{m-1;l,q}(f) \right).
\end{split}
\end{equation}
These two estimates imply that for $|\alpha|=2$,
\begin{equation}
I_5\lesssim \sqrt{\mathcal{ E} _{2;2,0}(f)}\widetilde{\mathcal{ D}}_{2;l,q}(f);
\end{equation}
and for $|\alpha|=m\ge 3$,
\begin{equation}
I_5\lesssim \eta\norm{\partial ^{\alpha}f}_{\sigma,w}^2 +C_{m,\eta}  \widetilde{\mathcal{ D}}_{m-1;l,q}(f){\mathcal{ {E}}_{m;l,q}(f)} .
 \end{equation}

Consequently, collecting the estimates we thus conclude the lemma.
\end{proof}

We now turn to the mixed spatial-velocity derivatives of the solution. First notice that in view of Lemma \ref{energy lemma 1} and Lemma \ref{energy lemma 2}, it suffices to estimate the remaining microscopic part $\partial^\al_\beta \{{\bf I-P}\}f$
for $|\al|+|\beta|\le m$ and $|\al|\le m-1$ with $m\ge 2$. We   use the macro-micro decomposition:
\begin{equation}\label{I-P equation}
\begin{split}
& \{\partial _{t} +v\cdot \nabla _{x} \mp \nabla_x\phi\cdot \nabla _{v} \}\{\FI-\FP\}f_{\pm }
\pm2\{\FI-\FP\}( \nabla_x\phi\cdot v \sqrt{\mu})+L_{\pm }f \\
&\quad=  \Gamma _{\pm
}(f,f)\mp \nabla_x\phi\cdot v\{\FI-\FP\}f_{\pm }+\FP(v\cdot \nabla _{x}f_{\pm })-v\cdot \nabla _{x}\FP f_{\pm } \\
&\qquad \pm \nabla_x\phi\cdot \left(\FP(vf_{\pm }) - v \FP f_{\pm }
 + \nabla _{v}\FP f_{\pm }-  \FP(\nabla _{v}f_{\pm })\right).
\end{split}
\end{equation}
\begin{lemma}\label{energy I-P}
Assume $f_{0}\in \testF$ and assume $f$ is the solution constructed in Theorem \ref{local solution} with $\mathcal{ {E}}%
_{2;2,0}(f)\leq M$. Let $w=w(\al,\be)$ in \eqref{weight}. For $|\alpha |+ |\beta| \leq 2$ with $|\al|\leq 1$, we have that for any $\eta>0$,
\begin{equation}\label{I-P 2 inequality}
\begin{split}
&\frac{d}{dt} \int \sum_{\pm }\frac{e^{\pm 2(q+1)\phi }w^2|\partial_{\beta} ^{\alpha
}\{\FI-\FP\}f_{\pm }|^{2}}{2}+\la\norm{\partial_{\beta}^\al\{\FI-\FP\}f }_{\si, w}^2
\\&\quad
\le C_l\left[\norm{\pa_t\phi}_{\infty}+\norm{\na_x \phi}_{\infty}\right]\norm{\partial_{\beta}^\al\{\FI-\FP\}f }_{2, w}^2
+C\left(\sqrt{M}+\eta\right)\widetilde{\mathcal{D}}_{2;l,q}(f )\\&\qquad +C_{l,\eta} \left(\sum_{\beta'<\beta}\norm{\partial _{\beta' }^{\alpha }\{\FI-\FP\} f }_{\sigma ,w(\alpha ,\beta'
)}^{2}+\norm{\partial ^{\alpha }\{\FI-\FP\} f }_{\sigma}^{2}+\norm{\nabla_x\phi}_{2}^2+\norm{\na^{|\al|+1}f}^2_\si\right).
\end{split}
\end{equation}
For $|\alpha |+ |\beta|\le m$ with $m\ge 3$ and $|\al|\leq m-1$, we have that for any $\eta>0$,
\begin{equation}\label{I-P m inequality}
\begin{split}
&\frac{d}{dt} \int \sum_{\pm }\frac{e^{\pm 2(q+1)\phi }w^2|\partial_{\beta} ^{\alpha
}\{\FI-\FP\}f_{\pm }|^{2}}{2}+\norm{\partial_{\beta}^\al\{\FI-\FP\}f}_{\si, w}^2
\\&\quad
\le C_{l,m} \left[\norm{\pa_t\phi}_{\infty}+\norm{\na_x \phi}_{\infty}\right]\norm{\partial_{\beta}^\al\{\FI-\FP\}f }_{2, w}^2
+C\left(\sqrt{M}+\eta\right)\widetilde{\mathcal{D}}_{m;l,q}(f )
 \\&\qquad+C_{l,m,\eta}\left( \left[{{\mathcal{E}_{m;l,q}(f )}}+1\right]  {{\widetilde{\mathcal{D}}_{m-1;l,q}(f )}} +\norm{\na^{|\al|+1}f}^2_\si\right).
\end{split}
\end{equation}
\end{lemma}

\begin{proof}

Applying $\partial_\beta^\alpha$ with $|\al|+|\beta|\le m$ for $m\ge 2$ and $|\al|\le m-1$ to \eqref{I-P equation} and then taking the $L^2$ inner product of the resulting identity with
$e^{\pm2(q+1)\phi}w^{2}\partial_{\beta }^{\alpha }\{{\bf I-P}\}f_{\pm }$, we obtain
\begin{eqnarray}
\begin{split}
& \frac{d}{dt}  \int \frac{e^{\pm 2(q+1)\phi }w^{2}|\partial
_{\beta }^{\alpha }\{\FI-\FP\}f_{\pm }|^{2}}{2}  +\int \Llangle
w^{2}\partial _{\beta }^{\alpha }L_\pm\{\FI-\FP\}f,\partial _{\beta }^{\alpha
}\{\FI-\FP\}f_\pm\Rrangle    \\
& \quad=-\int e^{\pm 2(q+1)\phi }w^{2}\delta _{\beta }^{\mathbf{e}%
_{i}}\partial _{\beta -\mathbf{e}_{i}}^{\alpha +\mathbf{e}_{i}}\{\FI-\FP\}f_{\pm }\partial _{\beta }^{\alpha }\{\FI-\FP\}f_{\pm }  \\
& \qquad\mp \int \left[\frac{2(l-|\alpha |-|\beta |)}{1+|v|^{2}}\nabla _{x}\phi
\cdot v-(q+1)\pa_t\phi\right]e^{\pm 2(q+1)\phi }w^{2}|\partial _{\beta
}^{\alpha }\{\FI-\FP\}f_{\pm }|^{2}
\\& \qquad+\int\left(1-e^{\pm 2(q+1)\phi }\right)w^{2}\partial _{\beta }^{\alpha }\{\FI-\FP\}f_{\pm }\partial _{\beta }^{\alpha }L_\pm\{\FI-\FP\}f \\
& \qquad \mp 2\int e^{\pm 2(q+1)\phi}w^2\partial^\al_\be\{\FI-\FP\}f_\pm\partial^\al_\be\{\FI-\FP\}( \nabla _{x}\phi\cdot v \sqrt{\mu}) \\
& \qquad+\int e^{\pm 2(q+1)\phi }w^{2}\partial _{\beta }^{\alpha }\Gamma _{\pm
}(f,f)\partial _{\beta }^{\alpha }\{\FI-\FP\}f_{\pm }
\\
& \qquad\pm \sum_{\alpha _{1}<\alpha }C_{\alpha }^{\alpha _{1}}\int e^{\pm
2(q+1)\phi }w^{2}\partial _{\beta }^{\alpha }\{\FI-\FP\}f_{\pm }\partial
^{\alpha -\alpha _{1}}\nabla _{x}\phi
\\&  \qquad\qquad\qquad\qquad
\cdot\left( \nabla _{v}\partial _{\beta
}^{\alpha _{1}}\{\FI-\FP\}f_{\pm }  -\partial _{\beta }[v\partial
^{\alpha _{1}}\{\FI-\FP\}f_{\pm }] \right)   \\
& \qquad\pm \sum_{\alpha _{1}\le\alpha }C_{\alpha }^{\alpha _{1}}\int  e^{\pm
2(q+1)\phi }w^{2}\partial _{\beta }^{\alpha }\{\FI-\FP\}f_{\pm }\partial
^{\alpha -\alpha _{1}}\nabla _{x}\phi
\\&  \qquad\qquad\qquad\qquad
\cdot \partial _{\beta }^{\alpha _{1}}\left(\FP(v f_{\pm })-v\partial
^{\alpha _{1}}\FP f_{\pm }+\nabla _{v} \FP f_{\pm }-\FP(\na_v f_{\pm })  \right)   \\
& \qquad+\int e^{\pm
2(q+1)\phi }w^{2}\partial _{\beta }^{\alpha }\{\FI-\FP\}f_{\pm }\partial _{\beta }^{\alpha }\left(\FP(v\cdot\na_x f_{\pm})
-v\cdot\na_x\FP f_{\pm}\right):=\sum_{i=1}^8 I_i.
\end{split}
\end{eqnarray}
Here $\delta _{\beta }^{\mathbf{e}_{i}}=1$ if $\mathbf{e}_{i}\leq
\beta$;  otherwise, $\delta _{\beta }^{\mathbf{e}_{i}}=0$.

First, by Lemma \ref{linear c weight}, we have that if $\beta=0$,
\begin{equation}
\int \langle
w^{2}\partial ^{\alpha }L \{\FI-\FP\}f,\partial ^{\alpha
}\{\FI-\FP\}f \rangle
 \gtrsim \norm{\partial ^{\alpha }\{\FI-\FP\} f }_{\sigma ,w}^{2}-C_{l,m} \norm{\partial ^{\alpha }\{\FI-\FP\} f }_{\sigma}^{2};
\end{equation}
if $\beta\neq 0$, then for any $\eta >0$,
\begin{equation}
\begin{split}
&\int \langle
w^{2}\partial _{\beta }^{\alpha }L\{\FI-\FP\}f,\partial _{\beta }^{\alpha
}\{\FI-\FP\}f\rangle    \\
&\quad \geq \norm{\partial _{\beta }^{\alpha }\{\FI-\FP\} f }_{\sigma ,w}^{2}-\eta \sum_{|\beta'|=|\beta|}\norm{\partial _{\beta' }^{\alpha }\{\FI-\FP\} f }_{\sigma ,w(\alpha ,\beta'
)}^{2}
-C_{l,m,\eta} \sum_{\beta'<\beta}\norm{\partial _{\beta' }^{\alpha }\{\FI-\FP\} f }_{\sigma ,w(\alpha ,\beta'
)}^{2}.
\end{split}
\end{equation}

We now estimate $I_1\sim I_8$. For any $\eta >0$ and $\beta \geq \mathbf{e}%
_{i}$, by Lemma 6 in \cite{G12} we have
\begin{equation}
\begin{split}
I_1 &\lesssim \norm{\delta _{\beta }^{\mathbf{e}_{i}}\partial
_{\beta -\mathbf{e}_{i}}^{\alpha }\{\FI-\FP\} f }_{\sigma ,w(\alpha ,\beta -%
\mathbf{e}_{i})}\norm{\partial _{\beta -\mathbf{e}_{i}}^{\alpha +\mathbf{e}%
_{i}}\{\FI-\FP\} f }_{\sigma ,w(\alpha +\mathbf{e}_{i},\beta -\mathbf{e}%
_{i})}    \\
&\leq \eta \mathcal{D}_{m;l,q}(f )+C_{\eta }\norm{\delta _{\beta }^{\mathbf{e}_{i}}\partial
_{\beta -\mathbf{e}_{i}}^{\alpha }\{\FI-\FP\} f }_{\sigma ,w(\alpha ,\beta -%
\mathbf{e}_{i})}^2.
\end{split}
\end{equation}
Clearly,
\begin{equation}
I_2\le C_{l,m} \left[\norm{\pa_t\phi}_{\infty }+\norm{\nabla
_{x}\phi }_{\infty }\right]\norm{\partial _{\beta }^{\alpha }\{\FI-\FP\}  f}_{2,w}^{2}
.
\end{equation}
By Lemma \ref{linear c weight}, we have that for
any $\eta >0$,
\begin{equation}
\begin{split}
I_3 &\lesssim \sqrt{M}  \norm{\partial _{\beta }^{\alpha }\{\FI-\FP\} f }_{\sigma ,w}^{2}+C_{l,m}\sqrt{M}\norm{\partial^\al \{\FI-\FP\} f}_\sigma^2
\\&\quad
+\eta \mathcal{D}_{m;l,q}(f )+C_{l,m,\eta }\sum_{\beta'<\beta}\norm{\partial _{\beta' }^{\alpha }\{\FI-\FP\} f }_{\sigma ,w(\alpha ,\beta'
)}^{2}.
\end{split}
\end{equation}

For the term $I_4$, we can move all the $v$ derivatives $\partial_\beta$ out of $\partial _{\beta }^{\alpha }\{\FI-\FP\} f$ to the remaining factor in $v$; using the exponential decay of $\mu$, we obtain
\begin{equation}
I_4 \le C_{l,m}   \norm{\partial^\alpha \nabla_x\phi}_2\norm{\mu^\delta \partial^{\alpha }\{\FI-\FP\} f}_2
\le C_{l,m}\left( \norm{ \nabla_x\phi}_{2}^2+\norm{\partial^{\alpha }\{\FI-\FP\} f}_{\sigma}^2\right).
\end{equation}

To estimate the term $I_5$, applying the estimate \eqref{ga es} of Lemma \ref{nonlinear c}, we obtain
\begin{equation} \label{I51}
I_{5}\le
\sum_{\alpha_1 \le \alpha\atop \bar{\beta}\le\beta_1 \le \beta}C_{\al}^{\al_1}C_{\beta}^{\beta_1}C_{(\beta_1,\bar{\beta})}\int
\norms{\partial _{\beta -\beta _1}^{\alpha -\alpha _1}   f}_{\sigma ,w} \norms{\mu^\delta\partial _{\bar{\beta}}^{\alpha _1}f}_{2}
\left\{\norms{\partial _\beta^\alpha \{\FI-\FP\} f }_{\sigma ,w}
+l
\norms{\partial _\beta^\alpha \{\FI-\FP\} f }_{2 ,\frac{w}{\langle v\rangle^\frac{3}{2}}}\right\}.
\end{equation}

First notice that for all $|\al|+|\be|\le 2$, if $(\alpha_1,\bar{\beta})=0$ we can always take $L^2-L^\infty-L^2$ in \eqref{I51}, and we use the split $f=\FP f+\{\FI-\FP\} f$ in the factor $\norms{\partial _{\beta -\beta _1}^{\alpha -\alpha _1}   f}_{\sigma ,w}$, to have an upper bound of
\begin{equation}
\begin{split}
 & C \left(\norm{\partial^{\alpha}{\bf P} f }_{2 ,w}+
\norm{\partial _{\beta -\beta _1}^{\alpha}\{\FI-\FP\} f }_{\sigma ,w}\right) \sum_{1\le |\gamma|\le 2}\norm{ \mu^\delta \pa^\ga f}_{2}
\norm{\partial _\beta^\alpha \{\FI-\FP\} f }_{\sigma ,w}
\\& \quad+C_{l}
\norm{\partial^{\alpha}{\bf P} f }_{2 ,w}\sum_{1\le |\gamma|\le 2}\norm{ \mu^\delta \pa^\ga f}_{2}
\norm{\partial _\beta^\alpha \{\FI-\FP\} f }_{\sigma ,w}
\\& \quad+C_{l}
\norm{\partial _{\beta -\beta _1}^{\alpha}\{\FI-\FP\} f }_{\sigma ,w}\sum_{1\le |\gamma|\le 2}\norm{ \mu^\delta \pa^\ga f}_{2}
\norm{\partial _\beta^\alpha \{\FI-\FP\} f }_{2 ,w}
\\& \quad\lesssim   \sqrt{\mathcal{E}_{2;l,q}(f )}\sqrt{\widetilde{\mathcal{D}}_{2;2,0}(f )}\sqrt{\widetilde{\mathcal{D}}_{2;l,q}(f )}
+\sqrt{\widetilde{\mathcal{D}}_{2;l,q}(f )}\sqrt{\mathcal{E}_{2;2,0}(f )}\sqrt{\widetilde{\mathcal{D}}_{m;l,q}(f )}
\\& \qquad+C_{l}  \sqrt{\mathcal{E}_{2;l,q}(f )}\sqrt{\widetilde{\mathcal{D}}_{2;2,0}(f )}\sqrt{\widetilde{\mathcal{D}}_{2;l,q}(f )}
+C_{l} \sqrt{\widetilde{\mathcal{D}}_{2;l,q}(f )}\sqrt{\widetilde{\mathcal{D}}_{2;2,0}(f )}\sqrt{\mathcal{E}_{2;l,q}(f )}
\\& \quad\lesssim \left(\sqrt{\mathcal{E}_{2;2,0}(f )}+\eta\right)\widetilde{\mathcal{D}}%
_{2;l,q}(f )+C_{l,\eta} {\widetilde{\mathcal{D}}_{2;2,0}(f )}{\mathcal{E}_{2;l,q}(f )} .
\end{split}
\end{equation}
Here we have used the fact $w(\alpha,\beta)\le w(\alpha,\beta -\beta _1)$. Notice carefully that here we adjusted the energy and dissipation components by making full use of the advantage that the terms $\norms{\mu^\delta\partial _{\bar{\beta}}^{\alpha _1}g_1}_{2}$ and $\norms{\partial _\beta^\alpha g_3}_{2 ,\frac{w}{\langle v\rangle^{\frac {3}{ 2}}}}$ can be included in either the energy or the dissipation when they are hit by the spatial derivatives. Note that we have concluded the case $|\al|+|\beta|=0$.

When $|\al|+|\beta|=1$, the remaing case is of $(\alpha-\alpha_1,\beta-\beta_1)=0$,
and we take $L^6-L^3-L^2$ in \eqref{I51} to have an upper bound of
\begin{equation}
\begin{split}
 &C \sum_{|\gamma|=1}\norm{\pa^\ga f }_{\sigma ,w}\sum_{|\gamma|\le 1}\norm{
\mu^\delta \pa^\ga  \partial _{\bar{\beta}}^{\alpha } f }_{2}
\norm{\partial _\beta^\alpha \{\FI-\FP\} f }_{\sigma ,w}
\\&\quad+C_l\sum_{|\gamma|=1}\norm{\pa^\ga f }_{\sigma ,w}\sum_{|\gamma|\le 1}\norm{
\mu^\delta \pa^\ga  \partial _{\bar{\beta}}^{\alpha } f }_{2}
\norm{\partial _\beta^\alpha \{\FI-\FP\} f }_{2 ,w}
\\&\quad\lesssim \sqrt{\widetilde{\mathcal{D}}%
_{2;l,q}(f )}\sqrt{\mathcal{E}_{2;2,0}(f )}\sqrt{\widetilde{\mathcal{D}}%
_{2;l,q}(f )}+C_l\sqrt{\widetilde{\mathcal{D}}%
_{2;l,q}(f )}\sqrt{\widetilde{\mathcal{D}}%
_{2;2,0}(f )}\sqrt{\mathcal{E}_{2;l,q}(f )}
\\&\quad\lesssim \left(\sqrt{\mathcal{E}_{2;2,0}(f )}+\eta\right)\widetilde{\mathcal{D}}%
_{2;l,q}(f )+C_{l,\eta} {\widetilde{\mathcal{D}}_{2;2,0}(f )}{\mathcal{E}_{2;l,q}(f )} .
\end{split}
\end{equation}
Here we have used the fact $w(\alpha,\beta)= w(\gamma,0)$ for $|\gamma|= 1$.  This concludes the case $|\al|+|\beta|=1$.

When $|\al|+|\be|= 2$,  if $(\alpha-\alpha_1,\beta-\beta_1)=0$
we take $L^\infty-L^2-L^2$  in \eqref{I51} to have an upper bound of
\begin{equation}
\begin{split}
 &C \sum_{1\le |\gamma|\le 2}\norm{\pa^\ga f }_{\sigma ,w}\norm{
\mu^\delta  \partial _{\bar{\beta}}^{\alpha } f }_{2}
\norm{\partial _\beta^\alpha \{\FI-\FP\} f }_{\sigma ,w}
\\&\quad+C_l\sum_{1\le |\gamma|\le 2}\norm{\pa^\ga f }_{\sigma ,w} \norm{
\mu^\delta \partial _{\bar{\beta}}^{\alpha } f }_{2}
\norm{\partial _\beta^\alpha \{\FI-\FP\} f }_{2 ,w}
\\&\quad\lesssim \sqrt{\widetilde{\mathcal{D}}%
_{2;l,q}(f )}\sqrt{\mathcal{E}_{2;2,0}(f )}\sqrt{\widetilde{\mathcal{D}}%
_{2;l,q}(f )}+C_l\sqrt{\widetilde{\mathcal{D}}%
_{2;l,q}(f )}\sqrt{\widetilde{\mathcal{D}}%
_{2;2,0}(f )}\sqrt{\mathcal{E}_{2;l,q}(f )}
\\&\quad\lesssim \left(\sqrt{\mathcal{E}_{2;2,0}(f )}+\eta\right)\widetilde{\mathcal{D}}%
_{2;l,q}(f )+C_{l,\eta} {\widetilde{\mathcal{D}}_{2;2,0}(f )}{\mathcal{E}_{2;l,q}(f )} .
\end{split}
\end{equation}
Here we have used the fact $w(\alpha,\beta)\le w(\gamma,0)$ for $1\le|\gamma|\le 2$.
If $(\alpha_1,\bar{\beta})\neq0$ and $(\alpha-\alpha_1,\beta-\beta_1)\neq 0$, then $|\alpha_1|+|\bar{\beta}|=1$ and
$|\alpha-\alpha_1|+|\beta-\beta_1|=1$ since $|\al|+|\be|= 2$. Hence we take $L^6-L^3-L^2$ in \eqref{I51} to have an upper bound of
\begin{equation}
\begin{split}
 &C
\sum_{|\gamma|=1}\norm{\pa^\ga\partial _{\beta -\beta _1}^{\alpha-\alpha_1} f }_{\sigma ,w}\sum_{|\gamma|\le 1}\norm{
\mu^\delta   \pa^\ga \partial _{\bar{\beta}}^{\alpha _1} f }_{2}
\norm{\partial _\beta^\alpha \{\FI-\FP\} f }_{\sigma ,w}
\\&\quad +C_l
\sum_{|\gamma|=1}\norm{\pa^\ga\partial _{\beta -\beta _1}^{\alpha-\alpha_1} f }_{\sigma ,w}\sum_{|\gamma|\le 1}\norm{
\mu^\delta   \pa^\ga \partial _{\bar{\beta}}^{\alpha _1} f }_{2}
\norm{\partial _\beta^\alpha \{\FI-\FP\} f }_{2 ,w}
\\&\quad\lesssim \sqrt{\widetilde{\mathcal{D}}%
_{2;l,q}(f )}\sqrt{\mathcal{E}_{2;2,0}(f )}\sqrt{\widetilde{\mathcal{D}}%
_{2;l,q}(f )}+C_l\sqrt{\widetilde{\mathcal{D}}%
_{2;l,q}(f )}\sqrt{\widetilde{\mathcal{D}}%
_{2;2,0}(f )}\sqrt{\mathcal{E}_{2;l,q}(f )}
\\&\quad\lesssim \left(\sqrt{\mathcal{E}_{2;2,0}(f )}+\eta\right)\widetilde{\mathcal{D}}%
_{2;l,q}(f )+C_{l,\eta} {\widetilde{\mathcal{D}}_{2;2,0}(f )}{\mathcal{E}_{2;l,q}(f )} .
\end{split}
\end{equation}
Here we have used the fact $w(\alpha,\beta)\le w(\gamma+\alpha-\alpha_1,\beta-\beta_1)$ for $|\gamma|\le 1$. This completes the case $|\al|+|\beta|=2$.

Now when $|\alpha |+|\beta |=m\geq 3$, we shall separate four cases. The first case is either $(\alpha_1,\bar{\beta})=(\alpha,\beta)$ or $(\alpha-\alpha_1,\beta-\beta_1)=(\alpha,\beta)$. Indeed we now have $C_{\al}^{\al_1}C_{\beta}^{\beta_1}C_{(\beta_1,\bar{\beta})}=C$. We note that $m\ge 3$ so that we can take $L^\infty$ of the term without derivatives and $L^2$ of the other two terms in \eqref{I51}, also when $(\alpha-\alpha_1,\beta-\beta_1)=(\alpha,\beta)$ we use the split $f=\FP f+\{\FI-\FP\}f$ in the factor $\norms{\partial _{\beta}^{\alpha}   f}_{\sigma ,w}$, to have an upper bound of
\begin{equation}
\begin{split}
 &C_{l}\sum_{1\le |\gamma|\le 2}\norm{\pa^\ga f }_{\sigma ,w}\norm{
\mu^\delta \partial _{{\beta}}^{\alpha} f }_{2}
\norm{\partial _\beta^\alpha \{\FI-\FP\} f }_{\sigma ,w}
\\&\quad+ C_m\norm{\partial^{\alpha}  {\bf P} f }_{2 ,w} \sum_{1\le |\gamma|\le 2}\norm{
\mu^\delta \pa^\ga f }_{2}
\norm{\partial _{{\beta}}^{\alpha} \{\FI-\FP\} f }_{\sigma ,w}
\\&\quad+ C\norm{\partial _{{\beta}}^{\alpha} \{\FI-\FP\}  f }_{\sigma ,w} \sum_{1\le |\gamma|\le 2}\norm{
\mu^\delta \pa^\ga f }_{2}
\norm{\partial _{{\beta}}^{\alpha} \{\FI-\FP\} f }_{\sigma ,w}
\\& \quad+C_{l,m}
\norm{\partial^{\alpha}{\bf P} f }_{2 ,w}\sum_{1\le |\gamma|\le 2}\norm{ \mu^\delta \pa^\ga f}_{2}
\norm{\partial _\beta^\alpha \{\FI-\FP\} f }_{\sigma ,w}
\\& \quad+C_{l,m}
\norm{\partial _{\beta}^{\alpha}\{\FI-\FP\} f }_{\sigma ,w}\sum_{1\le |\gamma|\le 2}\norm{ \mu^\delta \pa^\ga f}_{2}
\norm{\partial _\beta^\alpha \{\FI-\FP\} f }_{2 ,w}
\\&\qquad\lesssim C_{l,m}\sqrt{\widetilde{\mathcal{D}}_{2;l,q}(f ) }\sqrt{\mathcal{E}_{m;l,q}(f )}\sqrt{\widetilde{\mathcal{D}}_{m;l,q}(f )}+C_{l,m}\sqrt{\mathcal{E}_{m;l,q}(f )}\sqrt{\widetilde{\mathcal{D}}_{2;2,0}(f ) }\sqrt{\widetilde{\mathcal{D}}_{m;l,q}(f )}
\\&\qquad\quad  +  \sqrt{\widetilde{\mathcal{D}}_{m;l,q}(f )}\sqrt{ {\mathcal{E}}_{2;2,0}(f ) }\sqrt{ \widetilde{{\mathcal{D}}}_{m;l,q}(f )}
\\&\qquad\lesssim \left(\sqrt{ {\mathcal{E}}_{2;2,0}(f ) }+\eta\right) \widetilde{\mathcal{D}}%
_{m;l,q}(f )+C_{l,m,\eta} {\widetilde{\mathcal{D}}_{2;l,q}(f )}{\mathcal{E}_{m;l,q}(f )} .
\end{split}
\end{equation}

We next consider the case of $|\alpha_1|+|\bar{\beta}|=m-1$ and  $|\alpha-\alpha_1|+|\beta-\beta_1|=1$, and we take $L^6-L^3-L^2$  in \eqref{I51} to have an upper bound of
\begin{equation}
\begin{split}
&C_{l,m}\sum_{|\gamma|=1}\norm{\pa^\ga \partial _{{\beta}-\beta_1}^{\alpha-\alpha_1} f }_{\sigma ,w}
\sum_{|\gamma|\le 1}\norm{
\mu^\delta \pa^\ga \partial _{\bar{\beta}}^{\alpha_1} f }_{2}
\norm{\partial _\beta^\alpha \{\FI-\FP\} f }_{\sigma ,w}
\\&\quad\le   C_{l,m}\sqrt{\widetilde{\mathcal{D}}_{2;l,q}(f )}\sqrt{ {\mathcal{E}_{m;l,q}(f )}}\sqrt{{{\widetilde{\mathcal{D}}_{m;l,q}(f )}}}
\\&\quad\le   \eta \widetilde{\mathcal{D}}%
_{m;l,q}(f )+C_\eta \widetilde{\mathcal{D}}%
_{m-1;l,q}(f )\mathcal{E}_{m;l,q}(f ).
\end{split}
\end{equation}

We now consider the case of $|\alpha_1|+|\bar{\beta}|=1$ and $|\alpha-\alpha_1|+|\beta-\beta_1|=m-1$, and we shall use the split $f=\FP f+\{\FI-\FP\} f$ in the factor $\norms{\partial _{\beta -\beta _1}^{\alpha -\alpha _1}   f}_{\sigma ,w}$. For the hydrodynamic part of \eqref{I51} we take $L^3-L^6-L^2$; while for the microscopic part of \eqref{I51} we take $L^2-L^\infty-L^2$ respectively to have an upper bound of
\begin{equation}
\begin{split}
&  C_{l,m} \sum_{|\gamma|\le 1} \norm{ \pa^\ga \partial^{\alpha-\alpha_1}{\bf P} f }_{2 ,w}
\sum_{ |\gamma|= 1} \norm{
\mu^\delta \pa^\ga \partial _{\bar{\beta}}^{\alpha_1} f }_{2}
\norm{\partial _\beta^\alpha \{\FI-\FP\} f }_{\sigma ,w}
\\&\quad+C_{l,m}  \norm{  \partial _{{\beta}-\beta_1}^{\alpha-\alpha_1}\{\FI-\FP\}  f }_{\sigma ,w}
\sum_{1\le|\gamma|\le 2} \norm{
\mu^\delta \pa^\ga \partial _{\bar{\beta}}^{\alpha_1} f }_{2}
\norm{\partial _\beta^\alpha \{\FI-\FP\} f }_{\sigma ,w}
\\&\quad\le   C_{l,m}\sqrt{{\mathcal{E}_{m;l,q}(f )}}\sqrt{ \widetilde{\mathcal{D}}_{2;2,0}(f )}\sqrt{{{\widetilde{\mathcal{D}}_{m;l,q}(f )}}}+C_{l,m}
\sqrt{{{\widetilde{\mathcal{D}}_{m-1;l,q}(f )}}}\sqrt{ {\mathcal{E}_{3;3,0}(f )}}\sqrt{{{\widetilde{\mathcal{D}}_{m;l,q}(f )}}}
\\&\quad\le  \eta \widetilde{\mathcal{D}}%
_{m;l,q}(f )+C_{l,m,\eta} \widetilde{\mathcal{D}}%
_{m-1;l,q}(f )\mathcal{E}_{m;l,q}(f ).
\end{split}
\end{equation}

The remaining cases are of $2\le |\alpha_1|+|\bar{\beta}|\le m-2$ and $2\le |\alpha-\alpha_1|+|\beta-\beta_1|\le m-2$ (surely, now $m\ge 4$), and we take $L^6-L^3-L^2$ in \eqref{I51} to bound it by
\begin{equation}
\begin{split}
&C_{l,m}\sum_{|\ga|=1}\norm{\pa^\ga \partial _{{\beta}-\beta_1}^{\alpha-\alpha_1} f }_{\sigma ,w}\sum_{|\ga|\le 1} \norm{
\mu^\delta \pa^\ga \partial _{\bar{\beta}}^{\alpha_1}f }_{2}
\norm{\partial _\beta^\alpha \{\FI-\FP\} f }_{\sigma ,w}
\\& \quad\le  C_{l,m}\sqrt{{\widetilde{\mathcal{D}}_{m-1;l,q}(f )}}\sqrt{{\mathcal{E}_{m-1;m-1,0}(f )}}\sqrt{{\widetilde{\mathcal{D}}_{m;l,q}(f )}}
\\& \quad\le    \eta{{\widetilde{\mathcal{D}}_{m;l,q}(f )}}+C_{l,m,\eta} {{\widetilde{\mathcal{D}}_{m-1;l,q}(f )}}{{\mathcal{E}_{m-1;m-1,0}(f )}} .
\end{split}
\end{equation}

Notice that the term $I_6$ only occurs when $|\alpha|+|\beta|\ge 1$. Since $\{{\bf I-P}\}f$ is always part of our dissipation rate, so we can use the argument of Lemma 9 in \cite{G12} to obtain that for $|\alpha|+|\beta|\le m$,
\begin{equation}
\begin{split}
I_6&\le C_m\norm{\partial _{\beta }^{\alpha }\{\FI-\FP\} f }_{\sigma
,w}\left(\norm{\partial^\alpha \nabla_x\phi}_{H^1}\norm{ \norms{\{\FI-\FP\} f }_{\sigma ,\frac{w(0,0)}{\langle
v\rangle ^{2}}}}_{H^{\frac{3}{4}}}+ \norm{  \nabla_x^2\phi}_{H^2\cap H^{m-1}}^2 \widetilde{\mathcal{D}}%
_{m-1;l,q}(f)\right).
\end{split}
\end{equation}
In particular, if $|\alpha|+|\beta|\le 2$,
\begin{equation}
I_6\lesssim\sqrt{\mathcal{E}%
_{2;2,0}(f)}\widetilde{\mathcal{D}}%
_{2;l,q}(f);
\end{equation}
if $|\alpha|+|\beta|=m\ge 3$, since $|\al|\le m-1$,
\begin{equation}
I_6\le \eta \widetilde{\mathcal{D}}%
_{m;l,q}(f)+C_{m,\eta } \mathcal{E}%
_{m-1;m-1,0}(f)\widetilde{\mathcal{D}}%
_{m-1;l,q}(f) .
\end{equation}

For the term $I_7$, we make use of the exponential decay in $v$ of the hydrodynamic part to get
\begin{equation}\label{I70}
I_7\le C_{l,m}\sum_{\alpha _{1}\le\alpha } \int\norms{\partial _{\beta }^{\alpha }\{\FI-\FP\} f }_\sigma \norms{ \partial
^{\alpha -\alpha _{1}}\nabla _{x}\phi}\norms{
 \partial^{\alpha _{1}}  f}_\sigma\,dx.
\end{equation}
If $|\al_1|\geq 1$, we take $L^2-L^\infty-L^2$ in \eqref{I70} to have an upper bound of
\begin{equation}
\begin{split}
 & C_{l,m}
\norm{\partial _{\beta }^{\alpha }\{\FI-\FP\}f}_{\si}\sum_{1\le|\gamma|\le 2}\norm{\pa^\ga\partial
^{\alpha -\alpha _{1}}\nabla _{x}\phi }_{ 2 }\norm{\partial^{\alpha_1}f}_{\sigma}
\\&\quad\le C_{l,m}
\sqrt{{{\widetilde{\mathcal{D}}_{m;m,0}(f )}}}\sqrt{{{\mathcal{E}_{m-1;m-1,0}(f )}}}\sqrt{{{\widetilde{\mathcal{D}}_{m-1;m-1,0}(f )}}}
\\&\quad\le  \eta\widetilde{\mathcal{D}}%
_{m;m,0}(f )+C_{l,m,\eta} \mathcal{E}_{m-1;m-1,0}(f )\widetilde{\mathcal{D}}%
_{m-1;m-1,0}(f ).
\end{split}
\end{equation}
Here we have use the fact that $1\le |\alpha_1|\le |\alpha|\le m-1$. If $|\al_1|=0$, we take $L^2-L^3-L^6$ in \eqref{I70} to have an upper bound of
\begin{equation}
\begin{split}
 & C_{l,m}
\norm{\partial _{\beta }^{\alpha }\{\FI-\FP\}f}_{\si}\sum_{|\gamma|\le 1}\norm{\pa^\ga \partial
^{\alpha}\nabla _{x}\phi }_{2}\sum_{|\gamma|= 1}\norm{\pa^\ga f}_{\sigma}
\\&\quad\le C_{l,m}
\sqrt{{{\widetilde{\mathcal{D}}_{m;m,0}(f )}}}\sqrt{{{\mathcal{E}_{m-1;m-1,0}(f )}}}\sqrt{{{\widetilde{\mathcal{D}}_{2;2,0}(f )}}}
\\&\quad\le  \eta\widetilde{\mathcal{D}}%
_{m;m,0}(f )+C_{l,m,\eta} \mathcal{E}_{m-1;m-1,0}(f )\widetilde{\mathcal{D}}%
_{2;2,0}(f ).
\end{split}
\end{equation}

Finally, for $I_8$, we easily have
\begin{equation}
\begin{split}
I_8&\le C_{l,m}
\norm{\partial _{\beta }^{\alpha }\{\FI-\FP\}f}_{\si} \norm{\nabla_x^{|\alpha|+1}f}_{\sigma}
\\&\le  \eta\norm{\partial _{\beta }^{\alpha }\{\FI-\FP\}f}_{\si}^2+C_{l,m,\eta}\norm{\nabla_x^{|\alpha|+1}f}_{\sigma}^2.
\end{split}
\end{equation}

Consequently, collecting the estimates
for $I_1\sim I_8$, we thus conclude our lemma.\end{proof}

We now present the
\begin{proof}[Proof of Proposition \ref{energy estimate}]
For $m\ge 2$, we define
\begin{equation}
\begin{split}
\widetilde{\mathcal{E}}_{m,l,q}(f)&=\int \frac{|f|^2}{2}+\int |\na_x \phi|^2+\sum_{1\le k\le m}\left(\int \sum_{\pm}e^{\pm2\phi}\frac{|\na^k f_{\pm}|^2}{2}+\int |\na^k \na_x \phi|^2\right)
\\&\quad+\sum_{0\le k\le m-1}\varepsilon_k  G^k(t)+\sum_{1\le |\al|\le m}\varepsilon_\alpha\int \sum_{\pm }\frac{e^{\pm 2(q+1)\phi }w^2|\partial ^{\alpha
}f_{\pm }|^{2}}{2}
\\&\quad+\sum_{|\alpha|+|\beta|\le m\atop |\al|\leq m-1}\varepsilon_{\alpha,\beta}\int \sum_{\pm }\frac{e^{\pm 2(q+1)\phi }w^2|\partial_{\beta} ^{\alpha
}\{\FI-\FP\}f_{\pm }|^{2}}{2}.
\end{split}
\end{equation}
By taking $\varepsilon_k$ sufficiently small, $\widetilde{\mathcal{E}}_{m,l,q}(f)$ is equivalent to our energy $\mathcal{E}_{m,l,q}(f)$ of \eqref{energy} in the sense that there exists $C_{l,m}$ such that
\begin{equation}
 C_{l,m}^{-1} \widetilde{\mathcal{E}}_{m,l,q}(f) \le \mathcal{E}_{m,l,q}(f)\le   C_{l,m} \widetilde{\mathcal{E}}_{m,l,q}(f).
\end{equation}
By further taking $\varepsilon_k, \varepsilon_\alpha, \varepsilon_{\alpha,\beta},  \eta$ sufficiently small orderly and choose $M$ sufficiently small (independent of $l$ and $m$!), we deduce Proposition \ref{energy estimate} from Lemma \ref{energy lemma 1}--\ref{energy I-P}.
\end{proof}

\section{Time decay and global solution}\label{global}

In this section, we will derive a further energy estimate which allows us to extract the strong decay rate of $\phi$. Then we can close the energy estimates in Proposition \ref{energy estimate} and thus complete the proof of Theorem \ref{main theorem}. We recall the notations $f_1=f_++f_-$ and $f_2=f_+-f_-$ for the solution $f$ to the Vlasov-Poisson-Landau system \eqref{VPL_per}. The key point is to consider the evolution of $f_2$ and $\phi$ separating from \eqref{f_1f_2 equation}:
\begin{equation}\label{f_2 equation}
\begin{split}
 &\partial_tf_2 + v\cdot\nabla_xf_2+4\nabla_x\phi\cdot v\sqrt{\mu} + \mathcal{L}_2 f_2=\Gamma_\ast(f_1,f_2)+\nabla_x\phi\cdot \left(\nabla_vf_1-vf_1\right),
\\ &-\Delta_x\phi=\int_{\r3}f_2\sqrt{\mu}\,dv.
\end{split}
\end{equation}
\begin{lemma}\label{difference energy}
Let $f_{0}\in \testF$ and assume $f$ is the solution constructed in Theorem \ref{local solution} with
$\mathcal{{E}}_{2;2,0}(f)\leq M$. There exists a function $\mathcal{E}_0^1({f_2})(t)$ with
\begin{equation}\label{f201}
\mathcal{E}_0^1({f_2})(t)\sim \sum_{k=0,1} \norm{\nabla^k f_2 }_2^2+\norm{\nabla_x\phi}_2^2
\end{equation}
such that
\begin{equation}\label{decay energy}
\frac{d}{dt}\mathcal{E}_0^1({f_2})+\sum_{k=0,1} \norm{\nabla^k f_2 }_\sigma^2+\norm{\nabla_x\phi}_2^2 \le 0.
\end{equation}
\end{lemma}
\begin{proof}
The standard $\na^k$ energy estimates on \eqref{f_2 equation} yields for $k=0,1$,
\begin{equation}\label{4100}
\begin{split}
& \frac{d}{dt}\left(\frac{\norm{\nabla^k f_2 }_2^2}{2}+2\norm{\nabla^k \nabla_x\phi}_2^2\right) +\lambda\norm{\nabla^k\{{\bf I-P_2}\}f_2}_\sigma^2 \\&\quad\lesssim
\sum_{0\le j\le k} \left|\left(  \Gamma_\ast(\nabla^{j}f_1 ,\nabla^{k-j} f_2)+ \nabla^{ j}\nabla_x\phi\cdot \nabla^{k-j}(\nabla_v f_1-v f_1),\nabla^k f_2\right )\right|
:=I_1+I_2.
\end{split}
\end{equation}

Our goal is to conclude an energy estimate for $f_2$ and $\phi$ solely that excludes $f_1$. So when estimating $I_1$ and $I_2$ we will bound the $f_1$-related factors by the energy rather than the dissipation. More precisely, applying the estimate \eqref{ga es 0} of Lemma \ref{nonlinear c}, we have
\begin{equation}\label{4101}
I_1\lesssim \sum_{0\le j\le k}\int\norms{\nabla^j f_1}_2\norms{\nabla^{k-j} f_2}_\sigma \norms{\nabla^k f_2}_\sigma\,dx
\lesssim \sqrt{\mathcal{{E}}_{2;2,0}(f)}\norm{\nabla^k f_2}_\sigma^2.
\end{equation}
Here in the $x$ integration of \eqref{4101}, when $j=0$ we have taken $L^\infty-L^2-L^2$; and if $k=1$ and $j=1$ we have taken $L^3-L^6-L^2$ respectively. While we use an integration by parts in $v$ and recall \eqref{sigma norm =} to have
\begin{equation}\label{4102}
\begin{split}
I_2& =\sum_{0\le j\le k}\left|\left(\nabla^{k-j} f_1 \nabla^{j} \nabla_x\phi,  \nabla^{k}(\nabla_vf_2+v  f_2)\right)\right|
\\&\lesssim \sum_{0\le j\le k} \int \norms{\langle v\rangle^{3/2} \nabla^{k-j}f_1}_2\norms{\nabla^{ j}\nabla_x\phi}\norms{\nabla^{k}f_2}_\sigma\,dx
\\&\lesssim \sqrt{\mathcal{{E}}_{2;2,0}(f)} \norm{\nabla_x\phi}_{H^2}\norm{\nabla^{k}f_2}_\sigma
\\&\lesssim {\sqrt{\mathcal{{E}}_{2;2,0}(f)} } \left(\norm{\nabla_x\phi}_{H^2}^2+\norm{\nabla^{k}f_2}_\sigma^2\right).
\end{split}
\end{equation}
Here when $j=0$ we have taken $L^2-L^\infty-L^2$; and if $k=1$ and $j=1$ we have taken $L^4-L^4-L^2$ in the $x$-integration of \eqref{4102} respectively. So, we may conclude from \eqref{4100} that
\begin{equation}\label{decay en 1}
\begin{split}
& \frac{d}{dt}\left(\frac{\norm{\nabla^k f_2 }_2^2}{2}+2\norm{\nabla^k \nabla_x\phi}_2^2\right) +\lambda\norm{\nabla^k\{{\bf I-P_2}\}f_2}_\sigma^2 \\&\quad\lesssim \sqrt{\mathcal{{E}}_{2;2,0}(f)}\left(\sum_{k=0,1}\norm{\nabla^k f_2}_\sigma^2+\norm{\nabla_x\phi}_2^2\right).
\end{split}
\end{equation}
Here we have used the elliptic estimate on the Poisson equation to have
\begin{equation}
\norm{\nabla_x\phi}_{H^2}^2
\lesssim \norm{\nabla_x\phi}_2^2 +\norm{\nabla_x f_2}_\sigma^2.
\end{equation}

On the other hand, similarly as Lemma \ref{inter lemma}, there exists a function
$G_{f_2}^0(t)$ with
\begin{equation}\label{G1f_2}
G_{f_2}^0(t)\lesssim \norm{f_2}_2^2+\norm{\nabla_x f_2}_2^2+\norm{\na_x\phi}_2^2
\end{equation}
such that
\begin{equation}\label{4103}
\begin{split}
& \frac{d}{dt}G_{f_2}^0(t)+ \norm{ {\bf P_2} f_2}_{2}^2 +\norm{\nabla_x  {\bf P_2} f_2}_{2}^2+\norm{ \nabla_x\phi}_{2}^2
\\&\quad\lesssim \norm{  \{{\bf I- P_2}\} f_2}_{\sigma}^2+\norm{\nabla_x \{{\bf I- P_2}\} f_2}_{\sigma}^2+\norm{ N_\parallel}_{2}^2.
\end{split}
\end{equation}
The key point in \eqref{4103} is that we can include the term $\norm{ {\bf P_2} f_2}_{2}^2$, which follows by
\begin{equation}
\norm{ {\bf P_2} f_2}_{2}^2\lesssim\norm{ \rho}_{2}^2= \norm{ \Delta\phi}_{2}^2 \lesssim \norm{ \nabla_x\phi}_{2}^2+\norm{\nabla_x  {\bf P_2} f_2}_{2}^2.
\end{equation}
The last term in \eqref{4103} is bounded by
\begin{equation}
\begin{split}
\norm{ N_\parallel}_{2}^2&\equiv\norm{\langle \Gamma_\ast(f_1,f_2)+\nabla_x\phi\cdot \left(\nabla_vf_1-vf_1\right), v\sqrt{\mu}\rangle}_2^2
\\&\lesssim {\mathcal{{E}}_{2;2,0}(f)}\left(\norm{f_2}_\sigma^2+\norm{\nabla_x\phi}_2^2\right).
\end{split}
\end{equation}
Hence, \eqref{4103} implies that
\begin{equation}\label{decay en 2}
\begin{split}
& \frac{d}{dt}G_{f_2}^0(t)+\norm{ {\bf P_2} f_2}_{ 2}^2 +\norm{\nabla_x  {\bf P_2} f_2}_{ 2}^2+\norm{ \nabla_x\phi}_{ 2}^2
\\& \quad\lesssim  \norm{  \{{\bf I- P_2}\} f_2}_{\sigma}^2+\norm{\nabla_x \{{\bf I- P_2}\} f_2}_{\sigma}^2+\sqrt{\mathcal{{E}}_{2;2,0}(f)}\left(\norm{f_2}_\sigma^2+\norm{\nabla_x\phi}_2^2\right).
\end{split}
\end{equation}

Now we define
\begin{equation}
\mathcal{E}_0^1({f_2})(t):=\sum_{k=0,1}\left(\frac{\norm{\nabla^k f_2 }_2^2}{2}+2\norm{\nabla^k \nabla_x\phi}_2^2\right) -\varepsilon G_{f_2}^0(t).
\end{equation}
By \eqref{G1f_2}, for $\varepsilon$ sufficiently small, we deduce that $\mathcal{E}_0^1({f_2})(t)$ satisfies \eqref{f201} and that \eqref{decay energy} follows from \eqref{decay en 1} and \eqref{decay en 2} since $\mathcal{{E}}_{2;2,0}(f)\le M$ is small.
\end{proof}

We now establish the crucial strong decay rate of $\phi$ in the following proposition.
\begin{proposition}\label{decay pro}
Let $f_{0}\in \testF$ and assume $f$ is the solution constructed in Theorem \ref{local solution} with
$\mathcal{{E}}_{2;2,0}(f)\leq M$. Then there exists $C_l>0$ such that
\begin{equation}\label{polynomial decay1}
\begin{split}
&\norm{\partial _{t}\phi (t)}_{\infty }+\norm{\nabla _{x}\phi (t)}_{\infty}+\norm{\nabla _{x}\phi (t)}_{2}+\sum_{k=0,1} \norm{\na^k f_2(t)}_2
\\&\quad \leq
 C_{l}(1+t)^{-2l+2} \sup_{0\le \tau\le T}\sqrt{\mathcal{ {E}}_{2;l,0}(f(\tau))},
 \end{split}
\end{equation}
and
\begin{equation}\label{exponential decay1}
\begin{split}
&\norm{\partial _{t}\phi (t)}_{\infty }+\norm{\nabla _{x}\phi (t)}_{\infty}+\norm{\nabla _{x}\phi (t)}_{2}+\sum_{k=0,1} \norm{\na^k f_2(t)}_2
\\&\quad \le
 C_{l}e^{-C_lt^{2/3}} \sup_{0\le \tau\le T}\sqrt{\mathcal{ {E}}_{2;l,q}(f(\tau))}\text{ for }0<q\ll 1.
  \end{split}
\end{equation}
\end{proposition}
\begin{proof}
Note that \eqref{decay energy} is essentially same as (150) in \cite{G12}, then we can get the same decay rate as \cite{G12}. Indeed, taking $\ell=1$ and $m=2$ in \eqref{aaa} and \eqref{bbb}, we obtain
\begin{equation}\label{polynomial decay2}
 \norm{\nabla _{x}\phi (t)}_{2}+\sum_{k=0,1} \norm{\na^k f_2(t)}_2
 \leq
 C_{l}(1+t)^{-2l+2} \sup_{0\le \tau\le T}\sqrt{\mathcal{ {E}}_{2;l,0}(f(\tau))},
\end{equation}
and
\begin{equation}\label{exponential decay2}
  \norm{\nabla _{x}\phi (t)}_{2}+\sum_{k=0,1} \norm{\na^k f_2(t)}_2
  \le
 C_{l}e^{-C_lt^{2/3}} \sup_{0\le \tau\le T}\sqrt{\mathcal{ {E}}_{2;l,q}(f(\tau))}\text{ for }0<q\ll 1.
 \end{equation}
By the Poisson equation and the continuity equation, we obtain
\begin{equation}
\norm{\nabla_x\phi}_\infty   \lesssim \norm{\nabla_x^2\phi}_2  +\norm{\nabla_x^3\phi}_2
\lesssim \norm{f_2}_2  +\norm{\nabla_x f_2}_2
\end{equation}
and
\begin{equation}
\norm{\partial_t\phi}_\infty
=\norm{\Delta^{-1}\partial_t \rho}_\infty
=\norm{\Delta^{-1}\nabla_x J}_\infty \lesssim \norm{ J}_2+\norm{ \na_x J}_2
\lesssim \norm{f_2}_2^2 +\norm{\nabla_x f_2}_2 .
\end{equation}
We thus conclude the lemma.
\end{proof}

Now we can follow exactly the same strategy of \cite{G12} to prove Theorem \ref{main theorem}.
\begin{proof}[Proof of Theorem \ref{main theorem}]
We first choose the smooth initial data $f_0\in \testF$ with $F_0=\mu+\sqrt{\mu}f_0\ge 0$.

\textit{Step 1. Global Small }$\mathcal{E}_{2;2,0}$\textit{\ Solution. }

We denote
\begin{equation}\label{t*}
T_{\ast }=\sup_{t\geq 0}\left\{  {\mathcal{E}}_{2;2,0}(f(t))+\int_{0}^{t}%
\mathcal{D}_{2;2,0}(f(\tau))d\tau\leq M \text{ and}\int_{0}^{t}\left[\norm{\partial_t\phi(\tau)}_\infty+\norm{\nabla _{x}\phi
(\tau)}_{\infty }\right]d\tau\leq \sqrt{M}\right\} .
\end{equation}%
Clearly $T_{\ast }>0$ if $\mathcal{E}_{2;2,0}(f_{0})$ is sufficiently small
from Theorem \ref{local solution}. Our goal is to show $T_{\ast }=\infty $ if
we further choose $\mathcal{E}_{2;2,0}(f_{0})$ small.

In Proposition \ref{energy estimate}, since $\int_0^{T_\ast}\left[\norm{\partial_t\phi}_\infty+\norm{\nabla _{x}\phi
(\tau)}_{\infty }+{\mathcal{D}}_{2;2,0}(f)\right]d\tau\le 1$, we use the standard Gronwall lemma (Lemma \ref{Gronwall}) to deduce from \eqref{2lq energy estimate} that
\begin{equation}
\mathcal{ {E}}_{2;l,q}(f(t))+\int_{0}^{t}\widetilde{\mathcal{D}}_{2;l,q}(f(\tau))\,d\tau
 \le C_l
\mathcal{ {E}}_{2;l,q}(f_0).
\end{equation}
By Lemma \ref{difference energy}, we can improve the inequality above to be
\begin{equation}\label{2lq}
\mathcal{ {E}}_{2;l,q}(f(t))+\int_{0}^{t} {\mathcal{D}}_{2;l,q}(f(\tau))\,d\tau
 \le C_l
\mathcal{ {E}}_{2;l,q}(f_0).
\end{equation}
Combining Proposition \ref{decay pro} and this bound with $l=2$ and $q=2$, we obtain
\begin{equation}
\int_0^t\left[\norm{\partial _{t}\phi (\tau)}_{\infty }+\norm{\nabla _{x}\phi (\tau)}_{\infty} \right]d\tau
\lesssim  \sup_{0\le \tau\le T_\ast}\sqrt{\mathcal{ {E}}_{2;2,0}(f(\tau))} \int_0^t (1+\tau)^{-2 }d\tau\lesssim \sqrt{\mathcal{ {E}}_{2;2,0}(f_0)}.
\end{equation}
Upon choosing the initial condition $\mathcal{E}_{2;2,0}(f_{0})$ further smaller, we deduce that for $0\leq t\leq T_{\ast }$,
\begin{equation}
\mathcal{ {E}}_{2;2,0}(f(t))+\int_{0}^{t} {\mathcal{D}}_{2;2,0}(f(\tau))\,d\tau
 \le \frac{M}{2}<M
\end{equation}
and
\begin{equation}
 \int_0^t\left[\norm{\partial _{t}\phi (\tau)}_{\infty }+\norm{\nabla _{x}\phi (\tau)}_{\infty} \right]d\tau
 \le \frac{\sqrt{M}}{2}<\sqrt{M}.
\end{equation}
This implies that $T_{\ast }=\infty $ and the solution is global. Also, Proposition \ref{decay pro} and \eqref{2lq} imply that the assertion (1) holds.

\textit{Step 2. Higher Moments and Higher Regularity. }

We shall prove this by an induction of the total derivatives $|\al|+|\be|=m$. By \eqref{2lq}, clearly \eqref{energy inequality m} is valid for $m=2$. Assume \eqref{energy inequality m} holds for $ m-1$.  By Proposition \ref{decay pro}, Lemma \ref{difference energy} and the induction hypothesis, we apply Lemma \ref{Gronwall} to deduce from \eqref{mlq energy estimate} that
\begin{equation}
\begin{split}
&\mathcal{ {E}}_{m;l,q}(f(t))+\int_{0}^{t} {\mathcal{D}}_{m;l,q}(f(\tau))\,d\tau
\\&\quad\le C_{l,m}\mathcal{ {E}}_{m;l,q}(f_0) \left[1+ e^{(1+\int_0^t \mathcal{D}_{m-1;l,q}(f(\tau))\,d\tau)}
\left(1+\int_0^t\mathcal{D}_{m-1;l,q}(f(\tau))\,d\tau\right)\right]
\\&\quad\le C_{l,m}\mathcal{ {E}}_{m;l,q}(f_0) \left[1+ e^{P_{m-1,l}(\mathcal{ {E}}_{m;l,q}(f_{0}) )}
P_{m-1,l}(\mathcal{ {E}}_{m;l,q}(f_{0}) )\right]
\\&\quad \equiv P_{m,l}(\mathcal{ {E}}_{m;l,q}(f_{0}) ),
\end{split}
\end{equation}
where we have used
\begin{equation}
\mathcal{ {E}}_{m-1;l,q}(f_{0})\le \mathcal{ {E}}_{m;l,q}(f_{0})\text{ and } P_{m-1,l}(\mathcal{ {E}}_{m-1;l,q}(f_{0}) )\le P_{m-1,l}(\mathcal{ {E}}_{m;l,q}(f_{0}) ).
\end{equation}
This concludes Theorem \ref{main theorem} for $f_{0}\in \testF$. For a general datum $f_{0}\in \mathcal{E}_{m;l,q}$ we can use a sequence of smooth
approximation $f_{0}^{n\text{ }}$ to construct the approximate solutions and then take a limit in $n$ to conclude the theorem.
\end{proof}

\section{Further decay}\label{decays}

In this section, we will derive some further energy estimates to deduce some further decay rates of the global solution obtained in Theorem \ref{main theorem} to complete the proof of Theorem \ref{further decay}.

Recall $f_1$ and $f_2$ again. The first lemma is a general version of Lemma \ref{difference energy} concerning with the energy estimates for the higher-order spatial derivatives of $f_2=f_+-f_-$.
\begin{lemma}\label{f_2 further en es}
Let $f_{0}\in \testF$ and assume $f$ is the solution constructed in Theorem \ref{local solution}. Then for any $\ell=1,\dots, m$ with $m\ge 2$, there exists a functional $\mathcal{E}_0^\ell(f_2)(t)$ with
 \begin{equation}\label{f_2 0l}
\mathcal{E}_0^\ell(f_2)(t)\sim \sum_{k=0}^\ell  \norm{\nabla^k f_2 }_2^2+ \norm{  \nabla_x\phi}_2^2
\end{equation}
such that
 \begin{equation}\label{f_2 inequality}
 \begin{split}
 &\frac{d}{dt}\mathcal{E}_0^\ell(f_2)+ \lambda\left(\sum_{k=0}^\ell  \norm{\nabla^k f_2 }_\sigma^2+ \norm{ \nabla_x\phi}_2^2\right)
 \\& \quad\le C_m\left(\sqrt{\mathcal{E}_{m;m+3/4,0}(f)}+\mathcal{E}_{m;m,0}(f)\right)\left(\sum_{k=0}^\ell\norm{\nabla^k f_2}_\sigma^2+\norm{ \nabla_x\phi}_2^2 \right).
 \end{split}
\end{equation}
\end{lemma}
\begin{proof}
We recall \eqref{f_2 equation}. The standard $\na^k$ energy estimates on \eqref{f_2 equation} yields for $k=0,\dots,\ell$,
\begin{equation}\label{f_2 es 1}
\begin{split}
& \frac{d}{dt}\left(\frac{\norm{\nabla^k f_2 }_2^2}{2}+2\norm{\nabla^k \nabla_x\phi}_2^2\right) +\lambda\norm{\nabla^k\{{\bf I-P_2}\}f_2}_\sigma^2 \\&\quad \le
\sum_{0\le j\le k} C_k^j\left(  \Gamma_\ast(\nabla^{j}f_1 ,\nabla^{k-j} f_2)+ \nabla^{ j}\nabla_x\phi\cdot \nabla^{k-j}(\nabla_v f_1-v f_1),\nabla^k f_2\right )
:=I_1+I_2.
\end{split}
\end{equation}

We apply the estimate \eqref{ga es 0} of Lemma \ref{nonlinear c} to have, by bounding the $f_1$-related factors via $\sqrt{\mathcal{E}_{m;m,0}(f)}$ since $m\ge 2$,
\begin{equation}\label{f_2 es 2}
I_1\lesssim \sum_{0\le j\le k}C_k^j\int\norms{\nabla^j f_1}_2\norms{\nabla^{k-j} f_2}_\sigma \norms{\nabla^k f_2}_\sigma\,dx
\le C_m \sqrt{\mathcal{E}_{m;m,0}(f)}\sum_{0\le j\le k}\norm{\nabla^j f_2}_\sigma^2.
\end{equation}
While, we have
\begin{equation}\label{f_2 es 3}
\begin{split}
I_2& =\sum_{0\le j\le k}C_k^j\left|\left(\nabla^{k-j} f_1 \nabla^{j} \nabla_x\phi,  \nabla^{k}(\nabla_vf_2+v  f_2)\right)\right|
\\&\le C_m \sum_{0\le j\le k} \int \norms{\langle v\rangle^{3/2} \nabla^{k-j}f_1}_2\norms{\nabla^{ j}\nabla_x\phi}\norms{\nabla^{k}f_2}_\sigma\,dx
\\& \le C_m\sqrt{\mathcal{E}_{m;m+3/4,0}(f)} \sum_{0\le j\le k}\norm{\nabla^j\nabla_x\phi}_{2}\norm{\nabla^{k}f_2}_\sigma
\\& \le C_m \sqrt{\mathcal{E}_{m;m+3/4,0}(f)} \left( \norm{ \nabla_x\phi}_2^2+\norm{\nabla^{k}f_2}_\sigma^2\right)
\end{split}
\end{equation}
Note that we require the extra $3/4$ so that when $k=\ell=m$ and $j=0$ we can bound by $\norm{\langle v\rangle^{3/2} \nabla^{m}f_1}_2\le \sqrt{\mathcal{E}_{m;m+3/4,0}(f)}$; otherwise, we need to restrict that $\ell\le m-1$. In particular, when $m=2$ we can only take $\ell=1$ that we did in Lemma \ref{f201}.

So, summing over $k=0,\dots,\ell$, we deduce from \eqref{f_2 es 1} that
\begin{equation}\label{33333}
\begin{split}
&\frac{d}{dt}\sum_{k=0}^\ell\left(\frac{\norm{\nabla^k f_2 }_2^2}{2}+2\norm{\nabla^k \nabla_x\phi}_2^2\right) +\lambda\sum_{k=0}^\ell\norm{\nabla^k\{{\bf I-P_2}\}f_2}_\sigma^2
\\&\quad\le C_m\sqrt{\mathcal{E}_{m;m+3/4,0}(f)}\left(\sum_{k=0}^\ell\norm{\nabla^k f_2}_\sigma^2+\norm{ \nabla_x\phi}_2^2 \right).
\end{split}
\end{equation}

On the other hand, as in Lemma \ref{difference energy}, for $k=0,\dots,\ell-1$ there exists a function
$G_{f_2}^k(t)$ with
\begin{equation} \label{22222}
G_{f_2}^k(t)\lesssim \norm{\na^k f_2}_2^2+\norm{\na^{k+1} f_2}_2^2+\norm{\na^k \na_x\phi}_2^2
\end{equation}
such that
\begin{equation} \label{22222'}
\begin{split}
& \frac{d}{dt}G_{f_2}^k + \norm{ \nabla^k{\bf P_2} f_2}_{2}^2 +\norm{\na^{k+1} {\bf P_2} f_2}_{2}^2+\norm{ \nabla^k\nabla_x\phi}_{2}^2
\\&\quad\lesssim \norm{  \nabla^k\{{\bf I- P_2}\} f_2}_{\sigma}^2+\norm{\na^{k+1} \{{\bf I- P_2}\} f_2}_{\sigma}^2+\norm{\nabla^k N_{\parallel}}_{2}^2.
\end{split}
\end{equation}
Here the last term is bounded by
\begin{equation}
\begin{split}
\norm{ \nabla^k N_{ \parallel}}_{2}^2&\equiv\norm{\langle \nabla^k\left(\Gamma_\ast(f_1,f_2)+\nabla_x\phi\cdot \left(\nabla_vf_1-vf_1\right)\right), v{\mu}\rangle}_2^2
\\&\le C_m\mathcal{E}_{m;m,0}(f)\left(\sum_{0\le j\le k}\norm{\nabla^j f_2}_\sigma^2 +\norm{\nabla_x\phi}_2^2\right).
\end{split}
\end{equation}
Hence, summing over $k=0,\dots,\ell-1$, we deduce from \eqref{22222'} that
\begin{equation}\label{44444}
\begin{split}
&\frac{d}{dt}\sum_{k=0}^{\ell-1}G_{f_2}^k(t)+\sum_{k=0}^{\ell} \norm{\nabla^k {\bf P_2} f_2}_2^2 +\norm{ \nabla_x\phi}_{2}^2
\\&\quad \lesssim  \sum_{k=0}^{\ell}\norm{ \nabla^k \{{\bf I- P_2}\} f_2}_{\sigma}^2+C_m \mathcal{E}_{m;m,0}(f)\left(\sum_{k=0}^{\ell-1}\norm{\nabla^k f_2}_\sigma^2 +\norm{\nabla_x\phi}_2^2\right).
 \end{split}
 \end{equation}

 Now, for $\ell=1,\dots,m$ we define
 \begin{equation}
\mathcal{E}_0^\ell(f_2)(t):=\sum_{k=0}^\ell\left(\norm{\nabla^k f_2 }_2^2+2\norm{\nabla^k \nabla_x\phi}_2^2\right)+\varepsilon \sum_{k=0}^{\ell-1}G_{f_2}^k.
\end{equation}
By \eqref{22222}, for $\varepsilon$ sufficiently small, we deduce that $\mathcal{E}_0^\ell({f_2})(t)$ satisfies \eqref{f_2 0l} and that \eqref{f_2 inequality} follows from \eqref{33333} and \eqref{44444}.
 \end{proof}

Next, we will derive an energy estimates which allows us to derive the optimal decay rate of $f_1$ and its higher-order spatial derivatives. We recall the evolution for $f_1$ separating from \eqref{f_1f_2 equation}:
\begin{equation}\label{f1 eq 1}
\partial_tf_1+ v\cdot\nabla_xf_1+ \mathcal{L}_1f_1=\Gamma_\ast(f_1,f_1)+\nabla_x\phi\cdot \left(\nabla_vf_2-vf_2\right).
\end{equation}
\begin{lemma}\label{f_1 further en es}
Let $f_{0}\in \testF$ and assume $f$ is the solution constructed in Theorem \ref{local solution}.
Then for any $  \ell= 0,\dots, m-1$ with $m\ge 2$, there exists a function $\mathcal{E}^m_\ell({f_1})(t)$ with
 \begin{equation}\label{E_f_1 es}
\mathcal{E}^m_\ell({f_1})(t)\sim  \sum_{ k=\ell}^m\norm{\nabla^k f_1 }_2^2
\end{equation}
such that
 \begin{equation}\label{f_1 inequality}
 \begin{split}
 &\frac{d}{dt}\mathcal{E}^m_\ell({f_1})+\lambda \left(\sum_{k=\ell+1}^m  \norm{\nabla^\ell f_1 }_\sigma^2+ \norm{\nabla^\ell \{{\bf I-P_1}\}f_1 }_\sigma^2\right)
 \\& \quad\le C_m \left(\sqrt{{\mathcal{E}_{m;m,0}(f)}}+{\mathcal{E}_{m;m+3/4,0}(f)}\right)\left(\sum_{k=\ell+1}^m \norm{\na^k f_1}_\sigma^2+\sum_{k=0}^m \norm{\nabla^k f_2}_\sigma^2+\norm{ \nabla_x\phi}_2^2\right).
 \end{split}
\end{equation}
\end{lemma}
\begin{proof}
The standard $\na^k$ energy estimates on \eqref{f1 eq 1} yields for $k=\ell,\dots,m$,
\begin{equation}\label{f_1 es 1}
\begin{split}
& \frac{d}{dt} \frac{\norm{\nabla^k f_1 }_2^2}{2} +\lambda\norm{\nabla^k\{{\bf I-P_1}\}f_1}_\sigma^2 \\&\quad=
\sum_{0\le j\le k} C_k^j\left(  \Gamma_\ast(\nabla^{j}f_1 ,\nabla^{k-j} f_1)+ \nabla^{j}\nabla_x\phi\cdot \nabla^{k-j}(\nabla_v f_2-vf_2),\nabla^k f_1\right ):=I_1+I_2
.
\end{split}
\end{equation}

By the collision invariant property and the estimate \eqref{ga es 0} of Lemma \ref{nonlinear c}, we obtain
\begin{equation}\label{term1}
\begin{split}
\left|\left(  \Gamma_\ast(\nabla^{j}f_1 ,\nabla^{k-j} f_1) ,\nabla^k f_1\right )\right|&=\left|
 \left(\Gamma_\ast(\nabla^{j}f_1 ,\nabla^{k-j} f_1) ,\nabla^k\{{\bf I-P_1}\} f_1\right )\right|
\\&\lesssim \norm{|\mu^\delta \nabla^j f_1|_2|\nabla^{k-j}f_1|_\sigma}_2 \norm{\nabla^k \{{\bf I-P_1}\}f_1}_\sigma.
\end{split}
\end{equation}
Our goal is to finely estimate the right hand side of \eqref{term1} so that
it can be bounded by the right hand side of \eqref{f_1 inequality}. The crucial point is to use the Sobolev interpolation of Lemma
\ref{interpolation lemma}. Indeed, by H\"older's inequality, Minkowski's integral inequality \eqref{minkowski inequality} of Lemma \ref{Minkowski} and  Lemma
\ref{interpolation lemma}, we obtain that if $k=\ell,\dots,m-1$,
\begin{equation}\label{0000}
\begin{split}
 \norm{|\mu^\delta\nabla^j f_1|_2| \nabla^{k-j}f_1|_\sigma}_2
 & \lesssim \norm{\mu^\delta \nabla^j f_1}_{L^3_xL^2_v}\norm{
\nabla^{k-j}f_1}_{L^6_xL^2_\sigma}\lesssim \norm{\nabla^j f_1}_{L^2_v L^3_x}\norm{\nabla^{k-j}f_1}_{L^2_{\sigma}L^6_x}
\\& \lesssim \norm{\mu^\delta \nabla^\zeta f_1}_2^{1-\frac{j}{k+1}}\norm{\mu^\delta\nabla^{k+1}f_1}_2^{\frac{j}{k+1}}
\norm{f_1}_\sigma^{\frac{j}{k+1}}\norm{\nabla^{k+1}f_1}^{1-\frac{j}{k+1}}_\sigma
\\& \lesssim \sqrt{\mathcal{E}_{m;m,0}(f)}\norm{\nabla^{k+1}f_1}_\sigma,
\end{split}
\end{equation}
where $\zeta$ comes from the adjustment of the power index over $\norm{\nabla^{k+1}f_1}_\sigma$ and is defined by
\begin{equation}
j+3\left(\frac{1}{2}-\frac{1}{3}\right) = \zeta \times \left(1-\frac{j}{k+1}\right)
+(k+1)\times \frac{j}{k+1}
\Longrightarrow\zeta=\frac{k+1}{2(k+1-j)}.
\end{equation}
Note that we have used $\norm{ f_1}_\sigma\le \sqrt{\mathcal{E}_{m;m,0}(f)}$ since $m\ge 2$. Now let $k=m$. If $j=0$ we take $L^\infty-L^2$ to have
\begin{equation}
 \norm{\norms{\mu^\delta  f_1}_2\norms{\nabla^{m }f_1}_\sigma}_2\lesssim  \sqrt{\mathcal{E}_{2;2,0}(f)}\norm{\nabla^m f_1}_\sigma.
 \end{equation}
If $j=m$ we take $L^\infty-L^2$  to have
\begin{equation}
 \norm{|\mu^\delta\nabla^m  f_1|_2|  f_1|_\sigma}_2
 \lesssim  \norm{\mu^\delta\nabla^m f_1}_2 \sum_{j=1,2}\norm{\nabla^j f_1}_\sigma;
\end{equation}
if $m=2$, then we further bound the above by
\begin{equation}
\norm{\mu^\delta\nabla^2 f_1}_2 \norm{\nabla  f_1}_\sigma+\norm{\mu^\delta\nabla^2 f_1}_2 \norm{\nabla^2  f_1}_\sigma
\lesssim \sqrt{\mathcal{E}_{2;2,0}(f)}\norm{ \nabla^2 f_1}_\sigma.
\end{equation}
If $m\ge 3$, then we bound it instead by
\begin{equation}
\norm{\mu^\delta\nabla^3 f_1}_2 \sum_{j=1,2}\norm{\nabla^j f_1}_\sigma\lesssim \sqrt{\mathcal{E}_{3;3,0}(f)}\norm{\nabla^3 f_1}_\sigma.
\end{equation}
Here we have used $\norm{ \na^j f_1}_\sigma\le \sqrt{\mathcal{E}_{m;m,0}(f)}$ for $1\le j\le m-1$.
Now if $1\le j\le m-1$, similarly as in \eqref{0000}, we obtain
\begin{equation}
\begin{split}
\norm{|\mu^\delta\nabla^j f_1|_2| \nabla^{m-j}f_1|_\sigma}_2
 &\lesssim \norm{\mu^\delta \nabla^j f_1}_{L^2_v L^3_x}\norm{\nabla^{m-j}f_1}_{L^2_{\sigma}L^6_x}
\\& \lesssim \norm{\mu^\delta \nabla^\zeta f_1}_2^{1-\frac{j-1}{m}}\norm{\mu^\delta\nabla^{m}f_1}_2^{\frac{j-1}{m}}
\norm{f_1}_\sigma^{\frac{j-1}{m}}\norm{\nabla^mf_1}^{1-\frac{j-1}{m}}_\sigma
\\& \lesssim \sqrt{\mathcal{E}_{m;m,0}(f)}\norm{\nabla^m f_1}_\sigma,
\end{split}
\end{equation}
where $\zeta$  is defined by
\begin{equation}
j+3\left(\frac{1}{2}-\frac{1}{3}\right) = \zeta \times \left(1-\frac{j-1}{m}\right)
+m\times \frac{j-1}{m}\Longrightarrow\zeta=\frac{3m}{2(m+1-j)}.
\end{equation}
Hence, we conclude that for $k=\ell,\dots,m-1$,
\begin{equation} \label{f_1 es 2}
I_1\le \eta\norm{\nabla^k \{{\bf I-P_1}\}f_1}_\sigma^2+ C_{m,\eta} {\mathcal{E}_{m;m,0}(f)}\norm{\nabla^{k+1}f_1}_\sigma^2
\end{equation}
and for $k=m$,
\begin{equation}\label{f_1 es 3}
I_1\le \eta\norm{\nabla^m \{{\bf I-P_1}\}f_1}_\sigma^2+ C_{m,\eta}  {\mathcal{E}_{m;m,0}(f)}\norm{\nabla^m f_1}_\sigma^2.
\end{equation}

For the term $I_2$, we integrate by parts in $v$ and use the split $f_1={\bf P_1} f_1+\{{\bf I-P_1}\}f_1$ to have
\begin{equation}\label{f_1 es 4}
\begin{split}
I_2& =-\sum_{0\le j\le k}C_k^j \left(\na^{k-j} f_2  \na^{j}\nabla _{x}\phi, \na^k\left(\nabla _{v}f_1+vf_1\right) \right)
\\&\le  \sum_{0\le j\le k}C_k^j \int \left\{\norms{\mu^\delta \nabla^{k-j}f_2}_2\norms{\nabla^{ j}\nabla_x\phi}\norms{\mu^\delta\nabla^{k}  f_1}_2\right.
\\&\qquad\qquad\qquad\ \left.+\norms{\langle v\rangle^{3/2} \nabla^{k-j}f_2}_2\norms{\nabla^{ j}\nabla_x\phi}\norms{\nabla^{k}\{{\bf I-P_1}\} f_1}_\sigma\right\}dx
\\&\le C_m \left(\sqrt{{\mathcal{E}_{m;m,0}(f)}}+C_\eta{\mathcal{E}_{m;m+3/4,0}(f)}\right)\sum_{0\le j\le k}\left(\norm{\nabla^j\nabla_x\phi}_{2}^2+\norm{\nabla^j f_2}_{\sigma}^2\right)
\\&\quad+\eta\norm{\nabla^{k}\{{\bf I-P_1}\} f_1}_\sigma^2.
\end{split}
\end{equation}

So, summing over $k=\ell,\dots,m$, by \eqref{f_1 es 2}--\eqref{f_1 es 4}, we deduce from \eqref{f_1 es 1} that
\begin{equation}\label{77777}
\begin{split}
& \frac{d}{dt}\sum_{k=\ell}^m\frac{\|\nabla^k f_1 \|_2^2}{2}  + \lambda \sum_{k=\ell}^m\norm{\nabla^k\{{\bf I-P_1}\}f_1}_\sigma^2
\\& \quad\le C_m \left(\sqrt{{\mathcal{E}_{m;m,0}(f)}}+{\mathcal{E}_{m;m+3/4,0}(f)}\right)\left(\sum_{k=\ell+1}^m \norm{\na^k f_1}_\sigma^2+\sum_{k=0}^m \norm{\nabla^k f_2}_\sigma^2+\norm{ \nabla_x\phi}_2^2\right).
\end{split}
\end{equation}

On the other hand, for $k=\ell,\dots,m-1$, as before there exists a function
$G_{f_1}^k(t)$ with
\begin{equation} \label{66666}
G_{f_1}^k(t)\lesssim \norm{\na^k f_1}_2^2+\norm{\na^{k+1} f_1}_2^2
\end{equation}
such that
\begin{equation}
\frac{d}{dt}G_{f_1}^k+\norm{ \na^{k+1}  {\bf P_1} f_1}_{2}^2
\lesssim \norm{\na^k \{{\bf I- P_1}\} f_1}_{\sigma}^2+\norm{ \na^{k+1} \{{\bf I- P_1}\} f_1}_{\sigma}^2+\norm{\na^k N_{\parallel} }_{2}^2.
\end{equation}
Here the last term is bounded by
 \begin{equation}
\begin{split}
\norm{\na^k N_{\parallel} }_{2}^2&
\le  \sum_{0\le j\le k}C_k^j\norm{\langle \Gamma_\ast(\na^j f_1,\na^{k- j}f_1)+\na^ j\nabla_x\phi\cdot  \na^{k- j}(\nabla_v f_2-vf_2), \tilde{\mu}\rangle}_{2}^2
\\&\le  \sum_{0\le j\le k}C_k^j\norm{\norms{ \mu^\delta\na^ j f_1}_2\norms{\na^{k- j}f_1}_\sigma}_2^2
+\norm{  \norms{\na^j \nabla_x\phi}\norms{\mu^\delta\na^{k-j}f_2}_2}_{2}^2
\\&\le C_m  {\mathcal{E}_{m;m,0}(f)}\left(\norm{\nabla^{k+1}f_1}_\sigma^2+ \sum_{0\le j\le k}\left(\norm{\nabla^j\nabla_x\phi}_{2}^2+\norm{\nabla^j f_2}_{\sigma}^2\right)\right),
\end{split}
\end{equation}
where we have used \eqref{0000}. Hence, summing over $k=\ell,\dots,m-1$, we deduce
\begin{equation}\label{88888}
\begin{split}
&\frac{d}{dt}\sum_{k=\ell}^{m-1}G_{f_1}^k(t)+\sum_{k=\ell+1}^{m} \norm{\nabla^k {\bf P_1} f_1}_2^2
\\& \quad\lesssim \sum_{k=\ell}^{m}\norm{  \na^k \{{\bf I- P_1}\} f_1}_{\sigma}^2+C_m{\mathcal{E}_{m;m,0}(f)}\left(\sum_{k=\ell+1}^{m}\norm{\nabla^{k}f_1}_\sigma^2+\sum_{k=0}^{m-1} \norm{\nabla^k f_2}_\sigma^2+\norm{ \nabla_x\phi}_2^2\right).
 \end{split}
 \end{equation}

 Now, for $\ell=0,\dots,m-1$, we define
 \begin{equation}
 \mathcal{E}^m_\ell({f_1})(t):= \sum_{k=\ell}^m \frac{\norm{\nabla^k f_1 }_2^2}{2}+\varepsilon \sum_{k=\ell}^{m-1}G_{f_1}^k.
\end{equation}
By \eqref{66666}, for $\varepsilon$ sufficiently small, we deduce that $\mathcal{E}_\ell^m({f_1})(t)$ satisfies \eqref{E_f_1 es} and that \eqref{f_1 inequality} follows from \eqref{77777} and \eqref{88888}.
 \end{proof}

The following lemma provides the needed estimates for proving the faster decay rates of the microscopic part $\{{\bf I-P_1}\}f_1$. We shall use the macro-micro decomposition:
\begin{equation}\label{I-P equation f_1}
\begin{split}
& \left\{\partial _{t} +v\cdot \nabla _{x} \right\}\{{\bf I-P_1}\}f_{1}
+\mathcal{L}_1f_1 \\
&\quad=  \Gamma _\ast(f_1,f_1)+\{{\bf I-P_1}\}\left(\nabla_x\phi\cdot \left(\nabla_vf_2-vf_2\right)\right)+{\bf  P_1}(v\cdot \na_xf_{1})-v\cdot \na_x {\bf  P_1}f_{1}.
\end{split}
\end{equation}
\begin{lemma}\label{lemma micro 44}
Let $f_{0}\in \testF$ and assume $f$ is the solution constructed in Theorem \ref{local solution}.
Then for any $k=0,\dots, m-2$ with $m\ge 2$, we have
\begin{equation}\label{micro estimate 2}
\begin{split}
& \frac{d}{dt}  {\norm{\nabla^k \{{\bf I-P_1}\} f_1 }_2^2}  +\lambda\norm{\nabla^k\{{\bf I-P_1}\}f_1}_\sigma^2
\\&\quad\le C_m (1+\mathcal{E}_{m;m,0}(f))\left( \norm{\na^{k+1} f_1}_2^2+\sum_{j=0}^k \norm{\nabla^j f_2}_2^2+\norm{ \nabla_x\phi}_2^2\right)
\\&\qquad+C\sqrt{\mathcal{E}_{m;m,0}(f)}\norm{\na^{k}\{{\bf I-P_1}\} f_1}_\sigma^2.
\end{split}
\end{equation}
\end{lemma}
\begin{proof}
The standard $\na^k$ energy estimates on \eqref{I-P equation f_1} yields for $k=0,\dots,m-2$,
\begin{equation}
\begin{split}
& \frac{d}{dt} \frac{\norm{\nabla^k \{{\bf I-P_1}\} f_1 }_2^2}{2} +\lambda\norm{\nabla^k\{{\bf I-P_1}\}f_1}_\sigma^2 \\&\quad\le
\sum_{0\le j\le k} C_k^j\left|\left(  \Gamma_\ast(\nabla^{j}f_1 ,\nabla^{k-j} f_1) ,\nabla^k\{{\bf I-P_1}\} f_1\right )\right|
\\&\qquad+\sum_{0\le j\le k} C_k^j \left|\left(   \nabla^{j}\nabla_x\phi\cdot \nabla^{k-j}(\nabla_v f_2-vf_2),\nabla^k\{{\bf I-P_1}\} f_1\right )\right|
\\&\qquad+  \left|\left(   \nabla^{k}\left(v\cdot \na_x {\bf  P_1}f_{1}-{\bf  P_1}(v\cdot \na_xf_{1})\right),\nabla^k\{{\bf I-P_1}\} f_1\right )\right|
:= I_1+I_2+I_3.
\end{split}
\end{equation}

By the estimate \eqref{ga es 0} of Lemma \ref{nonlinear c}, we obtain
\begin{equation}
\begin{split}
&\left|\left(  \Gamma_\ast(\nabla^{j}f_1 ,\nabla^{k-j} f_1) ,\nabla^k\{{\bf I-P_1}\} f_1\right )\right|\lesssim \norm{\norms{\mu^\delta \nabla^j f_1}_2\norms{\nabla^{k-j}f_1}_\sigma}_2 \norm{\nabla^k \{{\bf I-P_1}\}f_1}_\sigma
\\&\quad\lesssim \norm{\norms{\mu^\delta \nabla^j f_1}_2\left(\norms{\nabla^{k-j} {\bf P_1}f_1}_2+\norms{\nabla^{k-j} \{{\bf I- P_1}\}{f_1}}_\sigma\right)}_2 \norm{\nabla^k \{{\bf I-P_1}\}f_1}_\sigma
\\&\quad:=I_{11}+I_{12}.
\end{split}
\end{equation}
As in \eqref{0000}, we have
\begin{equation}
I_{11}\le \eta\norm{\nabla^k \{{\bf I-P_1}\}f_1}_\sigma+C_{m,\eta} {\mathcal{E}_{m;m,0}(f)}\norm{\nabla^{k+1}f_1}_2^2.
\end{equation}
Similarly, we can obtain
\begin{equation}
I_{12}\lesssim \sqrt{\mathcal{E}_{m;m,0}(f)}\norm{\nabla^k \{{\bf I-P_1}\}f_1}_\sigma.
\end{equation}

As in \eqref{f_1 es 4}, we can obtain, since $k\le m-2$,
\begin{equation}
I_2\le \eta\norm{\nabla^k \{{\bf I-P_1}\}f_1}_\sigma+C_{m,\eta} \sum_{j=0}^k  \norm{ \nabla^j\nabla_x\phi}_2^2.
\end{equation}
Clearly, we have
\begin{equation}
I_3\lesssim \eta\norm{\na^{k}\{{\bf I-P_1}\} f_1}_\sigma^2+C_\eta \norm{\na^{k+1} f_1}_2^2.
\end{equation}
We then conclude the lemma by taking $\eta$ small.
\end{proof}

We now derive the evolution of the negative Sobolev norms of $f_1$. In order to estimate the nonlinear terms, we need to
restrict ourselves to that $s\in (0,3/2)$.
\begin{lemma}\label{lemma H-s}
Let $f_{0}\in \testF$ and assume $f$ is the solution constructed in Theorem \ref{local solution}. For $s\in (0, 1/2]$, we have
\begin{equation}\label{H-s1}
\frac{d}{dt}\norm{\Lambda^{-s}f_1}_2^2+\lambda \norm{\Lambda^{-s}\{{\bf I-P_1}\}f_1}_\sigma^2\lesssim \left(\norm{\Lambda^{-s}f_1}_2+\mathcal{E}_{2;2,0}\right)\mathcal{D}_{2;2,0};
\end{equation}
and for $s\in (1/2, 3/2)$, we have
\begin{equation}\label{H-s2}
 \frac{d}{dt}\norm{\Lambda^{-s}f_1}_2^2+\lambda\norm{\Lambda^{-s}\{{\bf I-P}_1\}f_1}_\sigma^2
 \lesssim \left(\norm{\Lambda^{-s}f_1}_2+\mathcal{E}_{2;2,0}\right)\mathcal{D}_{2;2,0}+\norm{ f_1
}_2^{2s+1}\norm{ \nabla f_1}_2^{3-2s}.
\end{equation}
\end{lemma}
\begin{proof}
The standard $\Lambda^{-s}$ energy estimates on \eqref{I-P equation f_1} yields for $0<s<3/2$,
\begin{equation}\label{Hs es 00}
\begin{split}
& \frac{d}{dt}\frac{\norm{\Lambda^{-s}f_1}_2^2}{2}+ \lambda\norm{\Lambda^{-s}\{{\bf I-P_1}\} f_1}_\nu^2
\\&\quad \le \left(\Lambda^{-s}\Gamma_\ast(f_1,f_1),\Lambda^{-s}f_1\right)+\left(\Lambda^{-s}\left(\nabla_x\phi\cdot \left(\nabla_vf_2-vf_2\right)\right), \Lambda^{-s} f_1\right):=I_1+I_2.
\end{split}
\end{equation}

By the collision invariant property and the representation of $\Gamma_\ast$ as Lemma 1 in \cite{G02}, we obtain
\begin{equation}\label{Hs es 4}
\begin{split}
&\left(\Lambda^{-s}\Gamma_\ast(f_1,f_1),\Lambda^{-s}f_1\right) =\left(\Lambda^{-s}\Gamma_\ast(f_1,f_1), \Lambda^{-s}\{{\bf I-P_1}\}f_1\right)
\\&\quad= \left(\Lambda^{-s}\left(\partial _i \left( \left\{\Phi ^{ij}*[\mu^{1/2} f_1]\right\}\partial _jf_1 \right)
-
\left\{\Phi^{ij}*[v_i\mu^{1/2} f_1]\right\} \partial _j f_1\right.\right.
\\
&\qquad\qquad\ \left.\left.\ -
\partial _i \left(\left\{\Phi ^{ij}*[\mu^{1/2}\partial _j f_1]\right\} f_1 \right)
+
\left\{\Phi ^{ij}*[v_i\mu^{1/2}\partial _j f_1]\right\}  f_1\right), \Lambda^{-s}\{{\bf I-P_1}\}f_1\right)
\\&\quad= -\left( \Lambda^{-s}\left(\left\{\Phi ^{ij}*[\mu^{1/2} f_1]\right\}\partial _jf_1
  -
 \left\{\Phi ^{ij}*[\mu^{1/2}\partial _j f_1]\right\} f_1
, \partial _i  \Lambda^{-s}\{{\bf I-P_1}\}f_1\right)\right.
\\&\qquad\left.-\left(\Lambda^{-s}\left(
\left\{\Phi^{ij}*[v_i\mu^{1/2} f_1]\right\} \partial _j f_1
-
\left\{\Phi ^{ij}*[v_i\mu^{1/2}\partial _j f_1]\right\}  f_1\right)\right), \Lambda^{-s}\{{\bf I-P_1}\}f_1\right)
\\&\quad:=I_{11}+I_{12}+I_{13}+I_{14}.
\end{split}
\end{equation}

We first estimate the term $I_{11}$. Since $0<s<3/2$, we let $1<p<2$ to be with ${1}/{2}+{s}/3={1}/{p}$. By the estimate
\eqref{Riesz estimate} for $\Lambda^{-s}$ in Lemma \ref{Riesz lemma}, Minkowski's integral inequality \eqref{minkowski inequality} in Lemma \ref{Minkowski}, and Lemma 2 in \cite{G02}, we have
\begin{equation}\label{023}
\begin{split}
I_{11}&= -\left( \Lambda^{-s}\left(\langle v\rangle ^{3/2}\left\{\Phi ^{ij}*[\mu^{1/2} f_1]\right\}\partial _jf_1
  \right)
, \langle v\rangle ^{-3/2}\partial _i  \Lambda^{-s}\{{\bf I-P_1}\}f_1\right)
\\&\lesssim \norm{\Lambda^{-s}\left(\langle v\rangle ^{3/2}\left\{\Phi ^{ij}*[\mu^{1/2} f_1]\right\}\partial _jf_1
  \right)}_2\norm{\Lambda^{-s}\{{\bf I-P_1}\}f_1}_\sigma
  \\&\lesssim C_\eta \norm{  \langle v\rangle ^{3/2}\left\{\Phi ^{ij}*[\mu^{1/2} f_1]\right\}\partial _jf_1
   }_{L^p_xL^2_v}+\eta\norm{\Lambda^{-s}\{{\bf I-P_1}\}f_1}_\sigma^2
     \\&\lesssim C_\eta \norm{\norms{\mu^{\delta} f_1}_2 ^2\norms{\langle v\rangle ^{1/2}\nabla_vf_1}_2^2
   }_{L^p_x}+\eta\norm{\Lambda^{-s}\{{\bf I-P_1}\}f_1}_\sigma^2.
\end{split}
\end{equation}
We shall now use the split $f_1={\bf P_1}f_1+\{{\bf I-P_1}\}f_1$: for the microscopic part, we take $L^{3/s}-L^2$ in the $L^p_x$ norm above and we further bound it by
\begin{equation}
\norm{ f_1}_{L^2_v L^{ \frac{3}{s}}_x}^2 \norm{\langle v\rangle ^{1/2}\nabla_v\{\FI-\FP_1\}f_1}_2^2
\lesssim \norm{ f_1}_{L^2_v H^2_x}^2 \norm{\{\FI-\FP_1\}f_1}_{\sigma, \langle v\rangle^2}^2
\lesssim  {\mathcal{E}_{2;2,0}}{\mathcal{D}_{2;2,0}};
\end{equation}
while for the hydrodynamic part, we shall separate the estimates according to the value of $s$. If $0< s\le 1/2$, then $3/s\ge 6$, we use
Sobolev's inequality to have
\begin{equation}\label{Hs es 7}
\norm{ \mu^\delta f_1}_{L^2_v L^{ \frac{3}{s}}_x}^2 \norm{\langle v\rangle ^{1/2}\nabla_v \FP_1 f_1}_2^2\lesssim\norm{f_1}_{L_v^2L_x^\frac{3}{s}}^2\norm{f_1}_2^2
\lesssim  \norm{f_1}_2^2  \sum_{1\le |\gamma|\le 2}\norm{\pa^\ga f_1}_{\sigma}^2\lesssim  {\mathcal{E}_{2;2,0}}{\mathcal{D}_{2;2,0}};
\end{equation}
and if $s\in(1/2,3/2)$, then $2<3/s<6$, we use the (different) Sobolev interpolation and H\"older's inequality to have
\begin{equation}\label{Hs es 8}
\norm{f_1}_{L_v^2L_x^\frac{3}{s}}\norm{f_1}_2\lesssim \norm{ f_1 } ^{s-1/2}_2\norm{\nabla_x f_1 }_2^{3/2-s} \norm{ f_1 }_2=\norm{ f_1 } _2 ^{s+1/2}\norm{\nabla_x f_1 }_2^{3/2-s}.
\end{equation}

For the second term $I_{12}$, as in \eqref{023} we have
\begin{equation}
I_2\lesssim C_\eta \norm{\norms{\mu^{\delta}\nabla_v f_1}_2  \norms{\langle v\rangle ^{1/2}f_1}_2
   }_{L^p_x}^2+\eta\norm{\Lambda^{-s}\{{\bf I-P_1}\}f_1}_\sigma^2
 .
\end{equation}
We again use the split: for the hydrodynamic part, we can estimate it as for $I_1$; while for the microscopic part, we take $L^{2p}-L^{2p}$ in $L^p_x$ norm to get
\begin{equation}
\begin{split}
 \norm{\norms{\mu^{\delta}\nabla_v f_1}_2   \norms{\langle v\rangle ^{1/2}\{\FI-\FP_1\}f_1}_2
   }_{L^p_x}^2
& \lesssim  \norm{ \mu^{\delta}\nabla_v f_1}_{L^{2p}_xL^2_v}^2   \norm{\langle v\rangle ^{1/2}\{\FI-\FP_1\}f_1
   }_{L^{2p}_xL^2_v}^2
    \\& \lesssim  \sum_{|\gamma|\le 1}\norm{ \mu^{\delta}\pa^\ga\nabla_v f_1}_{2}^2   \norm{\langle v\rangle ^{1/2}\pa^\ga\{\FI-\FP_1\}f_1
   }_2^2
   \\& \lesssim  {\mathcal{E}_{2;2,0}}\norm{ \pa^\ga\{\FI-\FP_1\}f_1
   }_{\sigma,\langle v\rangle }^2   \lesssim  {\mathcal{E}_{2;2,0}} {\mathcal{D}_{2;2,0}}.
   \end{split}
\end{equation}
Note that $I_{13}$ and $I_{14}$ can be estimated in the similar way and have the same upper bounds.

Now for the term $I_2$, we integrate by parts in $v$ and recall  \eqref{sigma norm =} to have
\begin{equation}\label{Hs es 0}
\begin{split}
I_2&=-\left(\Lambda^{-s}\left(\nabla_x\phi f_2 \right), \Lambda^{-s}(v f+\nabla_v  f_1)\right)
\\& \lesssim \norm{\Lambda^{-s}\left(\nabla_x\phi \langle v\rangle ^{3/2}f_2\right)}_2\norm{ \Lambda^{-s} f_1}_\sigma
 \lesssim \norm{ \norms{\nabla_x\phi} \norms{\langle v\rangle ^{3/2}f_2}_2}_{L^p_x}\norm{ \Lambda^{-s} f_1}_\sigma
\\& \lesssim \norm{  \nabla_x\phi}_{L_x^{\frac{3}{s}}} \norm{f_2}_{\sigma,\langle v\rangle ^2}\left(\norm{ \Lambda^{-s} {\FP_1}f_1}_\sigma+\norm{ \Lambda^{-s} \{\FI-\FP\}f_1}_\sigma\right)
\\& \lesssim  \mathcal{D}_{2;2,0}(f)  \norm{ \Lambda^{-s}  f_1}_2 +C_\eta\mathcal{E}_{2;2,0}(f)\mathcal{D}_{2;2,0}(f)+ \eta\norm{ \Lambda^{-s} \{\FI-\FP\}f_1}_\sigma^2.
\end{split}
\end{equation}

Consequently, in light of \eqref{Hs es 4}--\eqref{Hs es 0}, by taking $\eta$ sufficiently small, we deduce that \eqref{H-s1} holds for $s\in(0,1/2]$ and that \eqref{H-s2} holds for $s\in(1/2,3/2)$.
\end{proof}

Now we are ready to prove Theorem \ref{further decay}.
\begin{proof}[Proof of Theorem \ref{further decay}]
We only need to prove the theorem for the smooth initial data, and for a general datum  we can use a smooth
approximation.

\textit{Step 1. Decay of }$f_2=f_+-f_-$.

Suppose that $\mathcal{E}_{m;l,0}(f_0)\le M=M(m,l)$  is small for $l-\frac{3}{4}\ge m\ge 2$, then Theorem \ref{main theorem} implies that there exists a unique global $\mathcal{E}_{m;l,0}$ solution $f$, and $\mathcal{E}_{m;l,0}(f(t))\lesssim P_{m,l}(\mathcal{E}_{m;l,0}(f_0)) \le M$ is small for all time $t$. Since $l\ge m+\frac{3}{4}$, by Lemma \ref{f_2 further en es}, we have that for $\ell=1,\dots,m$,
 \begin{equation}\label{end 1}
  \frac{d}{dt}\mathcal{E}_0^\ell(f_2)+ \lambda\left(\sum_{k=0}^\ell  \norm{\nabla^k f_2 }_\sigma^2+ \norm{ \nabla_x\phi}_2^2\right)
\le 0.
\end{equation}

We now establish the polynomial decay by applying the interpolation method (among velocity moments) developed in \cite{SG06}. By H\"older's inequality, for $k=0,\dots,\ell$ with $1\le \ell \le m$ we have
 \begin{equation}\label{10101010}
 \begin{split}
\norm{\nabla^k f_2 }_2
&=\norm{\langle v\rangle^{-1/2}\nabla^k f_2 }_{2}^{\frac{4(l-\ell)}{4(l-\ell)+1}}\norm{\langle v\rangle^{4(l-\ell)}\nabla^k f_2 }_{2}^{\frac{1}{4(l-\ell)+1}}
\\&\le \norm{\nabla^k f_2 }_{\sigma}^{\frac{4(l-\ell)}{4(l-\ell)+1}}\left\{\sqrt{\mathcal{E}_{m;l,0}(f)}\right\}^{\frac{1}{4(l-\ell)+1}}
.\end{split}
\end{equation}
Hence, we deduce from \eqref{end 1} that
 \begin{equation}
  \frac{d}{dt}\mathcal{E}_0^\ell(f_2)+ \la\left\{ {\sup_{0\le \tau\le T}\mathcal{E}_{m;l,0}(f(\tau))}\right\}^{-\frac{1}{4(l-\ell)}} \left\{\mathcal{E}_0^\ell(f_2)\right\}^{1+\frac{1}{4(l-\ell)}}
\le 0.
\end{equation}
Solving this inequality directly, we obtain
 \begin{equation}\label{aaa}
  \begin{split}
\mathcal{E}_0^\ell(f_2)(t)&\le\left(\left\{\mathcal{E}_0^\ell(f_2)(0)\right\}^{-\frac{1}{4(l-\ell)}}+ \frac{\la}{4(l-\ell)}\left\{ {\sup_{0\le \tau\le T}\mathcal{E}_{m;l,0}(f(\tau))}\right\}^{-\frac{1}{4(l-\ell)}}  t  \right)^{-4(l-\ell)}
\\&\le C_{l,\ell}\left(1+ t  \right)^{-4(l-\ell)}{\sup_{0\le \tau\le T}\mathcal{E}_{m;l,0}(f(\tau))}.
\end{split}
\end{equation}
This implies the polynomial decay \eqref{polynomial decay k} by taking the square root of the above.

We now prove the stretched exponential decay by applying the splitting method (velocity-time) in \cite{SG08}. Assume in addition that $\mathcal{E}_{m;l,q}(f_0)<+\infty$ with $0<q\ll 1$, Theorem \ref{main theorem} implies that the unique global solution $f$ is of $\mathcal{E}_{m;l,q}$. For any $\theta>0$, $\varepsilon>0$ and $k=0,\dots,m$,
 \begin{equation}\label{909090}
 \begin{split}
 \norm{\nabla^k f_2 }_{\sigma}^2&\gtrsim\int \langle v\rangle^{-1}|\nabla^k f_2|^2=\int_{|v|\le \theta (1+t)^\varepsilon}+\int_{|v|\ge \theta (1+t)^\varepsilon}
 \\&\ge \frac{(1+t)^{-\varepsilon}}{\theta}\int_{|v|\le \theta (1+t)^\varepsilon}  |\nabla^k f_2|^2
 \\&= \frac{(1+t)^{-\varepsilon}}{\theta}\norm{\nabla^k f_2}_2^2-\frac{(1+t)^{-\varepsilon}}{\theta}\int_{|v|\ge \theta (1+t)^\varepsilon}  |\nabla^k f_2|^2,
\end{split}
\end{equation}
and we bound by
 \begin{equation}
\int_{|v|\ge \theta (1+t)^\varepsilon}  |\nabla^k f_2|^2
\le \int e^{q|v|^2}e^{-q\theta^2(1+t)^{2\varepsilon}}|\nabla^k f_2|^2
\le e^{-q\theta^2(1+t)^{2\varepsilon}}\mathcal{E}_{m;l,q}(f) .
\end{equation}
Hence, we deduce from \eqref{end 1} that, taking $\theta\ge 1$,
 \begin{equation}
  \frac{d}{dt}\mathcal{E}_0^m(f_2)+  \frac{\la(1+t)^{-\varepsilon}}{\theta}\mathcal{E}_0^m(f_2)
\le \frac{\la(1+t)^{-\varepsilon}}{\theta}e^{-q\theta^2(1+t)^{2\varepsilon}}\mathcal{E}_{m;l,q}(f) .
\end{equation}
We therefore have, taking $\varepsilon\neq 1$,
\begin{equation}\label{40404040}
  \begin{split}
\frac{d}{dt}\left\{e^{\frac{\la(1+t)^{1-\varepsilon}}{\theta(1-\varepsilon)}}\mathcal{E}_0^m(f_2)(t)\right\}
&\le e^{\frac{\la(1+t)^{1-\varepsilon}}{\theta(1-\varepsilon)}}\frac{\la(1+t)^{-\varepsilon}}{\theta}e^{-q\theta^2(1+t)^{2\varepsilon}}\mathcal{E}_{m;l,q}(f) \\& \le  \sup_{0\le \tau \le T}\mathcal{E}_{m;l,q}(f(\tau)) \frac{\la(1+t)^{-\varepsilon}}{\theta}e^{\frac{\la(1+t)^{1-\varepsilon}}{\theta(1-\varepsilon)}-q\theta^2(1+t)^{2\varepsilon}}.
\end{split}
\end{equation}
Integrating this inequality from $0$ to $t$, we obtain
\begin{equation}\label{bbb}
  \begin{split}
 \mathcal{E}_0^m(f_2)(t)
&\le   e^{-\frac{\la(1+t)^{1-\varepsilon}}{\theta(1-\varepsilon)}}
\left(e^{\frac{\la}{\theta(1-\varepsilon)}}\mathcal{E}_0^m(f_2)(0)+\sup_{0\le \tau\le T}\mathcal{E}_{m;l,q}(f(\tau))\right.
\\ & \left.\qquad\qquad\qquad\qquad\qquad  \times\int_0^t\frac{\la(1+\tau)^{-\varepsilon}}{\theta}e^{\frac{\la(1+\tau)^{1-\varepsilon}}{\theta(1-\varepsilon)}-q\theta^2(1+\tau)^{2\varepsilon}}d\tau \right).
\\ & \le C_{l} e^{-\frac{\la(1+t)^{1-\varepsilon}}{\theta(1-\varepsilon)}}
\sup_{0\le \tau\le T}\mathcal{E}_{m;l,q}(f(\tau)),
\end{split}
\end{equation}
if we have taken $1-\varepsilon\le 2\varepsilon$ and $\theta$ sufficiently large. This implies the stretched exponential decay \eqref{exponential decay k} by taking $\varepsilon=\frac{1}{3}$.

\textit{Step 2. Decay of }$f_1=f_++f_-$.

First, since $l\ge m+3/4$, by Lemma \ref{f_1 further en es}, we have that for $\ell=0,\dots,m-1$,
 \begin{equation}
\frac{d}{dt}\mathcal{E}^m_\ell({f_1})+\la\left(\sum_{k=\ell+1}^m  \norm{\nabla^k f_1 }_\sigma^2+ \norm{\nabla^\ell \{{\bf I-P_1}\}f_1 }_\sigma^2\right)
\lesssim  \sqrt{M} \left(\sum_{k=0}^m \norm{\nabla^k f_2}_\sigma^2+\norm{ \nabla_x\phi}_2^2\right).
\end{equation}
Combining this and \eqref{end 1} with $\ell=m$, we obtain
\begin{equation}\label{end 2}
 \begin{split}
& \frac{d}{dt}\left\{\mathcal{E}_\ell^m ({f_1})+\mathcal{E}_0^m(f_2)\right\}
+\lambda  \left(\sum_{k=\ell+1}^m  \norm{\nabla^k f_1 }_\sigma^2+ \norm{\nabla^\ell \{{\bf I-P_1}\}f_1 }_\sigma^2 \right.\\&\qquad\qquad\qquad\qquad\qquad\qquad\ \left.+\sum_{k=0}^m  \norm{\nabla^k f_2 }_\sigma^2+ \norm{\nabla_x\phi}_2^2\right)\le 0.
 \end{split}
\end{equation}

From \eqref{10101010} again, we have for some $\lambda_0$ depending on the initial data and $l,m,\ell$,
\begin{equation} \label{30303030}
 \begin{split}
&\sum_{k=\ell+1}^m  \norm{\nabla^k f_1 }_\sigma^2+ \norm{\nabla^\ell \{{\bf I-P_1}\}f_1 }_\sigma^2+\sum_{k=0}^m  \norm{\nabla^k f_2 }_\sigma^2+ \norm{\nabla_x\phi}_2^2
\\&\quad\ge\lambda_0\left\{\sum_{k=\ell+1}^m  \norm{\nabla^k f_1 }_2^2+ \norm{\nabla^\ell \{{\bf I-P_1}\}f_1 }_2^2+\mathcal{E}_0^m(f_2)\right\}^{1+\frac{1}{4(l-m)}}.
 \end{split}
\end{equation}
Hence, it suffices to bound the only one remaining term $\norm{\na^\ell {\bf P_1}f_1}_2^2$ in the energy in terms of the dissipation in \eqref{end 2}. We now use an interpolation method (among spatial regularity) together with the splitting method (velocity-time) in \cite{S10}. Assuming for the moment that we have proved \eqref{H-sbound}, then the key point is to do the Sobolev interpolation between the negative and positive Sobolev norms by using Lemma
\ref{interpolation lemma} that for $s\ge 0$ and $\ell+s>0$,
\begin{equation}\label{inter 1'}
\norm{\nabla^\ell  {\bf P_1}f_1}_2\le  \norm{\Lambda^{-s} f_1}_2^{\frac{1}{\ell+1+s}}\norm{\nabla^{\ell+1} {\bf P_1}f_1 }_2^{\frac{\ell+s}{\ell+1+s}}
\le  C_0\norm{\nabla^{\ell+1}  f_1 }_\sigma^{\frac{\ell+s}{\ell+1+s}}.
\end{equation}
Combining this and \eqref{30303030}, we deduce from \eqref{end 2} that
\begin{equation}\label{end 3}
 \frac{d}{dt}\left\{\mathcal{E}_\ell^m ({f_1})+\mathcal{E}_0^m(f_2)\right\}
+\lambda_0 \left\{\mathcal{E}_\ell^m ({f_1})+\mathcal{E}_0^m(f_2)\right\}^{1+\vartheta}\le 0,
\end{equation}
where $\vartheta=\max\left\{\frac{1}{\ell+s},\frac{1}{4(l-m)}\right\}$.
Solving this inequality directly, we obtain in particular
\begin{equation}
\mathcal{E}_\ell^m ({f_1})(t) \le \left({ {\mathcal{E}}_{m;l,0}(f_0)}^{-\vartheta}+\vartheta\lambda_0 t\right)^{-  {1}/{\vartheta}}\le C_0 (1+ t)^{- {1}/{\vartheta}}.
\end{equation}
So, if we want to have the optimal decay rate of $f_1$ for all derivatives of order up to $m-1$, we only need to take $l$ so large that
\begin{equation}
m-1+s\le 4(l-m)\Longleftrightarrow l\ge \frac{5}{4}m+\frac{s-1}{4}.
\end{equation}
This proves the optimal decay \eqref{decay f1} by taking $l\ge \max\{m+\frac{3}{4},\frac{5}{4}m+\frac{s-1}{4}\}$.

Next, by Lemma \ref{lemma micro 44}, we have that for any $k=0,\dots, m-2$,
\begin{equation}\label{micro estimate 2''}
\begin{split}
 \frac{d}{dt}  {\norm{\nabla^k \{{\bf I-P_1}\} f_1 }_2^2}  +\lambda \norm{\nabla^k\{{\bf I-P_1}\}f_1}_\sigma^2
 \le C_m\left( \norm{\na^{k+1} f_1}_2^2+\sum_{j=0}^k \norm{\nabla^j f_2}_2^2+\norm{ \nabla_x\phi}_2^2\right).
\end{split}
\end{equation}
For any $\varepsilon>0$ and $k=0,\dots,m-2$, as in \eqref{909090}, we have
 \begin{equation}
 \norm{\nabla^k\{{\bf I-P_1}\}f_1}_\sigma^2\ge  {(1+t)^{-\varepsilon}} \norm{\nabla^k\{{\bf I-P_1}\}f_1}_2^2-{(1+t)^{-\varepsilon}}\int_{|v|\ge (1+t)^{\varepsilon}}  |\nabla^k f_1|^2,
\end{equation}
and we bound by
 \begin{equation}
 \begin{split}
{(1+t)^{-\varepsilon}}\int_{|v|\ge (1+t)^{\varepsilon}}  |\nabla^k f_1|^2
&\le  {(1+t)^{-\varepsilon}}{(1+t)^{-4(l-k)\varepsilon }}  \int | v |^{4(l-k)} |\nabla^k f_1|^2
\\&\le   {(1+t)^{-(4l-4k+1)\varepsilon }}  \mathcal{E}_{m;l,0}(f) .
\end{split}
\end{equation}
Hence, we deduce from \eqref{micro estimate 2''} that
 \begin{equation}\label{99999}
  \begin{split}
  &\frac{d}{dt}\norm{\nabla^k \{{\bf I-P_1}\} f_1 }_2^2+ \lambda  (1+t)^{-\varepsilon}\norm{\nabla^k \{{\bf I-P_1}\} f_1 }_2^2
\\&\quad\le {(1+t)^{-(4l-4k+1)\varepsilon }}  \mathcal{E}_{m;l,0}(f) +C_m\left(\norm{\na^{k+1} f_1}_2^2+\sum_{j=0}^k \norm{\nabla^j f_2}_2^2+\norm{ \nabla_x\phi}_2^2 \right).
\end{split}
\end{equation}
Denoting
\begin{equation}
\vartheta=\min\left\{(4l-4k+1)\varepsilon,k+1+s,4l-4k,4l-4\right\},
\end{equation}
by \eqref{polynomial decay k}, \eqref{decay f1} and Lemma \ref{basic decay}, we thus deduce from \eqref{99999} that for $\varepsilon\neq 1$,
\begin{equation}
\norm{\nabla^k \{{\bf I-P_1}\} f_1 (t)}_2^2
\le  C_0 e^{-\frac{\lambda (1+t)^{1-\varepsilon}}{ 1-\varepsilon}}\left(1
+\int_0^t e^{\frac{ \lambda (1+\tau)^{1-\varepsilon} }{1-\varepsilon}}{(1+\tau)^{-\vartheta}} d\tau\right)
\le C_0(1+\tau)^{-\vartheta+\varepsilon}.
\end{equation}
So, if we want to have the almost optimal decay rate of $\{{\bf I-P_1}\} f_1$ with any small $\varepsilon$ in \eqref{decay f1 micro} for all derivatives of order up to $m-2$, we only need to take $l$ so large that
\begin{equation}
(4l-4(m-2)+1)\varepsilon,\ 4l-4(m-2)\ge m-1+s \Longleftrightarrow l\ge \frac{m-1+s}{4\varepsilon} +m-\frac{9}{4},
\end{equation}
since $\varepsilon$ is small. This proves the almost optimal decay \eqref{decay f1 micro}.

Finally, we turn back to prove \eqref{H-sbound}. First, for $s\in(0,1/2]$, integrating in time the estimate \eqref{H-s1} of Lemma \ref{lemma H-s}, by the bound \eqref{energy inequality}, we obtain
\begin{equation}
\begin{split}
\norm{\Lambda^{-s}f_1(t)}_2^2 &\le \norm{\Lambda^{-s}f_1(0)}_2^2 +C\int_0^t \left(\norm{\Lambda^{-s}f_1(\tau)}_2+\mathcal{E}_{2;2,0}(\tau)\right)\mathcal{D}_{2;2,0}(\tau)\,d\tau
\\&\le C_0\left(1+\sup_{0\le \tau \le t}\norm{\Lambda^{-s}f_1(\tau)}_2\right).
\end{split}
\end{equation}
By Cauchy's inequality, this together with \eqref{energy inequality} gives \eqref{H-sbound} for $s\in[0,1/2]$, and thus verifies \eqref{decay f1} for $s\in[0,1/2]$. Now we let $s\in(1/2,3/2)$. Observing that now $f_0\in L^2_v\dot{H}_x^{-1/2}$
since $ \dot{H}_x^{-s}\cap  L^2_x\subset  \dot{H}_x^{-s'}$ for any $s'\in [0,s]$, we then deduce from what we have proved for
\eqref{decay f1} with $s=1/2$ that the following decay result holds:
\begin{equation}\label{decay11}
\sum_{\ell\le k\le m}\norm{\nabla^k f_1(t)}_2^2\le C_0(1+t)^{-(\ell+ \frac{1}{2})}\, \hbox{ for } \ell=0, \dots,m-1.
\end{equation}
Hence, by \eqref{decay11} and \eqref{energy inequality}, we deduce from \eqref{H-s2} that for $s\in (1/2,3/2)$,
\begin{equation}
\begin{split}
\norm{\Lambda^{-s}f_1(t)}_2^2 &\le \norm{\Lambda^{-s}f_1(0)}_2^2 +C\int_0^t \left(\norm{\Lambda^{-s}f_1(\tau)}_2+\mathcal{E}_{2;2,0}(\tau)\right)\mathcal{D}_{2;2,0}(\tau)\,d\tau
\\&\quad+C \int_0^t  \norm{ f_1(\tau)
}_2^{2s+1}\norm{ \nabla f_1(\tau)  }_2^{3-2s} \,d\tau
\\&\le C_0+C_0\sup_{0\le \tau \le t}\norm{\Lambda^{-s}f_1(\tau)}_2+C_0\int_0^t(1+\tau)^{-(5/2-s)}\,d\tau
\\&\le C_0\left(1+\sup_{0\le \tau \le t}\norm{\Lambda^{-s}f_1(\tau)}_2\right).
\end{split}
\end{equation}
This proves \eqref{H-sbound} for $s\in (1/2,3/2)$  and thus verifies \eqref{decay f1} for $s\in(1/2,3/2)$. Hence, the proof of Theorem \ref{further decay} is completed.
\end{proof}

\end{document}